\pdfoutput=1
\documentclass[a4paper]{article}

\usepackage[all]{xy}\usepackage[latin1]{inputenc}  
\usepackage[dvips]{graphics,graphicx}
\usepackage{amsfonts,amssymb,amsmath,color,mathrsfs, amstext,bbm}
\usepackage{amsbsy, amsopn, amscd, amsxtra, amsthm,authblk, enumerate}
\usepackage{upref}
\usepackage{geometry}
\usepackage[displaymath]{lineno}
\usepackage{bm}
\usepackage{dsfont}
\usepackage{todonotes}
\usepackage{mathtools}
\usepackage{float}

\usepackage[colorlinks,
linkcolor=blue,
anchorcolor=green,
citecolor=blue
]{hyperref}

\numberwithin{equation}{section}

\def\e{\varepsilon}
\def\epsilon{\varepsilon}

\DeclareMathOperator{\Perm}{Perm}

\DeclareMathOperator{\pat}{pat}
\DeclareMathOperator{\occ}{occ}
\DeclareMathOperator{\Leb}{Leb}
\DeclareMathOperator{\Id}{Id}

\newcommand{\bbE}{\mathbb{E}}
\newcommand{\bbH}{\mathbb{H}}

\newcommand{\bbR}{\mathbb{R}}
\newcommand{\bbP}{\mathbb{P}}

\newcommand{\pocc}{\widetilde{\occ}}

\newcommand{\eqb}{\begin{equation}}
	\newcommand{\eqe}{\end{equation}}
\newcommand{\eqbn}{\begin{equation*}}
	\newcommand{\eqen}{\end{equation*}}

\newcommand{\BB}{\mathbbm}
\newcommand{\ol}{\overline}

\newcommand{\op}{\operatorname}

\newcommand{\ep}{\epsilon}

\newcommand{\wt}{\widetilde}
\newcommand{\wh}{\widehat} 
\newcommand{\mcl}{\mathcal}

\newtheorem{theorem}{Theorem}[section]
\newtheorem{lemma}[theorem]{Lemma}
\newtheorem{proposition}[theorem]{Proposition}
\newtheorem{corollary}[theorem]{Corollary}
\newtheorem{definition}[theorem]{Definition}

\newtheorem{example}[theorem]{\bf Example}
\theoremstyle{remark}
\newtheorem{remark}[theorem]{\bf Remark}

\numberwithin{equation}{section}

\newcommand{\C}{\mathbbm{C}}

\newcommand{\N}{\mathbbm{N}}

\newcommand{\Z}{\mathbbm{Z}}
\newcommand{\R}{\mathbbm{R}}
\renewcommand{\P}{\mathbbm{P}}

\begin{document}
	
	\title{Baxter permuton and Liouville quantum gravity}
	
	\author{
		Jacopo Borga\thanks{Stanford University, Department of Mathematics, \href{mailto:jborga@stanford.edu}{jborga@stanford.edu}.} $\qquad$ Nina Holden\thanks{University of Pennsylvania, Department of Mathematics, \href{mailto:xinsun@sas.upenn.edu}{xinsun@sas.upenn.edu}.} $\qquad$ Xin Sun\thanks{ETH Zurich, Department of Mathematics, \href{mailto:nina.holden@eth-its.ethz.ch}{nina.holden@eth-its.ethz.ch}.} $\qquad$ Pu Yu\thanks{Massachusetts Institute of Technology,
			Department of Mathematics, \href{mailto:puyu1516@mit.edu}{puyu1516@mit.edu}.}
	} 
	\date{  }
	
	\maketitle
	
	\begin{abstract}	
		The Baxter permuton is a random probability measure on the unit square which describes the scaling limit of uniform Baxter permutations. We determine an explicit formula for the density of the expectation of the Baxter permuton. This answers a question of Dokos and Pak (2014). 
		We also prove that all pattern densities of the Baxter permuton are strictly positive, distinguishing it from other permutons arising as scaling limits of pattern-avoiding permutations. Our proofs rely on a recent connection between the Baxter permuton and  Liouville quantum gravity (LQG) coupled with the Schramm-Loewner evolution (SLE).  The method works equally well for a two-parameter generalization of the Baxter permuton recently introduced by the first author, except that the density  is not as explicit. This new family of permutons, called \emph{skew Brownian permuton},   describes the scaling limit of a number of random constrained permutations. We finally observe that in the LQG/SLE framework, the expected proportion of  inversions in a skew Brownian permuton equals $\frac{\pi-2\theta}{2\pi}$ where $\theta$ is  the so-called imaginary geometry angle between a certain pair of SLE curves.
	\end{abstract}

	\section{Introduction}
	
	Baxter permutations were introduced by Glen Baxter in 1964 \cite{Bax64} while studying fixed points of commuting functions. 
	They are classical  examples of pattern-avoiding  permutations,  which have been intensively studied both in the probabilistic and combinatorial literature (see e.g.\ \cite{MR0250516,MR491652,MR555815,MR2028288,MR2679559,MR2763051,MR3882946}). They are known to be connected with various other interesting combinatorial structures, such as bipolar orientations \cite{BBMF11}, walks in cones \cite{KMSW19}, certain pairs of binary trees and a family of triples of non-intersecting lattice paths\cite{MR2763051}, and domino tilings of Aztec diamonds \cite{MR2679559}.  
	
	In recent years there has been an increasing interest in studying limits of random pattern-avoiding permutations.
	One approach is to look at the convergence of relevant statistics, such as  the number of cycles, the number of inversions,  or the length of the longest increasing subsequence. For a brief overview of this approach see e.g.\  \cite[Section 1.4]{borga2021random}. 
	The more recent  approach is to directly determine the scaling limits of permutation diagrams. Here given a permutation $\sigma$ of size $n$, its \emph{diagram} is a $n \times n$ table with $n$ points at position $(i,\sigma(i))$ for all $i\in[n]:=\{1,2,\dots,n\}$. (See Figure~\ref{fig:pattern_perm_exemp}, p.\ \pageref{fig:pattern_perm_exemp}, for an example.) Their scaling limits are called \emph{permutons}. See
	e.g.\  \cite[Section 2.1]{borga2021random} for an overview of this approach; and Section~\ref{sect:posit} and Appendix~\ref{sect:intro_patt} for an introduction to permutation pattern terminology.
	
	Dokos and Pak \cite{MR3238333}  studied the expected limiting permuton of the so-called \emph{doubly alternating Baxter permutations}. The authors raised the question of proving the existence of  the \emph{Baxter permuton} as the scaling limit of uniform Baxter permutations,  and determine its expected density. The existence of the Baxter permuton was established in \cite{BM20}  based on the  bijection between Baxter permutations and bipolar orientations. In  \cite{borga2021skewperm}, a two-parameter family of   permutons called the \emph{skew Brownian permuton} was introduced.  This family includes  the Baxter permuton and a well-studied one-parameter family of permutons, called the \emph{biased Brownian separable permuton} (\cite{bassino2018separable,bassino2017universal}), as special cases. 
	
	By~\cite{KMSW19,GHS16},  the scaling limit of random planar maps decorated with bipolar orientations is described by Liouville quantum gravity (LQG) decorated with two Schramm-Loewner evolution (SLE) curves. In \cite{borga2021skewperm}, the author built a direct connection between  the skew Brownian permuton (including the Baxter permuton)  and SLE/LQG (see also \cite{bgs22meanders} for further developments).   
	The main goal of the present paper is to use this connection to derive some properties of these permutons. In particular, we find an explicit formula for  the density of the intensity measure of the Baxter permuton (see Section \ref{sect:int_bax_perm} for definitions), which answers the aforementioned question of Dokos and Pak.  We also prove that all (standard) pattern densities of the Baxter permuton are strictly positive. The second result extends to the skew Brownian permuton except in one case where it is not true, namely for the biased Brownian separable permuton. 
	
	In the rest of the introduction,  we first state our  main results on the Baxter permuton in Section~\ref{sect:res_on_Bax_perm}. Then, in  Section~\ref{sect:con_perm_SBP}, we recall the construction of the skew Brownian permuton and state the corresponding  results. Finally, in Section~\ref{sect:rel_with_LQG} we review the connection with LQG/SLE and explain our proof techniques.
	
	\subsection{Main results on the Baxter permuton}\label{sect:res_on_Bax_perm}
	
	A Baxter permutation is a permutation which satisfies the following pattern avoidance property.
	
	\begin{definition}\label{def:baxter}
		A permutation $\sigma$ is a \emph{Baxter permutation} if it is not possible to find $i < j < k$ such that $\sigma(j+1) < \sigma(i) < \sigma(k) < \sigma(j)$ or $\sigma(j) < \sigma(k) < \sigma(i) < \sigma(j+1)$.
	\end{definition}
	
	Note that there are finitely many Baxter permutations of size $n$. Therefore it makes sense to consider a uniform Baxter permutation of size $n$.
	
	A Borel probability measure $\mu$ on the unit square $[0,1]^2$ is a \emph{permuton} if both of its marginals are uniform, i.e., $\mu([a,b]\times[0,1]) = \mu([0,1]\times[a,b])=b-a$ for any $0\le a<b\le 1$.
	A permutation $\sigma$ can be viewed as a permuton $\mu_{\sigma}$ by uniformly distributing mass to the squares $\{[\frac{i-1}{n}, \frac{i}{n}]\times  [\frac{\sigma(i)-1}{n}, \frac{\sigma(i)}{n}]: i \in [n]\}.$ More precisely,
	\begin{equation*}
		\mu_\sigma(A)
		= 
		n 
		\sum_{i=1}^n 
		\Leb
		\big([(i-1)/n, i/n]
		\times
		[(\sigma(i)-1)/n,\sigma(i)/n]
		\cap 
		A\big),
	\end{equation*}
	where $A$ is a Borel measurable set of $[0,1]^2$.
	
	For a \emph{deterministic} sequence of permutations $\sigma_n$, we say that $\sigma_n$ \emph{converge in the permuton sense} to a limiting permuton $\mu$, if the permutons $\mu_{\sigma_n}$ induced by $\sigma_n$ converge weakly to $\mu$.
	The set of permutons equipped with the topology of weak convergence of measures can be viewed as a compact metric space.  
	
	\begin{theorem}[{\cite[Theorem 1.9]{BM20}}]\label{thm:baxt_conv}
		Let $\sigma_n$ be a uniform Baxter permutation of size $n$. The following convergence w.r.t.\ the permuton topology holds:
		$
		\mu_{\sigma_n}\xrightarrow{d} 	\mu_B,
		$
		where $\mu_B$ is a random permuton called the \emph{Baxter permuton}.
	\end{theorem}
	
	We present our main results on the Baxter permuton in the next section.
	
	\subsubsection{The intensity measure of the Baxter permuton}\label{sect:int_bax_perm}
	
	The Baxter permuton $\mu_B$ is a random probability measure on the unit square (with uniform marginals). 
	{Our first result is an explicit expression of its intensity measure, defined by $\bbE [\mu_{B}](\cdot)\coloneqq\bbE [\mu_{B}(\cdot)]$, which answers   \cite[Question 6.7]{MR3238333}.}
	
	\begin{theorem}\label{thm-baxter-density}
		Consider the Baxter permuton $\mu_B$. Define the function 
		\begin{equation}\label{eqn-rho}
			\rho(t, x, r):= \frac{1}{t^2} \left(\left(\frac{3rx}{2t}-1\right)e^{-\frac{r^2+x^2-rx}{2t}}+e^{-\frac{(x+r)^2}{2t}}\right).
		\end{equation}
		Then the intensity measure $\bbE [\mu_B]$ is absolutely continuous with respect to the Lebesgue measure on $[0,1]^2$. {Moreover}, it has the following density function 
		\begin{equation}\label{eqn-baxter-density}
			p_B(x, y) = c\int_{\max\{0, x+y-1\}}^{\min\{x, y\}}\int_{\bbR_{+}^4}\rho(y-z, \ell_1, \ell_2)\rho(z, \ell_2, \ell_3)\rho(x-z, \ell_3, \ell_4)\rho(1+z-x-y, \ell_4, \ell_1)\,\ d\ell_1 d\ell_2 d\ell_3 d\ell_4\, dz, 
		\end{equation}
		where $c$ is a normalizing constant.
	\end{theorem}
	
	\begin{remark}
		As discussed in Section~\ref{sect:rel_tetra}, further computation of the integral \eqref{eqn-baxter-density} is tricky, as it involves integrating a four-dimensional Gaussian in the first quadrant. Nevertheless, this integral in $\bbR_4^+$ can be expressed as the volume function (and its derivatives) of a three-dimensional spherical tetrahedron as given in \cite{Mur12, AM14}.
	\end{remark}
	
	We highlight that the intensity measure of other universal random limiting permutons has been investigated in the literature. For instance, the intensity measure of the \emph{biased Brownian separable permuton}, was determined by Maazoun in \cite{MR4079636}. We recall that the biased Brownian separable permuton $\mu^q_S$, defined for all $q\in[0,1]$, is a one-parameter universal family of limiting permutons arising form pattern-avoiding permutations (see Section~\ref{sect:con_perm_SBP} for more explanations). In \cite[Theorem 1.7]{MR4079636}, it was proved that for all $q\in(0,1)$, the intensity measure $\bbE[\mu^q_S]$ of the biased Brownian separable permuton is absolutely continuous with respect to the Lebesgue measure on $[0,1]^2$. Furthermore, $\bbE[\mu^q_S]$  has the following density function 
	
	\begin{equation*}
		p^q_S(x, y) =\int_{\max\{0,x+y -1\}}^{\min\{x,y\}} \frac{3q^2(1-q)^2 \,da }
		{2\pi(a(x-a)(1-x-y+a)(y-a))^{3/2}{\left(\frac{q^2}{a}+\frac{(1-q)^2}{(x-a)}+\frac{q^2}{(1-x-y+a)}+\frac{(1-q)^2}{(y-a)}\right)^{5/2}}}.
	\end{equation*}
	The proof of \cite[Theorem 1.7]{MR4079636} relies on an explicit construction of the biased Brownian separable permuton $\mu^q_S$ from a one-dimensional Brownian excursion decorated with i.i.d.\ plus and minus signs. To the best of our knowledge this proof cannot be easily extended to the Baxter permuton case. 
	See the figures below for  some plots of $p_B(x, y)$ and $p^q_S(x, y)$ using numerical approximations of the integrals.
	
	\begin{figure}[ht]
		\centering
		\includegraphics[width=0.23\textwidth]{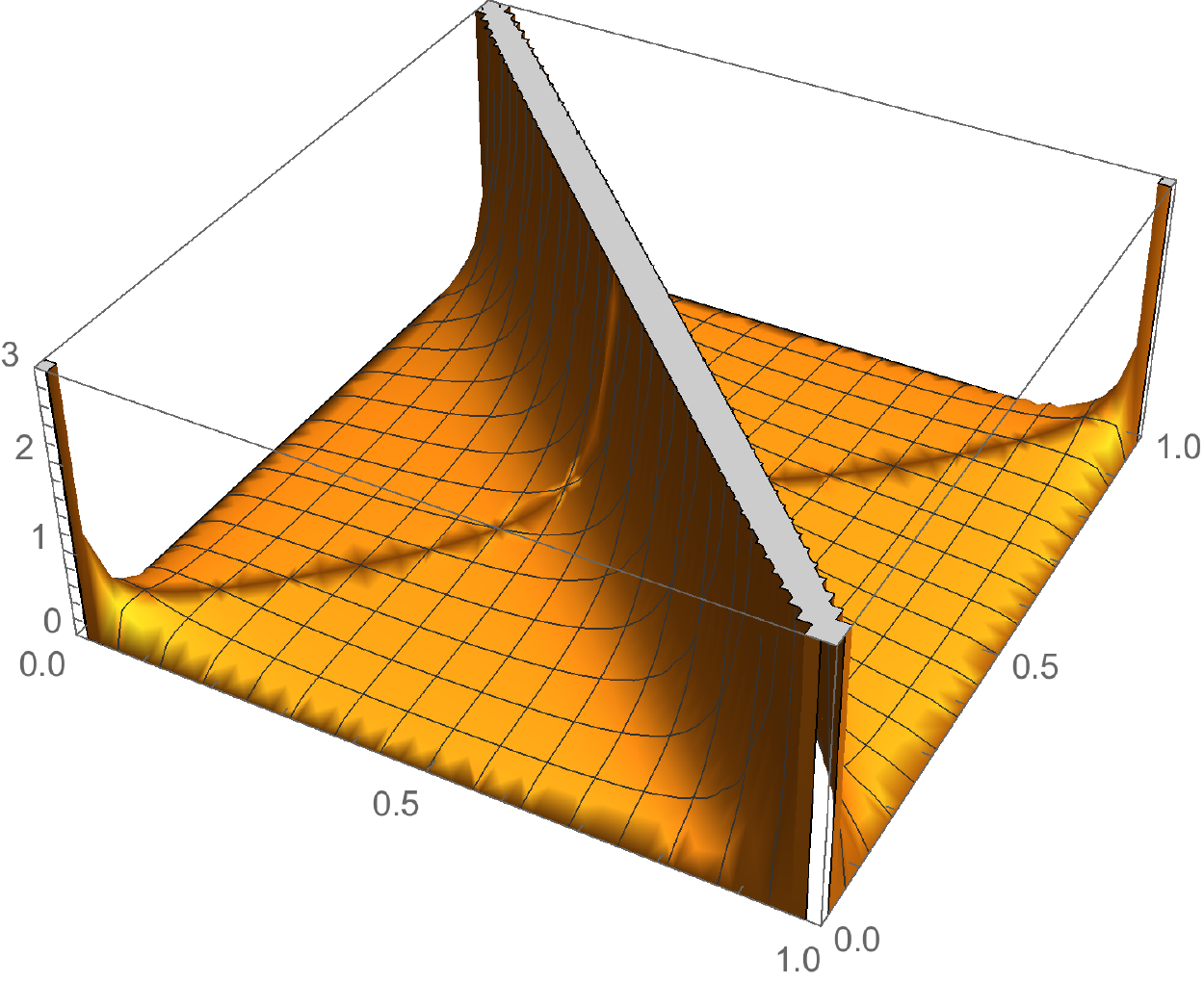}
		\includegraphics[width=0.23\textwidth]{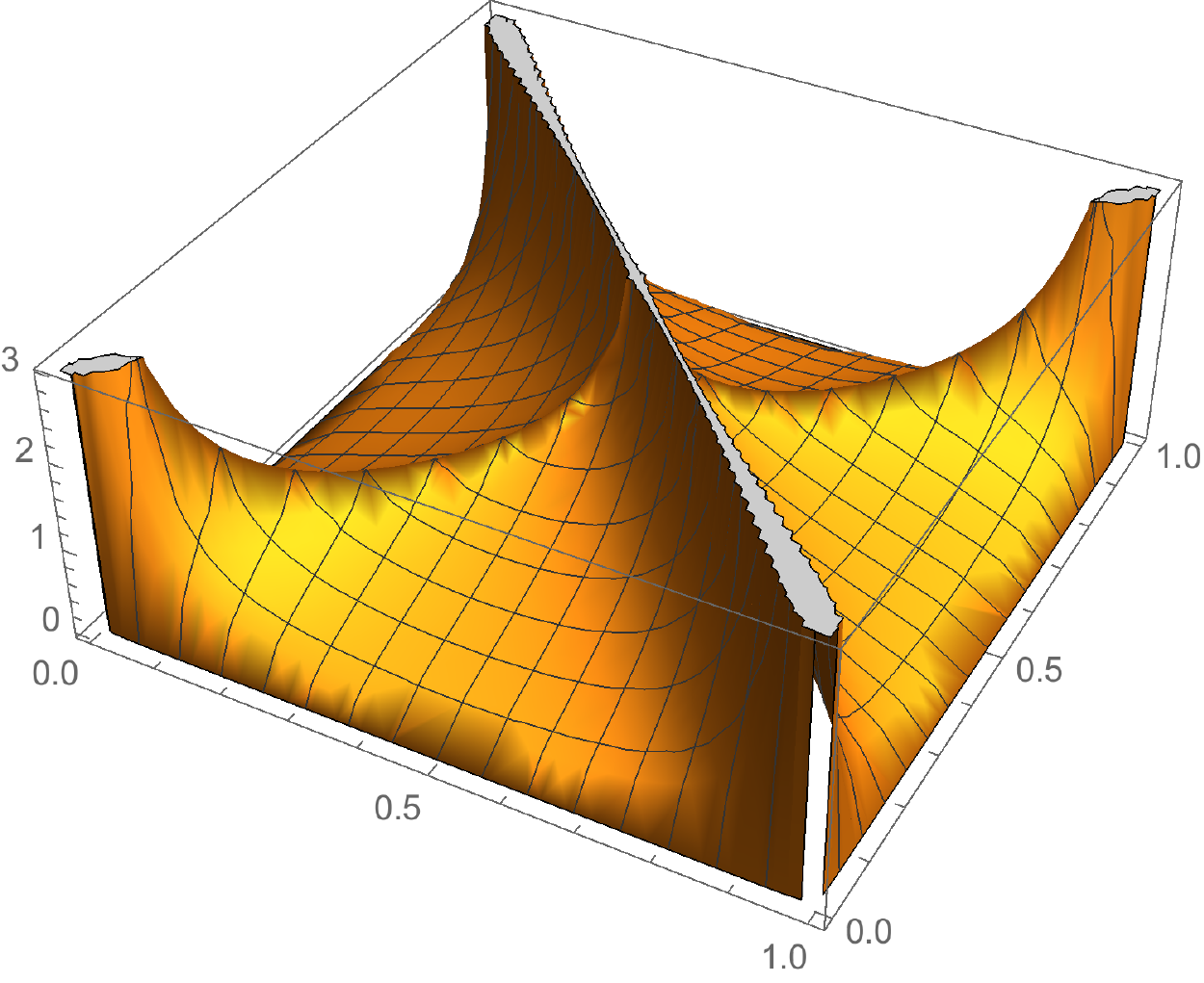}
		\includegraphics[width=0.23\textwidth]{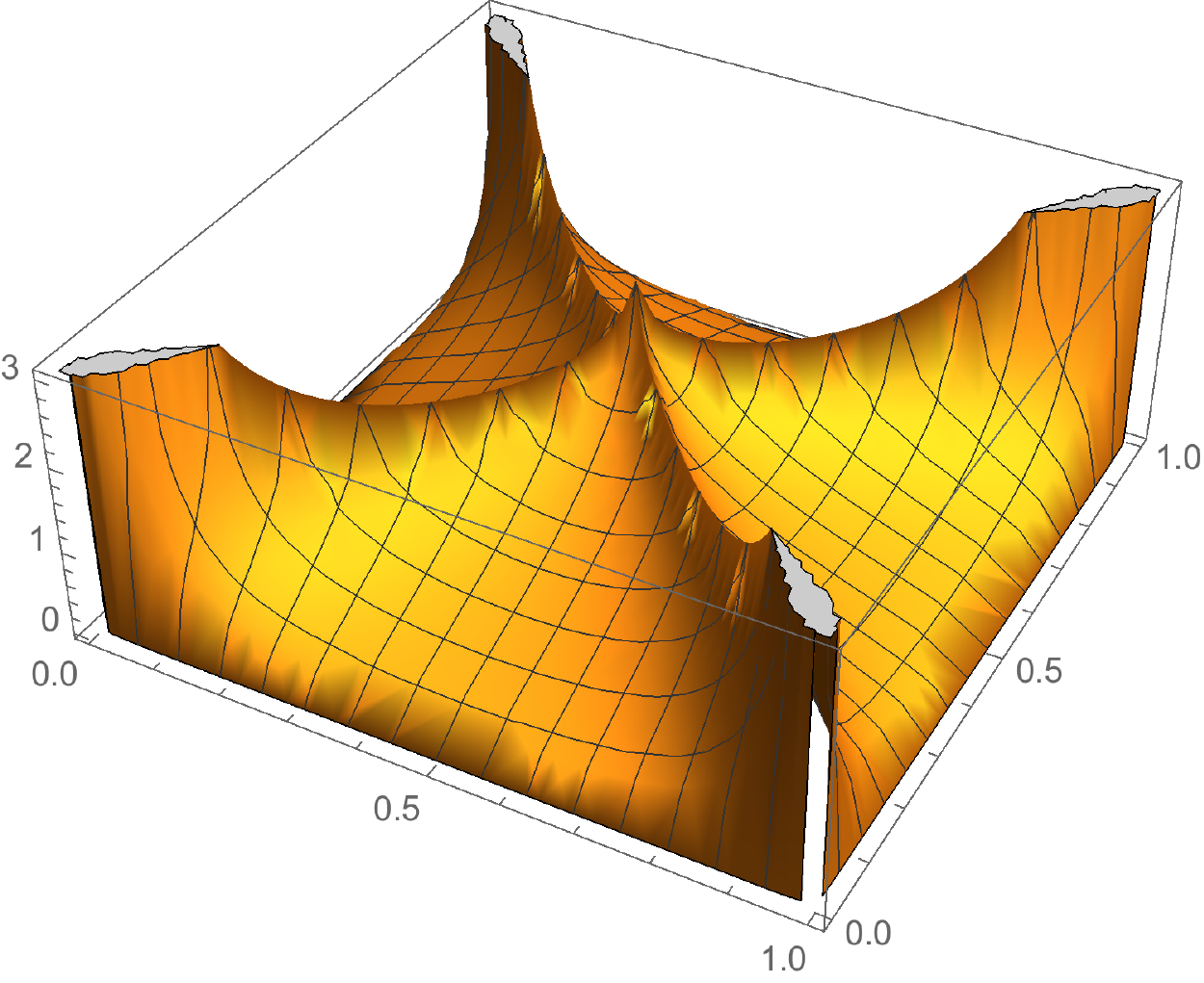}
		\includegraphics[width=0.23\textwidth]{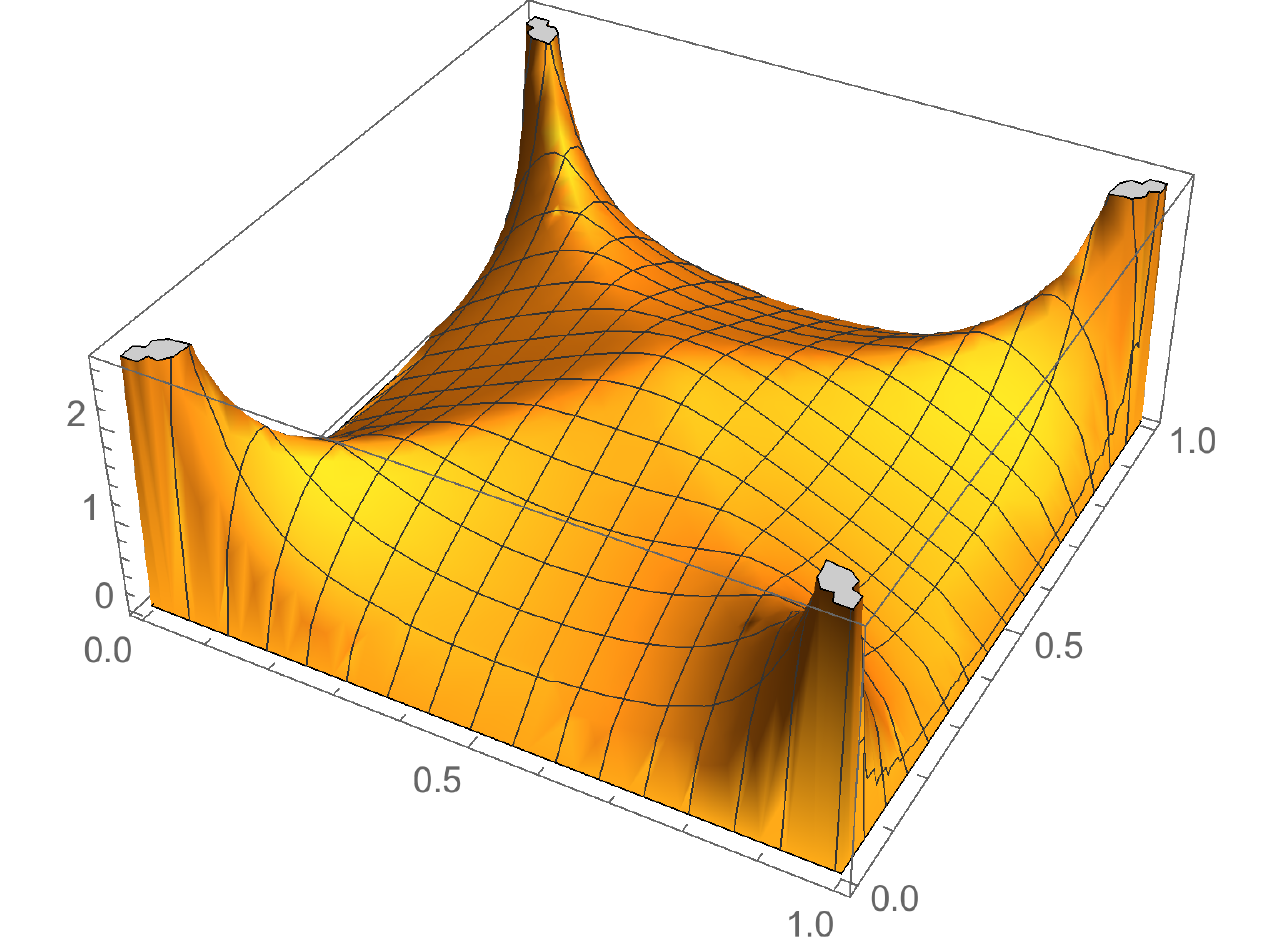}
		\caption{\textbf{From left to right:} The diagrams of the densities $p_S^{0.1}(x,y)$, $p_S^{0.4}(x,y)$, $p_S^{0.5}(x,y)$, and $p_B(x,y)$.}\label{fig:perm_densities1}
	\end{figure}
	
	\begin{figure}[ht]
		\centering
		\includegraphics[width=0.3\textwidth]{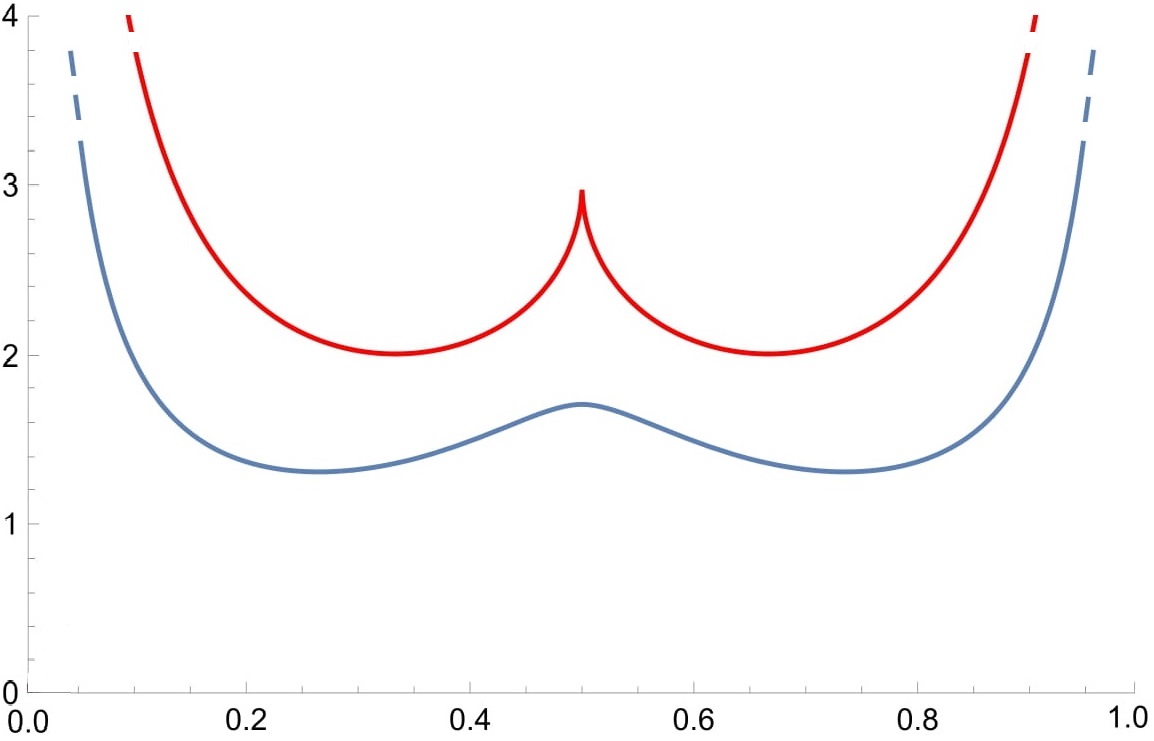}
		\includegraphics[width=0.3\textwidth]{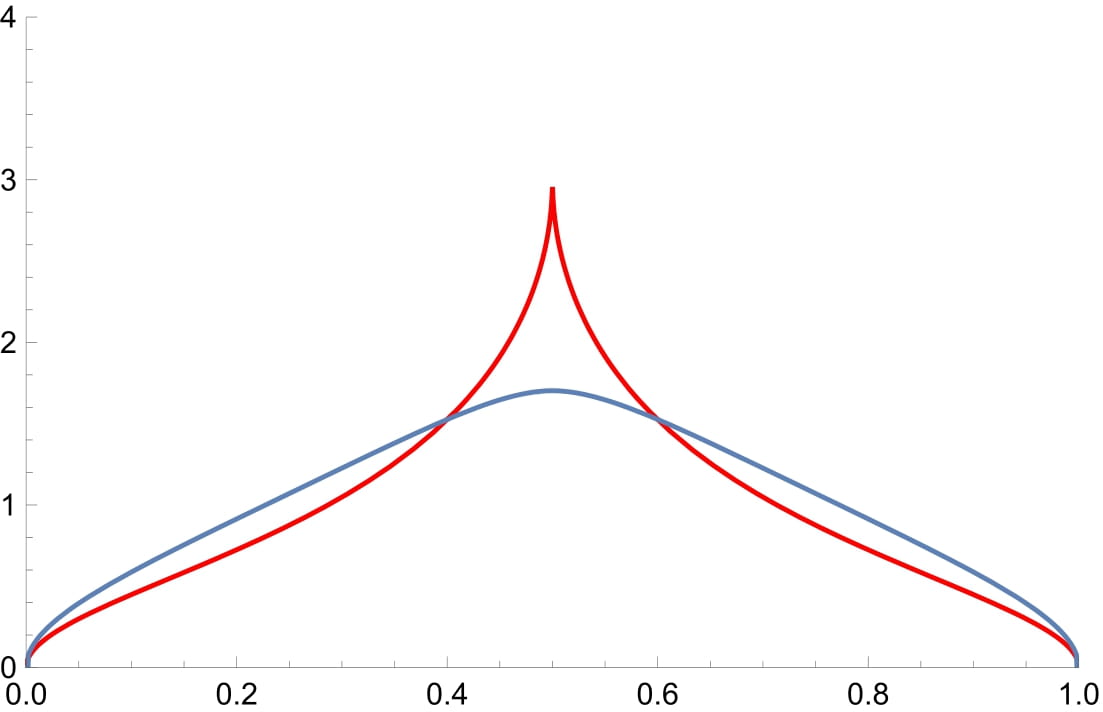}
		\includegraphics[width=0.3\textwidth]{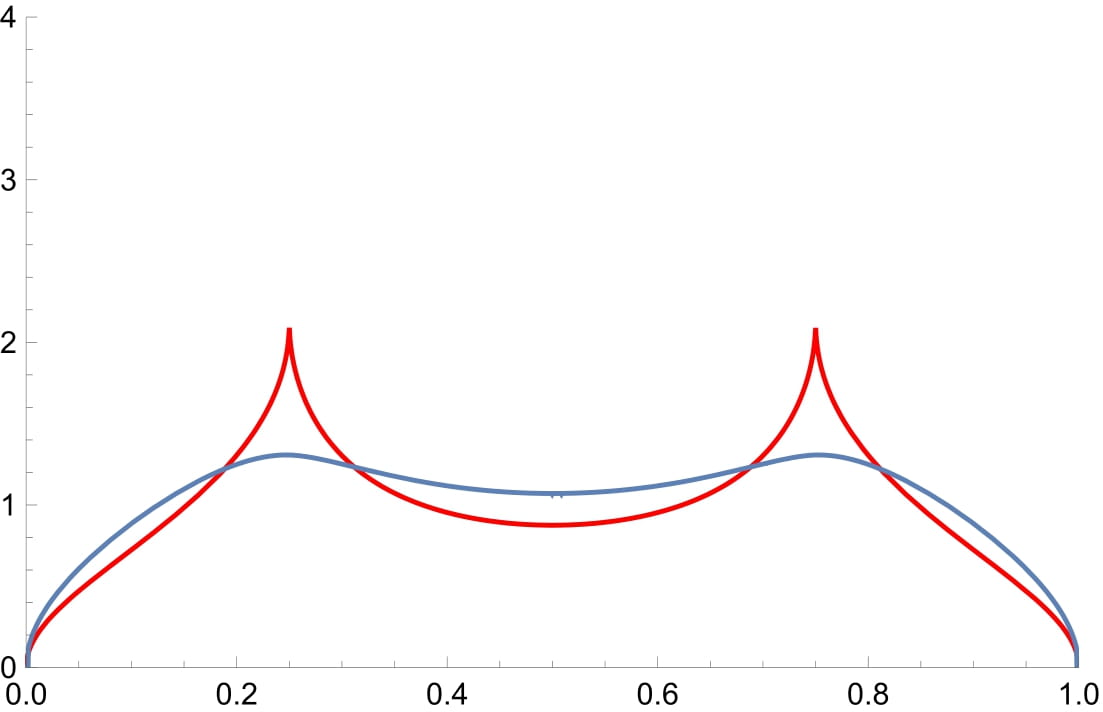}
		\caption{Some sections of the densities $p_S^{0.5}(x,y)$ and $p_B(x,y)$. \textbf{From left to right:} In red (resp.\ in blue) we plot the diagrams of $p_S^{0.5}(x,x)$ (resp.\ $p_B(x,x)$), $p_S^{0.5}(x,1/2)$ (resp.\ $p_B(x,1/2)$), and $p_S^{0.5}(x,1/4)$ (resp.\ $p_B(x,1/4)$).}\label{fig:perm_densities2}
	\end{figure}
	
	\subsubsection{Positivity of pattern densities for the Baxter permuton}\label{sect:posit}
	{Our second result is Theorem~\ref{thm:positivity_baxter} states that the Baxter permuton a.s.\ contains a positive density of every possible (standard) pattern.}  
	To state our result, we first define the permutation induced by $k$ points in the square $[0,1]^2$. Recall that $[n]=\left\{1,\dots,n\right\}$.
	
	\begin{definition}
		Let $(\vec{x},\vec{y})=((x_1,y_1),\dots,(x_k,y_k))=(x_i,y_i)_{i\in[k]}$ be a sequence of $k$ points in $[0,1]^2$ with distinct $x$ and $y$ coordinates. 
		The \emph{$x$-reordering} of $(\vec{x},\vec{y})$, denoted by $(x_{(i)},y_{(i)})_{i\in[k]}$, is the unique reordering of the sequence $(\vec{x},\vec{y})$ such that
		$x_{(1)}<\cdots<x_{(k)}$.
		The values $(y_{(1)},\ldots,y_{(k)})$ are then in the same
		relative order as the values of a unique permutation of size $k$, called the \emph{permutation induced by} $(\vec{x},\vec{y})$ and denoted by $\Perm_k(\vec{x},\vec{y})=\Perm_k((x_i,y_i)_{i\in[k]})$. 
	\end{definition}
	
	We now define the random permutation of size $k$ induced by a deterministic permuton. 
	
	\begin{definition}
		Let $\mu$ be a deterministic permuton and $k\in\mathbb Z_{>0}$. Let $(\vec{X},\vec{Y})=(X_i, Y_i)_{i\in [k]}$ be an i.i.d.\ sequence with distribution $ \mu$. 
		We denote by $\Perm_k(\mu,\vec{X},\vec{Y})$ the random permutation induced by $(\vec{X},\vec{Y})$. 
	\end{definition}
	
	We will also consider random permutations induced by random permutons $ \mu$. In order to do that, we need to construct a sequence
	$(X_i,Y_i)_{i \in [k]}$, where the points $(X_i,Y_i)$ are independent with common distribution $ \mu$ conditionally on $ \mu$. This is possible up to considering a new probability space where the joint distribution
	of $( \mu, (X_i,Y_i)_{i \in [k]})$ is determined as follows: for every positive measurable functional $H : \mathcal{M} \times [0, 1]^{2k} \to \mathbb R$,
	\begin{equation*}
		\mathbb E[H( \mu, (X_i,Y_i)_{i \in [k]})]=\mathbb E\left[\mathbb E\left[\int_{[0,1]^{2k}}H( \mu, (x_i,y_i)_{i \in [k]})\prod_{i=1}^k \mu(dx_i,dy_i)\middle| \mu\right]\right].
	\end{equation*}
	
	We now recall some standard notation related to permutation patterns; see Appendix~\ref{sect:intro_patt} for more details.  
	Let $\mathcal{S}_n$ be the set of permutations of size $n$ and $\mathcal{S}=\bigcup_{n\in\Z_{>0}} \mathcal{S}_n$ be the set of permutations of finite size. Fix $k\leq n$ and $\sigma\in\mathcal S_n$.
	Given a subset $I$ of cardinality $k$ of the indices of $\sigma$, the \emph{pattern induced by} $I$ in $\sigma$, denoted $\pat_{I}(\sigma)$, is the permutation corresponding to the diagram obtained by rescaling the points $(i,\sigma(i))_{i\in I}$ in a $|I|\times |I|$ table (keeping the relative position of the points). If $\pat_{I}(\sigma)=\pi\in \mathcal S_k$ we will say that $(\sigma(i))_{i\in I}$ is an \emph{occurrence} of $\pi$ in $\sigma$.
	We denote by $\occ(\pi,\sigma)$ the number of occurrences of a pattern $\pi$ in a permutation $\sigma$.
	Moreover, we denote by $\pocc(\pi,\sigma)$ the proportion of occurrences of $\pi$ in $\sigma,$ that is,
	\begin{equation*}
		\pocc(\pi,\sigma)=\frac{\occ(\pi,\sigma)}{\binom{n}{k}}.
	\end{equation*}
	We finally recall an important fact about permuton convergence. Suppose $( \sigma_n)_n$ is a sequence of random permutations converging in distribution in the permuton sense to a limiting random permuton $ \mu$, i.e.\ $ \mu_{ \sigma_n}\xrightarrow[]{d} \mu$. Then, from \cite[Theorem 2.5]{bassino2017universal}, it holds that
	$\left(\pocc(\pi, \sigma_n)\right)_{\pi\in\mathcal{S}}$ converges in distribution in the product topology as $n\to \infty$ to the random vector $\left(\pocc(\pi, \mu)\right)_{\pi\in\mathcal{S}}$, where the random variables $\pocc(\pi, \mu)$ are defined for all $\pi\in\mathcal{S}$ as follows
	\begin{equation}\label{eq:occ_permuton}
		\pocc(\pi, \mu)=\bbP(\Perm_k( \mu,\vec{X},\vec{Y})=\pi| \mu)=
		\int_{[0,1]^{2k}}
		\mathds{1}_{\left\{\Perm_k((x_i,y_i)_{i\in[k]})=\pi\right\}}
		\prod_{i=1}^k \mu(dx_i,dy_i).
	\end{equation}
	
	\begin{theorem}\label{thm:positivity_baxter}
		For all patterns $\pi\in\mathcal S$, it holds that 
		\begin{equation*}
			\pocc(\pi, \mu_{B})
			>0\qquad\text{a.s.}
		\end{equation*}
	\end{theorem}
	
	Our result  is \emph{quenched} in the sense that for almost every realization of the Baxter permuton $ \mu_{B}$,  it contains a strictly positive proportion of every  pattern $\pi\in\mathcal{S}$. Since pattern densities of random permutations converge to  pattern densities of the corresponding limiting random permuton, we have  the following corollary of  Theorems~\ref{thm:baxt_conv}~and~\ref{thm:positivity_baxter}.
	\begin{corollary}\label{corol:positivity3}
		Let $\sigma_n$ be a uniform Baxter permutation of size $n$. Then, for all $\pi\in\mathcal{S}$, we have that
		\begin{equation*}
			\lim_{n\to\infty}\pocc(\pi, \sigma_n)>0\qquad \text{a.s.}
		\end{equation*}
	\end{corollary}
	
	\subsection{Positivity of pattern densities for the skew Brownian permuton}\label{sect:con_perm_SBP}
	
	Permuton limits have been investigated for various models of random permutations. For many models, the permuton limits are deterministic, for instance, 
	Erd\"{o}s-Szekeres permutations \cite{MR2266895}, Mallows permutations \cite{starr2009thermodynamic}, random sorting networks \cite{dauvergne2018archimedean}, almost square permutations \cite{MR4149526,borga2021almost}, and permutations sorted with the \emph{runsort} algorithm \cite{alon2021runsort}. 
	For random constrained permutations which have a scaling limit, the limiting permutons appear to be random in many cases. 
	In~\cite{borga2021skewperm} a two-parameter family of permutons, called the \emph{skew Brownian permuton}, was introduced to cover most of the known examples.
	
	The skew Brownian permuton $\mu_{\rho,q}$ is indexed by $\rho\in(-1,1]$ and $q\in [0,1]$, and $\mu_{-1/2,1/2}$ coincides with Baxter permuton. 
	We now  recall the construction of the skew Brownian permuton for $\rho\in(-1,1)$ and $q\in [0,1]$.
	This is only for completeness since at the technical level we will use an alternative definition coming from SLE/LQG which has proven to be equivalent to Definition~\ref{def:baxter-Z} below; see Section~\ref{sect:rel_with_LQG}. We do not recall the $\rho=1$ case as our theorem only concerns $\rho\in(-1,1)$.
	
	For  $\rho\in(-1,1)$, let $( W_{\rho}(t))_{t\in \mathbb{R}_{\geq 0}}=( X_{\rho}(t), Y_{\rho}(t))_{t\in \mathbb{R}_{\geq 0}}$ be  
	a \emph{two-dimensional Brownian motion of correlation} $\rho$. This is a continuous two-dimensional Gaussian process such that the components $ X_{\rho}$ and $ Y_{\rho} $ are standard one-dimensional Brownian motions, and $\mathrm{Cov}( X_{\rho}(t), Y_{\rho}(s)) = \rho \cdot \min\{t,s\}$. Let  $( E_{\rho}(t))_{t\in [0,1]}$ be a \emph{two-dimensional Brownian loop of correlation} $\rho$. Namely, it is a two-dimensional Brownian motion of correlation $\rho$ conditioned to stay in the non-negative quadrant $\mathbb{R}_{\geq 0}^2$ and to end at the origin, i.e.\ $ E_{\rho}(1)=(0,0)$. 
	For $q\in [0,1]$,  consider the solutions of the following family of stochastic differential equations (SDEs) indexed by $u\in [0,1]$ and driven by $ E_{\rho} = ( X_{\rho}, Y_{\rho})$:
	
	\begin{equation}\label{eq:flow_SDE_gen}
		\begin{cases}
			d Z_{\rho,q}^{(u)}(t) = \mathds{1}_{\{ Z_{\rho,q}^{(u)}(t)\geq 0\}} d Y_{\rho}(t) - \mathds{1}_{\{ Z_{\rho,q}^{(u)}(t)< 0\}} d  X_{\rho}(t)+(2q-1)\cdot d L^{ Z_{\rho,q}^{(u)}}(t),& t\in(u,1),\\
			Z_{\rho,q}^{(u)}(t)=0,&  t\in[0,u],
		\end{cases} 
	\end{equation}
	where $ L^{ Z_{\rho,q}^{(u)}}(t)$ is the symmetric local-time process at zero of $ Z_{\rho,q}^{(u)}$, i.e.\ 
	\begin{equation*}
		L^{ Z^{(u)}_{\rho,q}}(t)=\lim_{\varepsilon\to 0}\frac{1}{2\varepsilon}\int_0^t\mathds{1}_{\left\{ Z^{(u)}_{\rho,q}(s)\in[-\varepsilon,\varepsilon]\right\}}ds.
	\end{equation*}
	The solutions to  the SDEs~\eqref{eq:flow_SDE_gen}  exist and are unique thanks to \cite[Theorem 1.7]{borga2021skewperm}.
	The collection of stochastic processes $\left\{ Z^{(u)}_{\rho,q}\right\}_{u\in[0,1]}$ is called the \emph{continuous coalescent-walk process} driven by $( E_{\rho},q)$. Here  $\left\{Z^{(u)}_{\rho,q}\right\}_{u\in[0,1]}$ is defined in the following sense: for a.e.\ $\omega,$ $ Z^{(u)}_{\rho,q}(\omega)$ is a solution for almost every $u\in[0,1]$. For more explanations see the discussion below \cite[Theorem 1.7]{borga2021skewperm}. Let
	
	\begin{equation*}
		\varphi_{ Z_{\rho,q}}(t)=
		\Leb\left( \big\{x\in[0,t)\,|\, Z_{\rho,q}^{(x)}(t)<0\big\} \cup \big\{x\in[t,1]\,|\, Z_{\rho,q}^{(t)}(x)\geq0\big\} \right), \quad t\in[0,1].
	\end{equation*}  
	
	\begin{definition}\label{def:baxter-Z}
		Fix $\rho\in(-1,1)$ and $q\in[0,1]$. The \emph{skew Brownian permuton} of parameters $\rho, q$, denoted $ \mu_{\rho,q}$, is the push-forward of the Lebesgue measure on $[0,1]$ via the mapping $(\mathbb{I},\varphi_{ Z_{\rho,q}})$, that is
		\begin{equation*}
			\mu_{\rho,q}(\cdot)=(\mathbb{I},\varphi_{ Z_{\rho,q}})_{*}\Leb (\cdot)= \Leb\left(\{t\in[0,1]\,|\,(t,\varphi_{ Z_{\rho,q}}(t))\in \cdot \,\}\right).
		\end{equation*} 
	\end{definition}
	
	We mention that it is also possible to generalize the previous construction when $\rho=1$. Then the permuton $\mu_{1,q}$  coincides with the biased Brownian separable permuton $\mu^{1-q}_S$ of parameter $1-q$ mentioned before; see \cite[Section 1.4 and Theorem 1.12]{borga2021skewperm} for further explanations.
	
	We now summarize the list of known random permutations which have the skew Brownian permuton as scaling limit.
	Uniform separable permutations \cite{bassino2018separable} converge to $ \mu_{1,1/2}$. Uniform permutations in  proper substitution-closed classes \cite{bassino2017universal, MR4115736} or classes having a finite combinatorial specification for the substitution decomposition \cite{bassino2019scaling} converge (under some technical assumptions) to $ \mu_{1,q}$, where the parameter $q$ depends on the chosen class.
	Uniform Baxter permutations converge to $\mu_{-1/2,1/2}$, namely the Baxter permuton. Uniform semi-Baxter permutations \cite{borga2021permuton}, converge to $ \mu_{\rho,q}$, where $\rho=-\frac{1+\sqrt 5}{4}\approx -0.8090$ and 
	$q=1/2$. Uniform strong-Baxter permutations \cite{borga2021permuton}, converge to $ \mu_{\rho,q}$, where
	$\rho\approx-0.2151$ is the unique real solution of the polynomial $1+6\rho+8\rho^2+8\rho^3$
	and $q\approx0.3008$ is the unique real solution of the polynomial $-1+6q-11q^2+7q^3$.
	
	We will not give the detailed definitions of all random constrained  permutations mentioned above but emphasize an important division. On the one hand, models converging to $\mu_{\rho,q}$ with $\rho\neq 1$ are similar to Baxter permutations in the following sense: their constraints are not defined by avoiding certain patterns completely, but only avoiding them when the index locations satisfy certain additional conditions; see e.g.\ Definition~\ref{def:baxter}. We  say that such permutations avoid \emph{generalized}  patterns.
	On the other hand, models converging towards the biased Brownian separable permuton  $\mu_{1,q}$, they avoid a certain set of patterns completely. For example, separable permutations avoid the patterns $2413$ and $3142$. We  say that such permutations avoid \emph{(standard)}  patterns. (Here the word \emph{standard} is added to distinguish from generalized patterns.)
	
	Our next theorem, which generalizes Theorem~\ref{thm:positivity_baxter}, shows that in the scaling limit, the division between $\rho\neq 1$ and $\rho=1$ becomes the following. On the one hand, for $\rho\neq 1$, the permuton  $\mu_{\rho,q}$  almost surely admits a positive density of any (standard) pattern. On the other hand, the biased Brownian separable permuton  $\mu_{1,q}$ presents a zero density of some (standard) patterns. For instance, $\mu_{1,q}$ almost surely avoids all the (standard) patterns that are not separable; see \cite[Definition 5.1]{bassino2017universal}.
	
	\begin{theorem}\label{thm:positivity}
		For all $(\rho,q)\in(-1,1)\times(0,1)$ and all (standard) patterns $\pi\in\mathcal S$, it holds that 
		\begin{equation*}
			\pocc(\pi, \mu_{\rho,q})
			>0\qquad \text{a.s.}
		\end{equation*}
	\end{theorem}
	
	Note that the latter theorem answers \cite[Conjecture 1.20]{borga2021skewperm}. By Theorem~\ref{thm:positivity},  if a sequence of random permutations avoiding (standard) patterns converges to a skew Brownian permuton then it has to be the biased Brownian separable permuton. Namely, we have the following result.
	\begin{corollary}\label{corol:positivity}
		Let $\mathcal C$ be a family of permutations avoiding (standard) patterns. Let $ \sigma_n$ be a random permutation of size $n$ in $\mathcal C$. Assume that for some $(\rho,q)\in(-1,1]\times(0,1)$ it holds that $ \mu_{ \sigma_n}\xrightarrow[n\to\infty]{d} \mu_{\rho,q}$.
		Then $\rho=1$.
	\end{corollary}
	
	\subsection{Relation with SLE and LQG}\label{sect:rel_with_LQG}
	
	We now review the connection between the skew Brownian permuton and SLE/LQG  established in \cite[Theorem 1.17]{borga2021skewperm}. Then we explain our proof techniques. Precise definitions and more background on various SLE/LQG related objects   will be given in Section~\ref{sect:perm-LQG-cor}.
	
	\subsubsection{The skew Brownian permuton and the SLE-decorated  quantum sphere}\label{sect:sbp_lqg_rel}
	
	Fix $\gamma \in(0,2)$ and some angle $\theta\in[-\frac{\pi}{2}, \frac{\pi}{2}]$. In what follows we consider:
	\begin{itemize}
		\item a unit-area $\gamma$-Liouville quantum sphere $(\wh{{\mathbb C}}, h, 0, \infty)$ with two marked points at 0 and $\infty$ and associated $\gamma$-LQG area measure $\mu_h$ (see Definition~\ref{def-quantum-sphere});
		\item an independent whole-plane GFF $\wh h$ (see Section~\ref{sect:gff});
		\item two space-filling SLE$_{\kappa'}$ counterflow lines of $\wh{h}$ in $\wh{{\mathbb C}}$ with angle $0$ and ${\theta-\frac{\pi}{2}}$ constructed from angle $\frac{\pi}{2}$ and $\theta$ flow lines  with $\kappa'=16/\gamma^2$ (see Section~\ref{sect:sle_ig}). We denote these two space-filling SLE$_{\kappa'}$ curves from $\infty$ to $\infty$ by $\eta'_{0}$ and $\eta'_{\theta-\frac{\pi}{2}}$.
	\end{itemize}    
	We emphasize the independence of the counterflow lines and the quantum sphere. In addition, we assume that the curves $\eta'_{0}$ and $\eta'_{\theta-\frac{\pi}{2}}$ are parametrized so that $\eta'_{0}(0)=\eta'_{0}(1)=\eta'_{\theta-\frac{\pi}{2}}(0)=\eta'_{\theta-\frac{\pi}{2}}(1)=\infty$ and $\mu_h(\eta'_{0}([s,t])) =\mu_h(\eta'_{\theta-\frac{\pi}{2}}([s,t])) = t-s$ for $0\le s<t\le 1$.
	We have the following result.
	
	\begin{theorem}[{\cite[Theorem 1.17]{borga2021skewperm}}]\label{thm:sbpfromlqg}
		Fix $\gamma\in(0,2)$ and $\theta\in[-\frac\pi 2,\frac\pi 2]$. Let $(\wh{{\mathbb C}}, h, 0, \infty)$ and $( \eta'_0, \eta'_{\theta-\frac{\pi}{2}})$ be the {unit-area} $\gamma$-Liouville quantum sphere and the pair of space-filling SLE$_{\kappa'}$ introduced above. 
		For $t\in[0, 1]$, let $\psi_{\gamma,\theta}(t)\in[0, 1]$ denote the first time\footnote{We recall that space-filling SLE curves have multiple points. Nevertheless, for each $z\in\mathbb C$, a.s. $z$ is not a multiple point of $\eta'_{\theta-\frac{\pi}{2}}$, i.e., $\eta'_{\theta-\frac{\pi}{2}}$ hits $z$ exactly once. Since $h$ is independent from $(\eta'_0,\eta'_{\theta-\frac{\pi}{2}})$ and $\eta'_0$ and $\eta'_{\theta-\frac{\pi}{2}}$ are parametrized by $\mu_{h}$-mass, a.s. the set
			of times $t \in [0, 1]$ such that $\eta'_0$ is a multiple point of  $\eta'_{\theta-\frac{\pi}{2}}$ has zero Lebesgue measure.} at which $\eta'_{\theta-\frac{\pi}{2}}$ hits the point $\eta'_0(t)$.
		Then the random permuton
		\[(\Id,\psi_{\gamma,\theta})_{*}\Leb\] 
		is a skew Brownian permuton of parameter $\rho=-\cos(\pi\gamma^2/4)\in(0,1)$ and $q=q_\gamma(\theta)\in[0,1]$.
	\end{theorem}
	For every fixed $\gamma\in(0,2)$, the function \[q_\gamma(\theta):\left[-\frac{\pi}{2},\frac{\pi}{2}\right]\to[0,1]\] 
	is a decreasing homeomorphism and therefore has an inverse function $\theta_\gamma(q)$.
	Finally, for all $\theta\in[0,\pi/2]$ and all $\gamma\in(0,2)$, it holds that $q_\gamma(\theta) + q_\gamma(-\theta) = 1$. In particular, 	$q_\gamma(0)=1/2$ for all $\gamma\in(0,2)$. The Baxter permuton corresponds to $\gamma=4/3$ and $\theta=0$. 
	
	\subsubsection{Proof techniques for the main results}\label{sect:proof_tec}
	
	To prove Theorem~\ref{thm-baxter-density}, we first extend Theorem~\ref{thm:sbpfromlqg} to give a more explicit description of the skew Brownian permuton measure $\mu_{\rho,q}$ in terms of a unit-area quantum sphere and two space-filling SLEs, i.e.\ we prove the following.
	
	\begin{proposition}\label{prop:permuton-q-area}
		Fix $\gamma\in(0,2)$ and $\theta\in[-\frac\pi 2,\frac\pi 2]$. 
		Let $(\wh{{\mathbb C}}, h, 0, \infty)$ and $( \eta'_0, \eta'_{\theta-\frac{\pi}{2}})$ be the {unit-area} $\gamma$-Liouville quantum sphere and the pair of space-filling SLE$_{\kappa'}$ introduced above. Let also $\rho\in(-1,1)$ and $q\in[0,1]$ be such that $\rho=-\cos(\pi\gamma^2/4)$ and $q=q_\gamma(\theta)$, and consider the skew Brownian permuton $\mu_{\rho,q}$ constructed as in Theorem~\ref{thm:sbpfromlqg}.	
		Then, almost surely, for every $0\le x_1\le x_2\le 1$ and $0\le y_1\le y_2\le 1$, 
		\begin{equation*}
			\mu_{\rho,q}\Big([x_1, x_2]\times [y_1, y_2]\Big)=\mu_{h}\Big(  \eta'_0([x_1, x_2])\cap \eta'_{\theta-\frac{\pi}{2}}([y_1, y_2])\Big).
		\end{equation*}
	\end{proposition}
	Using Proposition~\ref{prop:permuton-q-area}, we express the intensity measure $\bbE[\mu_{\rho,q}]$ in terms of a quantum sphere decorated by certain flow lines of a Gaussian free field, which are simple SLE$_\kappa$ curves with $\kappa=\gamma^2\in(0,4)$. This is proved in Proposition~\ref{prop-baxter-disk} using the rerooting invariance of marked points for quantum spheres (\cite{DMS14}; see also Proposition~\ref{prop-qs-resampling} below) and the fact that the outer boundaries of the SLE$_{16/\kappa}$-type curves $\eta'_0$ and $\eta_{\theta-\frac\pi2}'$ are simple SLE$_\kappa$ curves. Using conformal welding results and scaling properties of quantum disks and spheres (\cite{AHS20}; see also Section~\ref{sec:conf-welding}), this leads to a simpler expression for  $\bbE[\mu_{\rho,q}]$ via the density function $p_W(a,\ell_1,\ell_2)$ of the area $a$ of quantum disks with given quantum boundary lengths $\ell_1$ and $\ell_2$ (see Theorem~\ref{thm-baxter-density-general}). When $(\rho,q)=(-1/2,1/2)$, the density $p_W(a,\ell_1,\ell_2)$ is the same as the density of the duration of a Brownian excursion in a cone of angle $\frac{\pi}{3}$ (as argued in \cite[Section 7]{AHS20}; see also Proposition~\ref{prop-disk-excursion}). The latter can be computed using standard heat equation argument as done in Section~\ref{sec:density-a} via \cite{Ty85}. 
	We finally briefly explain the relation between the integrals in \eqref{eqn-baxter-density} and spherical tetrahedra (see Section~\ref{sect:rel_tetra}).   
	
	To prove Theorem~\ref{thm:positivity}, we begin by observing that the occurrence of a fixed (standard) pattern in $\mu_{\rho,q}$ can be reformulated in terms of a specific condition on the crossing and merging order of some collection of flow lines of a Gaussian free field (see Lemma~\ref{prop-pos-dens-key0} and Figure~\ref{fig-flowlines-perm} for a precise statement). Then building on the key result from \cite[Lemma 3.8]{MS17}, which roughly speaking states that a simple SLE$_\kappa$ curve can approximate any continuous simple curves with positive probability, we prove that this crossing and merging condition holds with positive probability (the main difficulty here is that we need to look at several flow lines of different angles together). Finally, by the scaling invariance of the whole-plane GFF and a tail triviality argument, we conclude the proof of Theorem~\ref{thm:positivity}. 
	
	We conclude the introduction with three observations. First, qualitatively, our method works equally well for $\mu_{\rho,q}$ with any $\rho\neq 1$. 
	But quantitatively,  the Baxter permuton corresponds to a special case where the function
	$p_W(a,\ell_1,\ell_2)$ is significantly simpler than the general case; see Remark~\ref{rem:generalization}. This is why Theorem~\ref{thm-baxter-density} is restricted to the Baxter case while Theorem~\ref{thm:positivity} is for the general case. 
	
	Second, the  angle $\theta$ in Theorem~\ref{thm:sbpfromlqg} has a simple permutation interpretation.
	\begin{proposition}\label{prop-expected-occ-21}
		For all $(\rho,q)\in(-1,1)\times[0,1]$, let $\theta\in[-\frac{\pi}{2}, \frac{\pi}{2}]$ be related to $q$ by the relation $q = q_\gamma(\theta)$ given in Theorem~\ref{thm:sbpfromlqg} with $\gamma\in(0,2)$ such that $\rho=-\cos(\pi\gamma^2/4)$. Then
		\begin{equation}\label{eq:inversion}
			\bbE[(\pocc(21, \mu_{\rho,q}))] = \frac{\pi-2\theta}{2\pi}.
		\end{equation}
	\end{proposition}
	The third and fourth named authors of this paper are working on deriving an exact formula for $q_\gamma(\theta)$ with Ang. For the Baxter permuton  $\bbE[(\pocc(21, \mu_{\rho,q}))]=\frac{1}{2}$ by symmetry, which is consistent with  $\theta=0$. In fact,  we can express $\bbE[(\pocc(\pi, \mu_{\rho,q}))] $ for any pattern $\pi$ in terms of SLE and LQG, but for $\pi\neq 21$  we do not find any exact information  as enlightening as~\eqref{eq:inversion}; see the end of Section~\ref{sec:density}. Note that Proposition~\ref{prop-expected-occ-21} answers \cite[Conjecture 1.22]{borga2021skewperm}, which conjectures that the expected proportion of inversions in $\mu_{\rho,q}$ is an increasing function of $q$ (recall that $\theta$ is a decreasing function of $q$). 
	
	Finally, our work falls into the line of work proving integrability results (i.e., exact formulas) via SLE and/or LQG techniques. Integrability results for SLE and LQG can be proven via a variety of methods, see e.g.\ \cite{schramm-left-passage,lawler-book,dub-watts,garban-tf,ssw-radii,schramm-zhou,sw-watts,hongler-smirnov-number-of-clusters,beliaev-izyurov, alberts-kozdron-lawler-radial-green,beliaev-viklund,lenells-viklund,als-cle-radius} for SLE results and \cite{krv-dozz,remy-fb-formula} for LQG results. As a particular class of methods, couplings between SLE and LQG \cite{shef-zipper,DMS14} allows to exploit the interplay between SLE and LQG to prove new results about both objects. For example, the KPZ formula \cite{kpz-scaling} has been used to predict exponents of statistical mechanics models and dimensions of SLE curves via combinatorics on planar maps; see e.g.\ \cite{ghm-kpz} for rigorous computations in this spirit. Chen, Curien and Maillard \cite{chen-curien-maillard} give a heuristic derivation of the conformal radius formula in \cite{ssw-radii} by using the coupling with LQG; see also \cite{hl-cle-on-lqg} for a proof via LQG techniques. A number of integrability results based on the SLE and LQG coupling have been established by subsets of the coauthors of this paper~\cite{AHS20,AHS21,ARS21,AS21}, and our Theorem~\ref{thm-baxter-density} can be viewed as a part of this ongoing endeavor.
	
	\bigskip
	
	\noindent\textbf{Acknowledgments.} 
	N.H.\ was supported by grant 175505 of the Swiss National Science Foundation.
	X.S.\ was supported by the  NSF grant DMS-2027986 and the Career grant 2046514.
	P.Y.\ was supported by the NSF grant DMS-1712862. We thank two anonymous referees for all their precious and useful comments.

	\section{Permuton-LQG/SLE correspondence}\label{sect:perm-LQG-cor}
	
	\label{sect:perm-lqg-corr}
	
	This section collects the background needed for later sections. We review in Section~\ref{sect:q_surf_and_SLE} some definitions related to the Gaussian free field, quantum surfaces and SLE curves. Then, in Section~\ref{sec:lqg-baxter} we prove Proposition~\ref{prop:permuton-q-area}. 
	
	\medskip
	
	\noindent\textbf{Notation.} In this paper we {will often} work with non-probability measures. We extend the terminology of ordinary probability to this setting: For a (possibly infinite but $\sigma$-finite) measure space $(\Omega, \mathcal{F}, M)$, we say $X$ is a \emph{random variable} if $X$ is an $\mathcal{F}$-measurable function with its \textit{law} defined via the push-forward measure $M_X=X_*M$. In this case, we say $X$ is \textit{sampled} from $M_X$ and write $M_X[f]$ for $\int f(x)M_X(dx)$. By \textit{conditioning} on some event $E\in\mathcal{F}$ with $0<M[E]<\infty$, we are referring to the probability measure $\frac{M[E\cap \cdot]}{M[E]} $  over the space $(E, \mathcal{F}_E)$ with $\mathcal{F}_E = \{A\cap E: A\in\mathcal{F}\}$. Finally, for a finite measure $\mu$ we let $|\mu|$ be its total mass and $\mu^\#: = \frac{\mu}{|\mu|}$ be its normalized version.
	
	\subsection{Gaussian free fields, quantum surfaces and SLE curves}\label{sect:q_surf_and_SLE}
	
	We assume throughout the rest of the paper that $\gamma\in(0,2)$, unless otherwise stated. We introduce the following additional parameters defined in terms of $\gamma$
	\begin{equation*}\label{eq:paramLQG}
		\qquad Q=2/\gamma+\gamma/2, \qquad \chi=2/\gamma-\gamma/2, \qquad \kappa=\gamma^2, \qquad {\kappa'}=16/\gamma^2.
	\end{equation*} 
	
	\subsubsection{Gaussian free fields}\label{sect:gff}
	
	Recall that the \emph{Gaussian free field (GFF) with free boundary conditions} (resp.\ \emph{zero boundary conditions}) $h$ on a planar domain $D\subsetneq\mathbb C$ is defined by taking an orthonormal basis $\{f_n\}$ of $H(D)$ (resp.\ $H_0(D)$), the Hilbert space completion of the
	set of smooth functions on $D$ with finite Dirichlet energy (resp.\ {finite Dirichlet energy and} compact support) with respect to the Dirichlet inner product, an i.i.d.\ sequence $\{\alpha_n\}$ of standard normal random variables, and considering the sum $h = \sum_{n=1}^\infty \alpha_nf_n$. This series converges in an appropriate Sobolev space and hence in the space of distributions; see for instance \cite[Theorem 1.24]{berestycki-powell-notes}. In the case of the free boundary GFF, we view $h$ as a distribution modulo a global additive constant (see also \cite[Definition 5.2]{berestycki-powell-notes}). 
	
	The \emph{whole-plane Gaussian free field} $h$, viewed as a distribution on $\mathbb{C}$ modulo a global additive constant, is defined in a similar manner as the free boundary GFF, but with $D=\mathbb{C}$ (see for instance \cite[Section 5.4]{berestycki-powell-notes}). Sometimes we will fix this additive constant or view the whole-plane GFF as a distribution modulo a global additive integer multiple of some other fixed constant; if this is done, it will be always specified in the paper. We refer to \cite{She07, pw-gff-notes, berestycki-powell-notes} for more background on the GFF.
	
	\subsubsection{Quantum surfaces}\label{sect:quan_surf}

	Consider the space of pairs $(D,h)$, where $D\subseteq \mathbb C$ is a planar domain and $h$ is a distribution on $D$ (often some variant of the GFF). Define the equivalence relation $\sim_\gamma$, where $(D, h)\sim_\gamma(\wt{D}, \wt{h})$ if there is a conformal map $\varphi:\wt{D}\to D$ such that 
	\begin{equation}\label{eqn-lqg-changecoord}
		\wt{h} = h\circ \varphi+Q\log |\varphi'|.
	\end{equation}
	A \textit{quantum surface} $S$ is an equivalence class of pairs $(D,h)$ under the relation $\sim_\gamma$, and we say $(D,h)$ is an \emph{embedding} of $S$ if $S = (D,h)/\mathord\sim_\gamma$.
	In this paper, the domain $D$ shall be either the upper half plane $\mathbb{H}:=\{z\in\mathbb{C}:\mathrm{Im}\ z>0\}$, the Riemann sphere $\wh{\mathbb{C}}:=\mathbb{C}\cup\{\infty\}$, or a planar domain cut out by $\mathrm{SLE}_\kappa$ curves. We will often abuse notation and identify the pair $(D,h)$ with its equivalence class, e.g.\ we may refer to $(D,h)$ as a quantum surface (rather than a representative of a quantum surface).
	
	A \emph{quantum surface with} $k$ \emph{marked points} is an equivalence class of elements of the form
	$(D, h, x_1,\dots,x_k)$, where $(D,h)$ is a quantum surface, the points  $x_i\in {\overline{D}}$, and with the further
	requirement that marked points (and their ordering) are preserved by the conformal map $\varphi$ in \eqref{eqn-lqg-changecoord}.
	
	A \emph{curve-decorated quantum surface} is an equivalence class of tuples $(D, h, \eta_1, ..., \eta_k)$,
	where $(D,h)$ is a quantum surface, $\eta_1, ..., \eta_k$ are curves in $\overline D$, and with the further
	requirement that $\eta$ is preserved by the conformal map $\varphi$ in \eqref{eqn-lqg-changecoord}. Similarly, we can
	define a curve-decorated quantum surface with $k$ marked points. Throughout this paper, the curves $\eta_1, ..., \eta_k$ are $\mathrm{SLE}_\kappa$ type curves (which have conformal invariance properties) sampled independently of the surface $(D, h)$.  
	
	Given some variant $h$ of the GFF, one can make sense of the bulk measure $\mu_h$, where $\mu_h(A) = \int_A e^{\gamma h(z)}\,dz$, by considering the circle average $h_\e(z)$ of $h$ on the circle $\partial B_\e(z)$ and taking the weak limit of $\e^{\gamma^2/2}e^{\gamma h_\e(z)}dz$ as $\e\to0$; see for instance \cite{kahane,rhodes-vargas-log-kpz,DS11}. The measure $\mu_h$ is then called the \emph{$\gamma$-Liouville quantum gravity area measure}. Similarly, one can define a length measure $\nu_h$, called the \emph{$\gamma$-Liouville quantum gravity length measure}, on certain curves in the closure $\overline{D}$ of $D$. Notice that $\mu_h$ and $\nu_h$ depend on the parameter $\gamma$, but we skip $\gamma$ from the notation since the value of $\gamma$ will always be implicitly understood from the context. The two measures satisfy natural scaling properties, namely $\mu_{h+c}(A) = e^{\gamma c}\mu(A)$ and  $\nu_{h+c}(S) = e^{\frac{\gamma}{2} c}\nu(S)$ for {$c\in\R$ an arbitrary constant}. See \cite{DS11} for more details on these bulk and boundary Liouville quantum gravity measures. A self-contained  introduction to GFF, Liouville quantum gravity measures, and quantum surfaces, can be also found in \cite[Section 3.2-3]{GHS19} or in the lecture notes \cite{berestycki-powell-notes}.
	
	Now we formally introduce \emph{quantum spheres} and \emph{quantum disks}, which are the main types of quantum surfaces considered in this paper and are defined in terms of some natural variants of the GFF. The reader can find intuitive explanations of the next definitions at the end of this section. We also highlight that it is not strictly necessary to understand the technical details involved in the following definitions in order to then follow our proofs in the consecutive sections.
	
	As argued in \cite[Section 4.1]{DMS14}, when $D =\wh{\mathbb{C}}$ (resp.\ $D =\mathbb{H}$), we have the decomposition $H(\wh{\mathbb{C}}) = H_1(\wh{\mathbb{C}})\oplus H_2(\wh{\mathbb{C}})$ (resp.\  $H(\mathbb{H}) = H_1(\mathbb{H})\oplus H_2(\mathbb{H})$), where $H_1(\wh{\mathbb{C}})$ (resp.\ $H_1(\mathbb{H})$) is the subspace of radially symmetric functions, and $H_2(\wh{\mathbb{C}})$  (resp.\ $H_2(\mathbb{H})$) is the subspace of functions having mean 0 on all circles $\{|z|=r\}$ (resp.\ semicircles $\{|z|=r,\ \text{Im}\ z>0\}$). For a whole-plane GFF $h$, we can decompose $h=h_1+h_2$, where $h_1$ and $h_2$ are independent distributions given by the projection of $h$ onto $H_1(\wh{\mathbb{C}})$ and $ H_2(\wh{\mathbb{C}})$, respectively. We remark that $h_1$ is defined modulo an additive constant while $h_2$ is not.
	The same result applies for the upper half plane $\mathbb{H}$. 
	
	Since a quantum surface is an equivalence class of pairs $(D,\psi)$ (or, more generally, an equivalence class of tuples $(D,\psi,z_1,\dots,z_k)$ with $z_1,\dots,z_k\in \ol D$), in order to describe the law of a quantum sphere, we will start by describing the law of its random field $\psi$.
	
	\begin{definition}[{Quantum sphere}]\label{def-quantum-sphere}
		Fix $\gamma\in(0,2)$ and let $(B_s)_{s\ge0}$ and $(\wt{B}_s)_{s\ge0}$ be independent standard one-dimensional Brownian motions. Fix a weight parameter $W>0$ and set $\alpha := Q - \frac{W}{2\gamma}$.  Let $\mathbf{c}$ be sampled from the infinite measure $\frac{\gamma}{2}e^{2(\alpha-Q)c}dc$ on $\R$ independently from $(B_s)_{s\ge0}$ and $(\wt{B}_s)_{s\ge0}$.
		Let	
		\begin{equation*}
			X_t=\left\{ \begin{array}{rcl} 
				B_t+\alpha t+\mathbf{c} & \mbox{for} & t\ge 0,\\
				\wt{B}_{-t} +(2Q-\alpha) t+\mathbf{c} & \mbox{for} & t<0,
			\end{array} 
			\right.
		\end{equation*}
		conditioned on  $B_t-(Q-\alpha)t<0$ and  $ \wt{B}_t - (Q-\alpha)t<0$ for all $t>0$. 	Let $h$ be a whole-plane GFF on $\wh{\mathbb{C}}$ independent of $(X_t)_{t\in\bbR}$ with projection onto $ H_2(\wh{\mathbb{C}})$ given by $h_2$. We consider the random distribution 
		\begin{equation*}
			\psi(\cdot)=X_{-\log|\cdot|} + h_2(\cdot)\, .
		\end{equation*}
		Let $\mathcal{M}_2^{\mathrm{sph}}(W)$ be the infinite measure describing the law of $(\wh{\mathbb{C}}, \psi,0,\infty)/\mathord\sim_\gamma$. 
		We call a sample from $\mathcal{M}_2^{\textup{sph}}(W)$ a \emph{quantum sphere} of weight $W$ with two marked points.
		
		A \emph{unit-area $\gamma$-quantum sphere} with two marked points is the quantum sphere of weight $4-\gamma^2$ with two marked points conditioned on having total $\gamma$-LQG area measure $\mu_{\psi}(\wh{\mathbb{C}})$ equal to one.
	\end{definition}
	
	It is explained in \cite[Sections 4.2 and 4.5]{DMS14} that the considered conditioning on $B$ and $\wt B$, along with the conditioning on the quantum area of a weight $4-\gamma^2$ quantum sphere, can be made rigorous via a limiting procedure, although we are conditioning on probability zero events.
	
	We remark that the weight $4-\gamma^2$ here is ``typical'' because in this case the two marked points (which currently  {correspond} to 0 and $\infty$) can be realized as independent samples from the {$\gamma$-}LQG area measure $\mu_\psi$ (see Proposition~\ref{prop-qs-resampling} below for a precise statement). This important rerooting invariance property shall later be used in Section~\ref{sec:density-a} in order to compute the density of the Baxter permuton via quantum surfaces. 
	
	We now turn to the definition of \emph{quantum disks}, which is splitted in two different cases: \emph{thick quantum disks} and \emph{thin quantum disks}.
	
	\begin{definition}[Thick quantum disk]\label{def-quantum-disk}
		Fix $\gamma\in(0,2)$ and let $(B_s)_{s\ge0}$ and $(\wt{B}_s)_{s\ge0}$ be independent standard one-dimensional Brownian motions.  Fix a weight parameter $W\ge\frac{\gamma^2}{2}$ and let $\beta = \gamma+ \frac{2-W}{\gamma}\le Q$. Let $\mathbf{c}$ be sampled from the infinite measure $\frac{\gamma}{2}e^{(\beta-Q)c}dc$ on $\R$ independently from $(B_s)_{s\ge0}$ and $(\wt{B}_s)_{s\ge0}$.
		Let 	
		\begin{equation*}
			Y_t=\left\{ \begin{array}{rcl} 
				B_{2t}+\beta t+\mathbf{c} & \mbox{for} & t\ge 0,\\
				\wt{B}_{-2t} +(2Q-\beta) t+\mathbf{c} & \mbox{for} & t<0,
			\end{array} 
			\right.
		\end{equation*}
		conditioned on  $B_{2t}-(Q-\beta)t<0$ and  $ \wt{B}_{2t} - (Q-\beta)t<0$ for all $t>0$. Let $h$ be a free boundary  GFF on $\mathbb{H}$ independent of $(Y_t)_{t\in\bbR}$ with projection onto $H_2(\mathbb{H})$ given by $h_2$. Consider the random distribution
		\begin{equation*}
			\psi(\cdot)=X_{-\log|\cdot|} + h_2(\cdot) \, .
		\end{equation*}
		Let $\mathcal{M}_2^{\mathrm{disk}}(W)$ be the infinite measure describing the law of $({\mathbb{H}}, \psi,0,\infty)/\mathord\sim_\gamma $. 
		We call a sample from $\mathcal{M}_2^{\textup{disk}}(W)$ a \emph{quantum disk} of weight $W$ with two marked points.
		
		We call $\nu_\psi((-\infty,0))$ and $\nu_\psi((0,\infty))$ the left and right boundary quantum length of the quantum disk  $(\mathbb{H}, \psi, 0, \infty)$.
	\end{definition}
	
	When $0<W<\frac{\gamma^2}{2}$, we define the \emph{thin quantum disk} as the concatenation of weight $\gamma^2-W$ thick disks with two marked points as in \cite[Section 2]{AHS20}.

	\begin{definition}[Thin quantum disk]\label{def-thin-disk}
		Fix $\gamma\in(0,2)$. For $W\in(0, \frac{\gamma^2}{2})$, the infinite measure $\mathcal{M}_2^{\textup{disk}}(W)$ is defined as follows. First sample a random variable $T$ from the infinite measure $(1-\frac{2}{\gamma^2}W)^{-2}\textup{Leb}_{\mathbb{R}_+}$; then sample a Poisson point process $\{(u, \mathcal{D}_u)\}$ from the intensity measure $\mathds{1}_{t\in [0,T]}dt\times \mathcal{M}_2^{\textup{disk}}(\gamma^2-W)$; and finally consider the ordered (according to the order induced by $u$) collection of doubly-marked thick quantum disks $\{\mathcal{D}_u\}$, called a \emph{thin quantum disk} of weight $W$.
		
		Let $\mathcal{M}_2^{\textup{disk}}(W)$ be the infinite measure describing the law of this ordered collection of doubly-marked quantum disks $\{\mathcal{D}_u\}$.
		The left and right boundary length of a sample from $\mathcal{M}_2^{\textup{disk}}(W)$ is set to be equal to the sum of the left and right boundary lengths of the quantum disks $\{\mathcal{D}_u\}$.
	\end{definition}
	
	We give a heuristic interpretation of the last definition. Note that one can interpret the ordered collection of doubly-marked quantum disks $\{\mathcal{D}_u\}$ as if we are concatenating the surfaces $\{\mathcal{D}_u\}$ by ``gluing'' them at their marked points, as shown in Figure~\ref{fig:thindisk}. These collections of doubly-marked quantum surfaces are sometime called \emph{beaded surfaces}.

	\begin{figure}[ht!]
		\begin{center}
			\includegraphics[width=.6\textwidth]{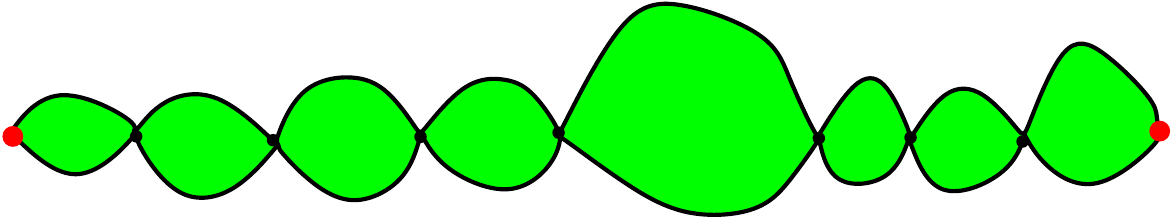}  
			\caption{Schematic representation of a sample of a thin quantum disk of weight $W\in(0,\frac{\gamma^2}{2})$ as the concatenation of weight $\gamma^2-W$ thick quantum disks. The green bubbles correspond to the thick quantum disks $\{\mathcal{D}_u\}$ involved in the construction. Note that there are in fact (countably) infinite many thick quantum disks which are not drawn near the two endpoints (shown in red) and between each pair of macroscopic disks.}\label{fig:thindisk}
		\end{center}
	\end{figure}
	
	\begin{remark}
		The quantum spheres and disks introduced in this section can also be equivalently constructed via methods in Liouville conformal field theory (LCFT); see e.g.~\cite{DKRV16, HRV-disk} for these constructions and see \cite{ahs-sphere,cercle-quantum-disk,AHS20} for proofs of equivalence with the surfaces defined above. Fundamental properties of the surfaces such as structure constants and correlation functions have also been established via methods in LCFT \cite{krv-dozz,gkrv-bootstrap,gkrv-genus}, confirming predictions from the physics literature \cite{bpz-conformal-symmetry,zz-dozz,do-dozz}. The quantum spheres and disks also arise as the scaling limit of certain random planar maps. For example, when $\gamma=\sqrt{8/3}$, $\mathcal{M}_2^{\textup{disk}}(2)$ is the law of the LQG realization of the Brownian disk with two marked boundary points with free area and free boundary length~\cite{lqg-tbm-1,lqg-tbm-2}, where we recall that the Brownian disk is the scaling limit of triangulations or quadrangulations with disk topology sampled from the critical Boltzmann measure~\cite{BM15browniandisk,GM19browniandisk}.
	\end{remark}

	We conclude this section by briefly explaining some intuitions behind the definitions of quantum spheres and quantum disks. We remark that these explanations are not  
	needed to follow the rest of the paper. 
	
	Following~\cite[Section 1.2]{DMS14}, we explain why the weight parameter $W$ encodes in some sense how ``thick/thin'' the surface is.
	In Definition~\ref{def-quantum-sphere} (resp.\ Definition~\ref{def-quantum-disk}), the process $X_t$ (resp.\ $Y_t$) encodes the average of the field $\psi$ on $\wh{\mathbb{C}}$ (resp.\ on $\mathbb H$) over the circle (resp.\ semicircle) of radius $e^{-t}$ centered at 0, and can be defined by taking the logarithm of Bessel excursions of dimensions $2+\frac{2}{\gamma^2}W$ and $1+\frac{2}{\gamma^2}{W}$, respectively; see~\cite[Section 4]{DMS14}.  Note that the dimension of the Bessel process increases as $W$ increases. 
	
	Since the processes $X_t$ and $Y_t$ (and so also the corresponding Bessel excursions) are sampled from infinite measures, the measures for quantum spheres and quantum disks are infinite. Moreover, the random constant $\mathbf{c}$ appearing in the two definitions encodes the largest value of the processes $X_t-Qt$ and $Y_t-Qt$, which is attained by definition at $t=0$ (equivalently, the random constant $e^\mathbf{c}$ encodes the largest value attained during the corresponding Bessel excursions). The process $X_t-Qt$ (resp.\ $Y_t-Qt$) encodes some other field average process when the quantum sphere (resp.\ the quantum disk) is embedded onto the cylinder $\mathbb{R}\times[0,i2\pi]/\mathord\sim$ instead of $\wh{\mathbb{C}}$, where $\sim$ stands for the equivalence relation $x\sim y$ if $x=y+2\pi i$ (resp.\ onto the strip $\mathbb R\times (0,i2\pi)$ instead of $\mathbb H$). The term $Qt$ in these  processes comes from the change of coordinates formula in \eqref{eqn-lqg-changecoord}. As a consequence, the random constant $\mathbf{c}$ reflects the largest value of these field average processes under these other embeddings. Note that $\mathbf{c}$ ``tends'' to be larger when $W$ increases.

	The two marked points on each quantum surface are related to the starting and the ending points of the corresponding Bessel excursion, and near these marked points the field $\psi$ looks like $h-\beta\log|\cdot|$ (if the surface is embedded in $\bbH$ with the relevant marked point at 0), where for quantum spheres $\beta = Q-\frac{W}{2\gamma}$ and $h$ is a whole-plane GFF, while for quantum disks $\beta = \gamma+\frac{2-W}{\gamma}$ and $h$ is a free boundary  GFF on $\mathbb{H}$. That is, near these marked points, the field $\psi$ looks like a GFF plus a $\beta$-log-singularity, and such singularity is smaller when $W$ increases, decreasing the amount of mass in the neighborhood of the two marked points. In fact, for readers familiar with the Liouville CFT approach, as proved in~\cite{AHS21}, a weight $W$ quantum sphere with two marked points can be understood as the uniform embedding of the Liouville field $\mathrm{LF}_{\mathbb{C}}^{(\beta,0),(\beta,\infty)}$ with insertion points $(\beta,0),(\beta,\infty)$, while a weight $W$ quantum disk with two marked points can be realized as the uniform embedding of the Liouville field $\mathrm{LF}_{\mathbb{H}}^{(\beta,0),(\beta,\infty)}$ with insertion points $(\beta,0),(\beta,\infty)$. Moreover, as shown in~\cite{DMS14, AHS20, ASY22}, the weight $W$ is additive under the operation of \emph{conformal welding}, which shall be further discussed in Section~\ref{sec:conf-welding}. 
	
	\subsubsection{SLE curves and imaginary geometry}\label{sect:sle_ig}
	Now we briefly recall the construction of  the Schramm--Loewner evolution ($\mathrm{SLE}_\kappa$) curves with parameter $\kappa>0$, which were introduced by Schramm~\cite{Sch00} and arise as scaling limits of many statistical physics models, see e.g.~\cite{SmirnovICM,Law08,Sch06ICM}. Roughly speaking, on the upper half plane $\mathbb{H}$, the $\mathrm{SLE}_\kappa$ curve $\eta$ can be described via the Loewner equation 
	\begin{equation*}
		\frac{dg_t}{dt}(z) = \frac{2}{g_t(z)-W_t}; \qquad g_0(z) = z;
	\end{equation*}
	where $g_t$ is the conformal map from $\mathbb{H}\backslash\eta([0,t])$ to $\mathbb{H}$ with $\lim_{|z|\to\infty} |g_t(z)-z|=0$ and $W_t$ is $\sqrt{\kappa}$ times a standard Brownian motion. This curve starts at $0$, ends at $\infty$, and travels on the upper half plane $\mathbb{H}$ \cite{RS05}. Moreover, it has conformal invariance properties and therefore the definition can be extended to other domains (with other starting and ending points) via conformal maps. When $\kappa\in (0,4]$ the curve is simple, while for $\kappa>4$ the curve is self hitting (later on, when $\kappa>4$ we denote $\kappa$ by ${\kappa'} = \frac{16}{\kappa}>4$ for $\kappa\in (0,4]$, being consistent with \cite{MS16a, MS17}). We refer the reader to the lecture note \cite{berestycki-sle-notes} for more background on SLEs. 
	
	It is also possible to define a variant of the $\mathrm{SLE}_\kappa$ on $\mathbb{H}$ from 0 to $\infty$ known as the $\mathrm{SLE}_\kappa(\rho_1;\rho_2)$ on $\mathbb{H}$ from 0 to $\infty$, where $\rho_1,\rho_2>-2$. For $\kappa\in(0,4)$ the curve is still simple but a.s.\ hits (countably) infinitely many times the left (resp.\ right) boundary of $\mathbb H$  when $\rho_1<\frac\kappa 2-2$ (resp.\ $\rho_2<\frac\kappa 2-2$), and it does not hit at all the left (resp.\ right) boundary of $\mathbb H$  when $\rho_1\geq\frac\kappa 2-2$ (resp.\ $\rho_2\geq\frac\kappa 2-2$). Also in this case, the definition can be extended to other domains (with other starting and ending points) via conformal maps. See \cite[Section 2]{MS16a} for more details. 
	
	We shall also consider the \emph{whole-plane $\mathrm{SLE}_\kappa(\rho)$} for $\rho>-2$, which is a random curve in $\widehat{\mathbb{C}}$ from a starting point $z\in\mathbb C$ to $\infty$. For $\kappa\in(0,4)$ the curve hits itself (countably) infinitely many times when $\rho<\frac\kappa 2-2$, but does not hit itself at all when $\rho\geq\frac\kappa 2-2$. See \cite[Section 2.1]{MS17} for more details.

	Given a whole-plane GFF $\wh h$  viewed modulo a global additive integer multiple of $2\pi \chi$ (see~\cite[Section 2.2]{MS17} for further details) and $\theta\in\R$, one can construct the $\theta$-\emph{angle} \emph{flow lines} $\eta_{\theta}^{z}$ of $\wh h$ (or more precisely of $e^{i(\wh h/\chi+\theta)}$) from $z\in \mathbb{C}$ to $\infty$ as shown in \cite{MS16a, MS17}. The marginal law of $\eta_\theta^z$ is that of a whole-plane SLE$_{\kappa}(2-\kappa)$ curve from $z$ to $\infty$. We remark that we measure angles in counter-clockwise order, where zero angle corresponds to the north direction.
	
	For distinct $z,w\in \mathbb{C}$, the flow lines $\eta_{\theta}^{z}$ and $\eta_{\theta-\pi}^{z}$ cannot cross $\eta_{\theta}^{w} \cup \eta_{\theta-\pi}^{w}$, but they may hit and bounce off when $\kappa \in (2,4)$. See Figure \ref{flow-lines} for an illustration.
	Additionally, flow lines of $\wh{h}$ with the same angle started at different points of $\mathbb{Q}^2$ merge into each other when intersecting and form a tree \cite[Theorem 1.9]{MS17}. This gives an ordering of $\mathbb{Q}^2$, where $z\preceq w$ whenever the $\theta$-angle flow line from $z$ merges into the $\theta$-angle flow line from $w$ on the left side. Equivalently, $z \preceq w$ if and only if $z$ lies in a connected component of $\mathbb{C}\setminus (\eta_{\theta}^{w}\cup \eta_{\theta-\pi}^{w})$ which lies to the left of $\eta_{\theta}^{w}$ and to the right of $\eta_{\theta-\pi}^{w}$. 
	One can construct a unique Peano curve which visits points of $\mathbb{Q}^2$ with respect to this ordering \cite[Theorem 1.16]{MS17}. We call this curve the \textit{space-filling SLE$_{{\kappa'}}$  counterflow line of $\wh{h}$ in $\widehat{\mathbb{C}}$ with angle ${\theta-\frac{\pi}{2}}$} {and we denote this curve by $\eta'_{\theta-\frac{\pi}{2}}$.} 
	It follows from the construction that a.s.\ for any fixed $z\in\mathbb C$, the flow lines $\eta_{\theta}^{z}$ and $\eta_{\theta-\pi}^{z}$ are the left and right boundaries of $\eta'_{\theta-\frac{\pi}{2}}$ stopped upon hitting $z$.
	We highlight that any pair of counterflow lines of $\wh{h}$ with different angles or different starting points are not independent and their coupling is encoded via the whole-plane GFF $\wh h$.

	\begin{figure}[ht!]
		\begin{center}
			\includegraphics[width=.9\textwidth]{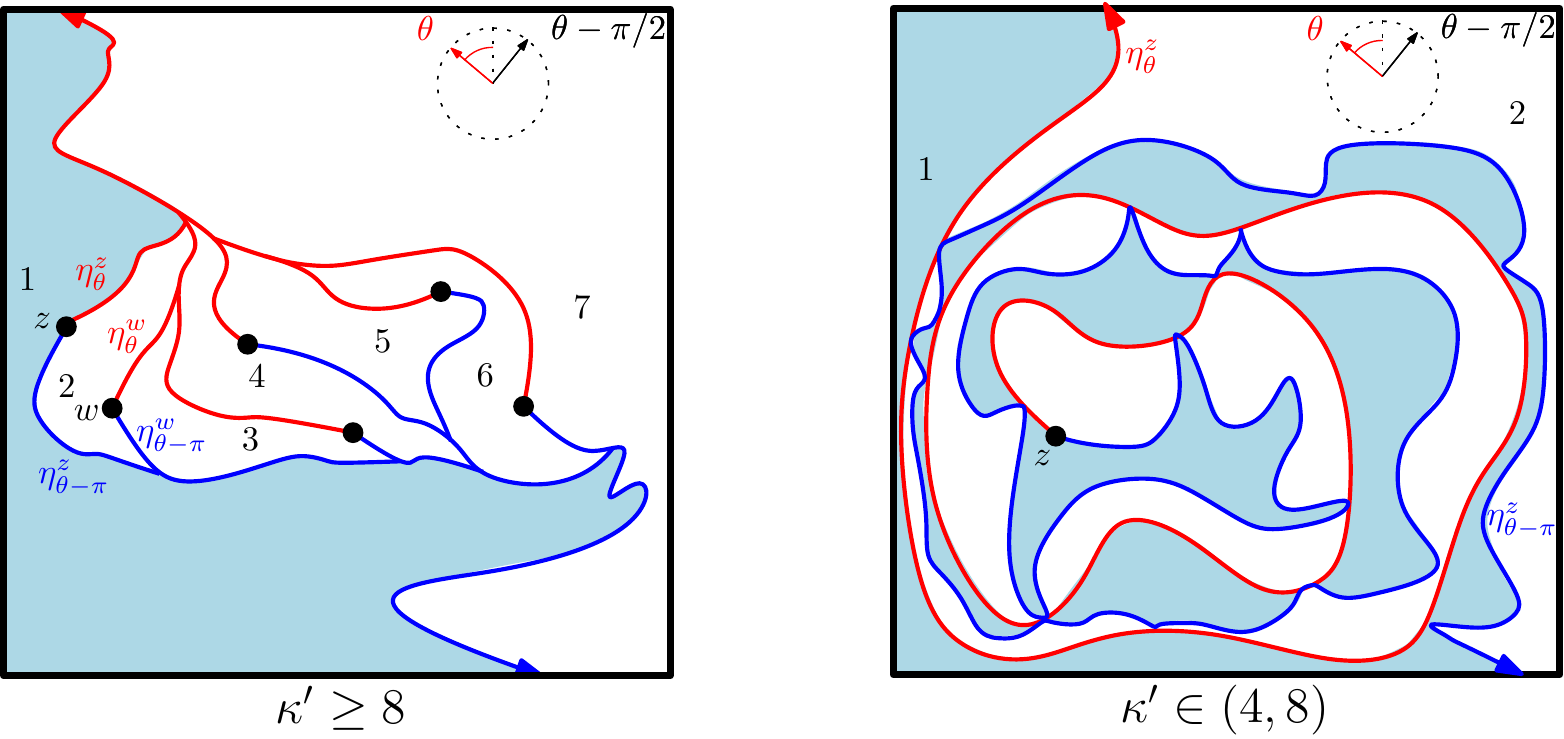}  
			\caption{\label{flow-lines} \textbf{Left:} The squared box is a portion of the complex plane $\mathbb C$. We fix $\theta$ as shown in the picture. We plot in red the flow lines $\eta_{\theta}^{x}$  and in blue the flow lines $\eta_{\theta-\pi}^{x}$ for six points $x\in\mathbb C$. For every $x\in\mathbb C$, the flow lines $\eta_{\theta}^{x}$ and $\eta_{\theta-\pi}^{x}$ are the left and right outer boundaries of the space-filling SLE$_{\kappa'}$ $\eta'_{\theta-\frac{\pi}{2}}$ stopped when it hits $x$. The space-filling SLE$_{\kappa'}$ $\eta'_{\theta-\frac{\pi}{2}}$ fills in the regions 1 (in light blue),2,3,4,5,6, and 7 in this order.  The left figure illustrates the case when ${\kappa'}\geq 8$. \textbf{Right:} The same illustration as in the left-hand side when ${\kappa'}\in (4,8)$. In this case we just considered a single point $z\in\mathbb C$. The flow lines $\eta_{\theta}^{z}$ (in red) and $\eta_{\theta-\pi}^{z}$ (in blue) started from the same point $z$ can hit each other and bounce off. The space-filling SLE$_{\kappa'}$ $\eta'_{\theta-\frac{\pi}{2}}$ fills first the regions 1 (in light blue) and then the region 2 (in white).}
		\end{center}
	\end{figure}

	\subsection{LQG description of the skew Brownian permuton}\label{sec:lqg-baxter}
	
	In this section, we prove Proposition~\ref{prop:permuton-q-area} by directly applying Theorem~\ref{thm:sbpfromlqg} (we invite the reader to review the statement of these theorem now that all the objects have been properly introduced). 
	Fix $\gamma \in(0,2)$ and an angle $\theta\in[-\frac{\pi}{2}, \frac{\pi}{2}]$. In what follows we consider:
	\begin{itemize}
		\item a unit-area $\gamma$-Liouville quantum sphere $(\widehat{{\mathbb C}}, h, 0, \infty)$ with two marked points at 0 and $\infty$;
		\item the associated $\gamma$-LQG area measure $\mu_h$ (which is in particular a random, non-atomic, Borel probability measure on $\widehat{{\mathbb C}}$ which assigns positive mass to every open subset of $\widehat{{\mathbb C}}$);
		\item an independent whole-plane GFF $\widehat h$ (viewed modulo a global additive integer multiple of $2\pi \chi$);
		\item two space-filling SLE$_{{\kappa'}}$ counterflow lines of $\widehat{h}$ in $\widehat{\mathbb{C}}$ with angles $0$ and ${\theta-\frac{\pi}{2}}$, denoted by $\eta':=\eta'_{0}$ and $\eta'_{\theta-\frac{\pi}{2}}$ and started from $\infty$ at time $t=0$ and ending at $\infty$ at $t=1$;
		\item the skew Brownian permuton $\mu_{\rho,q}$ with  $\rho=-\cos(\pi\gamma^2/4)\in(-1,1)$ and $q=q_\gamma(\theta)\in[0,1]$ as constructed in Theorem~\ref{thm:sbpfromlqg} from $(\mu_h,\eta',\eta'_{\theta-\frac{\pi}{2}})$.
	\end{itemize}   
	Also recall that $\psi_{\gamma,\theta}(t)$ is the first time when $\eta'_{\theta-\frac{\pi}{2}}$ hits $\eta'(t)$, and that the curves $\eta'$ and $\eta'_{\theta-\frac{\pi}{2}}$ are parametrized so that 
	$\mu_h(\eta'([s,t])) =\mu_h(\eta'_{\theta-\frac{\pi}{2}}([s,t])) = t-s$ for $0\le s<t\le 1$.
	
	\begin{proof}[Proof of Proposition~\ref{prop:permuton-q-area}]
		By \cite[Theorem 1.11]{borga2021skewperm}, the random measure $\mu_{\rho,q}$ is almost surely a permuton, i.e., almost surely its marginals are uniform. We first prove that for fixed $0\le x_1\le x_2\le 1$ and $0\le y_1\le y_2\le 1$, $\mu_{\rho,q}([x_1, x_2]\times [y_1, y_2])$ is a.s.\ equal to the quantum area of  $\eta'([x_1, x_2])\cap \eta'_{\theta-\frac{\pi}{2}}([y_1, y_2])$.   By Theorem~\ref{thm:sbpfromlqg}, we know that a.s.\ \begin{equation}\label{eqn-parameterize-1}
			\mu_{\rho,q}([x_1, x_2]\times [y_1, y_2]) = \textup{Leb}(\{t\in [x_1, x_2]: \psi_{\gamma,\theta}(t)\in [y_1, y_2] \}).
		\end{equation}
		Using the fact that the set of multiple points for $\eta'$ and $\eta'_{\theta-\frac{\pi}{2}}$ a.s.\ has zero quantum area (see e.g.\ \cite[Section B.5]{DMS14}), and since we are parameterizing $\eta'_{\theta-\frac{\pi}{2}}$ by quantum area, it follows that a.s.\ for almost every $t\in [0,1]$, $\psi_{\gamma,\theta}(t)\in [y_1, y_2]$ if and only if $\eta'(t)\in  \eta'_{\theta-\frac{\pi}{2}}([y_1, y_2])$. This implies that a.s.\, 	\begin{equation}\label{eqn-parameterize-2}
			\textup{Leb}(\{t\in [x_1, x_2]:  \psi_{\gamma,\theta}(t)\in [y_1, y_2] \}) = \textup{Leb}(\{t\in [x_1, x_2]:  \eta'(t)\in  \eta'_{\theta-\frac{\pi}{2}}([y_1, y_2]) \}).
		\end{equation} 
		Again since we are parameterizing $\eta'$ using quantum area, by \eqref{eqn-parameterize-1} and \eqref{eqn-parameterize-2} it follows that a.s.
		\begin{equation}\label{eq:permuton-q-area}
			\mu_{\rho,q}\Big([x_1, x_2]\times [y_1, y_2]\Big)=\mu_{h}\Big(  \eta'_0([x_1, x_2])\cap \eta'_{\theta-\frac{\pi}{2}}([y_1, y_2])\Big),
		\end{equation} 
		for fixed $0\le x_1\le x_2\le 1$ and $0\le y_1\le y_2\le 1$.
		Now we fix $x_1,y_1,y_2$ and let $x_2$ vary. By
		Fubini's theorem, there exists a set $\mathbf{A}_{x_1,y_1,y_2}\subset [x_1,1]$ with Lebesgue measure zero, such that \eqref{eq:permuton-q-area} holds for any $x_2\in [x_1,1]\backslash\mathbf{A}_{x_1,y_1,y_2}$. Since both sides of \eqref{eq:permuton-q-area} are monotone in $x_2$, and the right-hand side of \eqref{eq:permuton-q-area} is a.s.\ continuous in $x_2$ (this follows because {$\gamma$-}LQG measure a.s.\ has no atoms), we see that a.s.\ for fixed $x_1,y_1,y_2$ and all $x_2\in [x_1,1]$, \eqref{eq:permuton-q-area} holds. We can continue this argument by fixing $y_1$ and $y_2$ and letting both of $x_1$ and $x_2$ vary, and then only fixing $y_1$, and finally letting $x_1,x_2,y_1,y_2$ vary. Therefore we arrive at the conclusion that a.s.\ \eqref{eq:permuton-q-area} holds for all $0\le x_1\le x_2\le 1$ and $0\le y_1\le y_2\le 1$. 
	\end{proof} 
	
	\section{Density of the Baxter permuton}
	\label{sec:density}
	
	In this section, building  on Proposition~\ref{prop:permuton-q-area}, we study the expectation of the skew Brownian permuton $\mu_{\rho,q}$ and express it in terms of the law of the areas of certain quantum disks. 
	In the special case $q = \frac{1}{2}$ and $\rho = -\frac{1}{2}$, i.e.\ when $\mu_{\rho,q}$ is the Baxter permuton, we compute this area law by considering the random duration of certain Brownian excursions and derive Theorem~\ref{thm-baxter-density}. The main tools are the rerooting invariance for marked points of quantum spheres and the conformal welding of quantum disks. 
	
	This section is organized as follows. In Sections~\ref{sec:conf-welding}~and~\ref{sect:res_prop}, we review the input from conformal welding  and the rerooting invariance, respectively. Then, in Section~\ref{sec:density-a}, we give an expression for the intensity measure of the skew Brownian permuton and in particular we prove Theorem~\ref{thm-baxter-density}. Finally, in Section~\ref{sec:E21theta} we will show that the expected occurrence  $\bbE[(\pocc(21, \mu_{\rho,q}))]$ linearly depends on $\theta=\theta_\gamma(q)$, proving Proposition~\ref{prop-expected-occ-21}. 
	
	{Throughout this section we fix $\gamma\in(0,2)$ and $\kappa=\gamma^2\in (0,4)$, except that in Section~\ref{sec:density-Baxter} we restrict to the Baxter case where $\gamma=\sqrt{4/3}$.}
	
	\subsection{Conformal welding of quantum disks}\label{sec:conf-welding}
	
	We start by reviewing in Section~\ref{sec:quantum-disks} the disintegration and the scaling properties of quantum disks and spheres, 
	and then in Section~\ref{sec:qs-conf-welding} we recall the notion of conformal welding of quantum disks from \cite{AHS20}. 
	
	\subsubsection{{P}roperties of quantum disks and quantum spheres}\label{sec:quantum-disks}
	
	We recap the disintegration of measures on quantum surfaces as in \cite[Section 2.6]{AHS20}. For the infinite measure $\mathcal{M}_2^{\text{disk}}(W)$, one has the following \textit{disintegration} for the 
	{quantum boundary length}:
	\begin{equation}\label{eq:dis-formula}
		\mathcal{M}_2^{\text{disk}}(W) = \int_{0}^\infty\int_{0}^\infty \mathcal{M}_2^{\text{disk}}(W; \ell_1, \ell_2)\,d\ell_1\,d\ell_2,
	\end{equation} 
	where $\mathcal{M}_2^{\text{disk}}(W; \ell_1, \ell_2)$ are $\sigma$-finite measures supported on doubly boundary-marked quantum surfaces with left and right boundary arcs having quantum length $\ell_1$ and $\ell_2${, respectively}. See for instance \cite[Definition 2.22 and Proposition 2.23]{AHS20}. 
	We remark that the exact meaning of the identity in \eqref{eq:dis-formula} is that $\mathcal{M}_2^{\text{disk}}(W)\,(S) = \int_0^\infty\int_0^\infty \mathcal{M}_2^{\text{disk}}(W; \ell_1, \ell_2)\,(S)\,d\ell_1\,d\ell_2$ for all measurable sets $S$. 
	The measure $\mathcal{M}_2^{\text{disk}}(W; \ell_1, \ell_2)$ is finite when $W<2+\frac{\gamma^2}{2}$ (see e.g.\ \cite[Lemmas 2.16 and 2.18]{AHS20}); the measure is also finite for certain larger $W$ (e.g.\ $W=4$) but the range $W<2+\frac{\gamma^2}{2}$ is sufficient for us (see the proof of Lemma \ref{prop-density}).

	Using precisely the same argument, we can disintegrate the measure $\mathcal{M}_2^{\text{sph}}(W)$ over the quantum area $A$. In particular, we have 
	\begin{equation}\label{eqn-disintegrate-area}
		\mathcal{M}_2^{\text{sph}}(W) = \int_0^\infty \mathcal{M}_2^{\text{sph}}(W; a)\, da,
	\end{equation}
	where for all $a>0$ the measures $\mathcal{M}_2^{\text{sph}}(W; a)$ are $\sigma$-finite (and finite if and only if $W<4$~\cite{RV10}) supported on doubly marked quantum surfaces with quantum area $A=a$. 
	
	We also remark that if $(D,h,x,y)$ is a sample from $\mathcal{M}_2^{\text{disk}}(W; \ell_1, \ell_2)$ or from $\mathcal{M}_2^{\text{sph}}(W; a)$, then $h$ is a random field on $D$ (more precisely, a random distribution on $D$) and in particular not a random field modulo additive constant.
	
	The next input is a scaling property of quantum disks. Recall our definition of sampling given at the beginning of Section~\ref{sect:perm-lqg-corr}.
	\begin{lemma}[Lemma 2.24 of \cite{AHS20}]\label{lmm-scale-disks}
		Fix $W, \ell_1, \ell_2>0$. The following two random variables sampled as follows have the same law for all $\lambda>0$:
		\begin{enumerate}
			\item Sample a quantum disk from $\mathcal{M}_2^{\textup{disk}}(W; \lambda \ell_1, \lambda\ell_2)$;
			\item  Sample a quantum disk from $\lambda^{-\frac{2W}{\gamma^2}-1}\mathcal{M}_2^{\textup{disk}}(W; \ell_1, \ell_2)$ and add $\frac{2}{\gamma}\log\lambda$ to its field.
		\end{enumerate} 
	\end{lemma}
	Similarly we have the following scaling property for quantum spheres, which can be proved in
	the same manner as \cite[Lemma 2.24]{AHS20}.
	\begin{lemma}\label{lmm-scale-sphere}
		Fix $W, a>0$. The following two random variables sampled as follows have the same law for all $\lambda>0$:
		\begin{enumerate}
			\item Sample a quantum sphere from $\mathcal{M}_2^{\textup{sph}}(W; \lambda a)$;
			\item Sample a quantum sphere from $\lambda^{-\frac{W}{\gamma^2}-1}\mathcal{M}_2^{\textup{sph}}(W; a)$ and add $\frac{1}{\gamma}\log\lambda$ to its field.
		\end{enumerate} 
	\end{lemma} 
	
	We end this subsection with two results on the measures $\mathcal{M}_2^{\textup{disk}}(W; \ell_1, \ell_2)$.
	Before stating them, let us recall the definition of the \textit{Brownian excursion in a cone with non-fixed duration} as constructed in  \cite[Section 3]{LW04}. (We highlight that here we are considering non-fixed time interval excursions. This is a key difference with the Brownian loops introduced in Section~\ref{sect:con_perm_SBP}, where the time interval was fixed and equal to $[0,1]$.)
	Fix an angle $\phi\in(0, 2\pi)$ and let $\mathcal{C}_\phi$ be the cone $\{z\in\mathbb{C}:\arg z\in (0,\phi)\}$. Let $\mathcal{K}$ be the collection of continuous planar curves $\gamma$ {in}
	$\mathcal{C}_\phi$ defined for time $t\in[0, t_\gamma]$, where $t_\gamma$ is the 
	duration of the curve. Then $\mathcal{K}$ can be seen as a metric space with
	\[
	d(\gamma_1, \gamma_2) = \inf_{\beta}\left\{\sup_{0<t<t_{\gamma_1}}|t-\beta(t)|+|\gamma_1(t)-\gamma_2(\beta(t))|\right\},
	\]
	where $\beta$ ranges from all the possible increasing homeomorphisms from $[0,t_{\gamma_1}]$ to $[0,t_{\gamma_2}]$. For $z\in\mathcal{C}_\phi$ and  $r>0$, let $\mu_{\mathcal{C}_\phi}^\#(z, re^{i\phi})$ be the law of the standard planar Brownian motion starting from $z$ and conditioned on exiting $\mathcal{C}_\phi$ at $re^{i\phi}$ (see \cite[Section 3.1.2]{LW04} for further details on this conditioning). This is a Borel probability measure on $\mathcal{K}$, and for all  $\ell, r>0$ {the following limit exists for the Prohorov metric} 
	\begin{equation}\label{eqn-excursion-convergence}
		\lim_{\e\to 0} \mu_{\mathcal{C}_\phi}^\#(\ell+i\e, re^{i\phi}).
	\end{equation}
	We denote the limiting measure by $\mu_{\mathcal{C}_\phi}^\#(\ell, re^{i\phi})$ and call it the law of the Brownian excursion in the cone $\mathcal{C}_\phi$ from $\ell$ to $re^{i\phi}$ with non-fixed duration.
	
	The next result describes the area of a disk sampled from $\mathcal{M}_2^{\textup{disk}}(W; \ell_1, \ell_2)$  when $W=\frac{\gamma^2}{2}$. We remark that this result holds only in the special case $W=\frac{\gamma^2}{2}$. Recall also from the beginning of Section \ref{sect:perm-LQG-cor} that for a finite measure $\nu$ we let $|\nu|$ be its total mass and we let $\nu^\#: = \frac{\nu}{|\nu|}$ denote its normalized version.
	
	\begin{proposition}[Proposition 7.7 of \cite{AHS20}]\label{prop-disk-excursion}
		Fix $\gamma\in (0,2)$ and $\phi = \frac{\pi\gamma^2}{4}$. There exists a constant $c{>0}$ such that for all $\ell_1, \ell_2>0$, 
		\begin{equation}\label{eqn-size-disk-measure}
			\left|\mathcal{M}_2^{\textup{disk}}\left(\frac{\gamma^2}{2}; \ell_1, \ell_2 \right)\right| = c \, \frac{(\ell_1 \ell_2)^{\frac{4}{\gamma^2}-1}}{\left(\ell_1^{\frac{4}{\gamma^2}}+\ell_2^{\frac{4}{\gamma^2}}\right)^{2}}.
		\end{equation}
		Moreover, the quantum area of a sample from $\mathcal{M}_2^{\textup{disk}}(\frac{\gamma^2}{2}; \ell_1, \ell_2)^\#$ has the same law as the duration of a sample from $\mu_{\mathcal{C}_\phi}^\#(\ell_1\sqrt{2\sin\phi}, \ell_2\sqrt{2\sin\phi}e^{i\phi})$.
	\end{proposition}
	
	The next result holds for arbitrary $W\in(0,2+\frac{\gamma^2}{2})$. The upper bound   $2+\frac{\gamma^2}{2}$ guarantees that the measure $\mathcal{M}_2^{\textup{disk}}(W; \ell_1, \ell_2)$ is finite, as explained after \eqref{eq:dis-formula}.
	
	\begin{lemma}\label{prop-density}
		For any $W\in(0,2+\frac{\gamma^2}{2})$ and $\ell_1,\ell_2>0$ the quantum area of a sample from  $\mathcal{M}_2^{\textup{disk}}(W; \ell_1, \ell_2)$ is absolutely continuous with respect to Lebesgue measure.
	\end{lemma}

	\begin{proof}
		For $W=\gamma^2/2$ the result follows from Proposition~\ref{prop-disk-excursion}. For $W\in (\gamma^2/2,2+\frac{\gamma^2}{2})$ we get the result from \cite[Theorem 2.2]{AHS20}\footnote{The inexperienced reader might consider to skip this part of the proof at a first read and come back to this proof after reading Section~\ref{sec:qs-conf-welding} where a counterpart of \cite[Theorem 2.2]{AHS20} for quantum spheres (i.e.\ Theorem~\ref{thm-sph-welding} below) will be presented in detail.} with $W_1=\gamma^2/2$ and $W_2=W-\gamma^2/2$,  
		along with the fact that the sum of two independent random variables has a density function if at least one of the summands has a density. Note that the statement of \cite[Theorem 2.2]{AHS20} involves the measures $\mathcal{M}_2^{\textup{disk}}(W; \ell_1, \ell_2)$ and we are using the fact that these measures are finite when $W<2+\frac{\gamma^2}{2}$ as explained after \eqref{eq:dis-formula}. 
		
		Finally, we get the result for $W\in(0,\gamma^2/2)$ by using that we know the lemma for thick quantum disks with weights in $(\gamma^2,2+\gamma^2/2)$ and that a thin quantum disk of weight $W\in(0,\gamma^2/2)$ can be described as an ordered collection of doubly-marked thick quantum disks of weight $\gamma^2-W\in (\gamma^2/2,\gamma^2)$, as done in Definition~\ref{def-thin-disk}.
	\end{proof}
	
	\subsubsection{Conformal welding of quantum disks}
	\label{sec:qs-conf-welding}
	
	In this section we review one of the main results of \cite{AHS20}, which is stated as Theorem~\ref{thm-sph-welding} below and will be a key input in the proof of Theorem~\ref{thm-baxter-density}. We first give the formal statement of the theorem and then we explain the interpretation of the theorem as a conformal welding of quantum surfaces. 
	
	\begin{figure}[ht]
		\centering
		\includegraphics[scale=0.575]{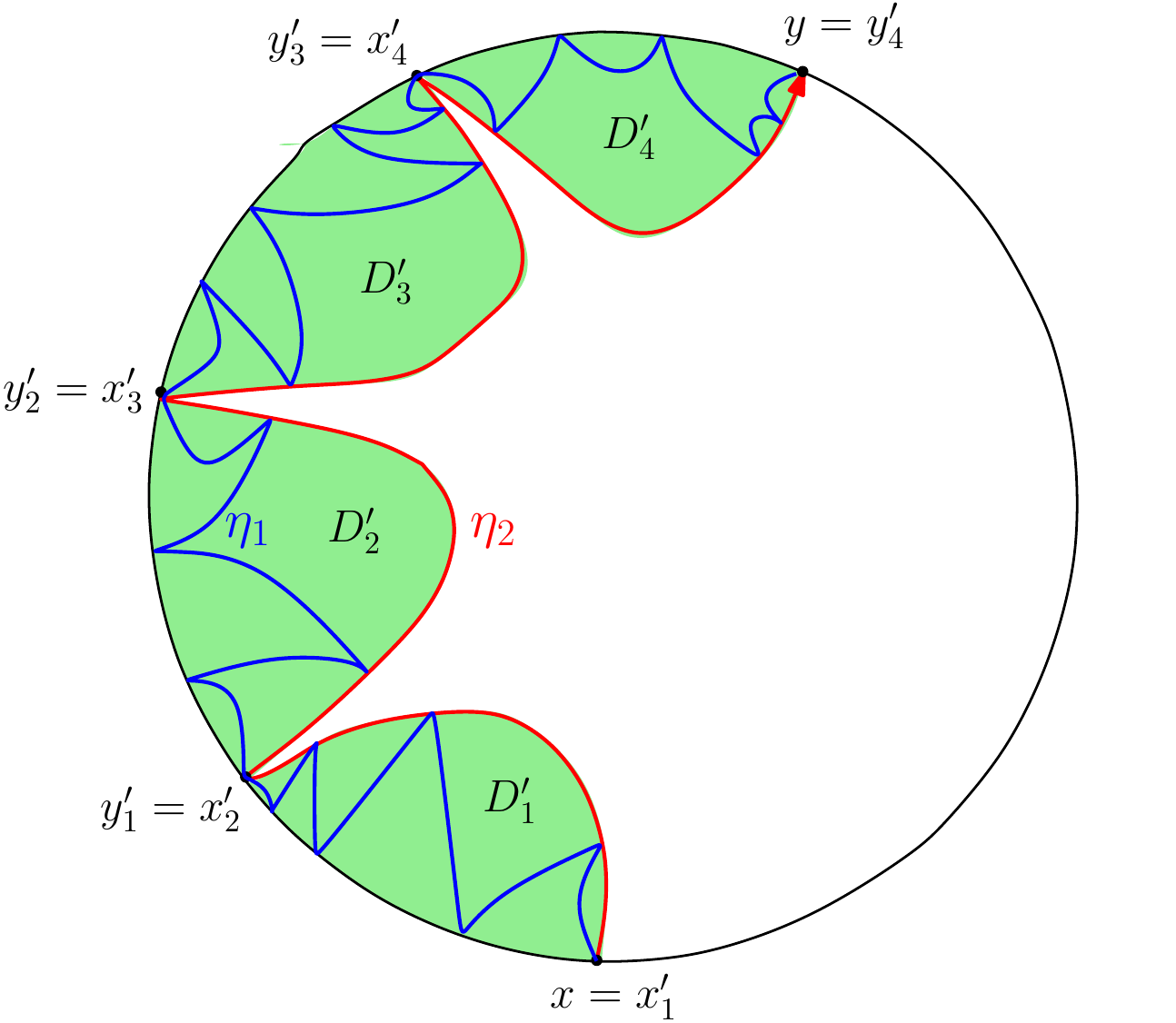} 
		\caption{An illustration of the iterative construction of the measure $\mathcal{P}^{\text{disk}}(W_1, ..., W_n)$ for $n=3$. We consider the case when $D$ is a disk with two marked boundary points $x$ and $y$ as plotted in the picture, $W_1+W_{2}<\frac{\gamma^2}{2}$, and $W_3\geq \frac{\gamma^2}{2}$. We first sample a (red) curve $\eta_{2}$ from SLE$_\kappa(W_1+W_2-2;W_3-2)$ on $(D, x, y)$ which hits the left boundary infinitely many times because $W_1+W_{2}<\frac{\gamma^2}{2}$ and does not hit the right boundary because $W_3\geq \frac{\gamma^2}{2}$.
			In each (green) domain $D'_j$ we sample an independent collection of curves from $\mathcal{P}^{\text{disk}}(W_1, ...,W_{n-1})$ starting at $x'_j$ and ending at $y'_{j}$. In our specific case when $n=3$, we sample a (blue) curve from SLE$_\kappa(W_1-2; W_2-2)$ in each (green) domain $D'_j$ starting at $x'_j$ and ending at $y'_{j}$, and then we consider the (blue) curve $\eta_1$ obtained as the concatenation of these (blue) curves.
			We note that our figure is simplified since there are actually (countably) infinitely many domains $D'_j$ cut out by the (red) $\mathrm{SLE}$ curve $\eta_{2}$. The law of the two curves $(\eta_1,\eta_2)$ is $\mathcal{P}^{\text{disk}}(W_1,W_2,W_3)$. \label{fig-conformal_welding}}
	\end{figure}

	Recall that $\mathrm{SLE}_\kappa(\rho_1;\rho_2)$ and whole-plane $\mathrm{SLE}_\kappa(\rho)$ were introduced in Section~\ref{sect:sle_ig}. In particular, recall that $\mathrm{SLE}_\kappa(\rho_1;\rho_2)$ from $a\in\partial D$ to $b\in\partial D$ in a domain $D\subset \mathbb C$ hit (countably) infinitely many times the left (resp.\ right) boundary if and only if $\rho_1<\frac\kappa 2-2$ (resp.\ $\rho_2<\frac\kappa 2-2$), and whole-plane SLE$_{\kappa}(\rho)$ curves from $0$ to $\infty$ hit themselves (countably) infinitely many times if and only if $\rho<\frac\kappa 2-2$.

	Fix $n\ge 2$, $W_1, .., W_n>0$, $\kappa=\gamma^2\in (0,4)$ and let $W=W_1+...+W_n$. Let $(D,x,y)$ be a proper simply connected domain contained in $\mathbb C$ with two points $x$ and $y$ lying on the boundary of $D$. We inductively define some probability measures $\mathcal{P}^{\text{disk}}(W_1, ..., W_{n})$ on non-crossing curves $(\eta_1,..,\eta_{n-1})$ in $D$ joining $x$ and $y$ for all $n\ge 2$. When $n=2$, define the measure $\mathcal{P}^{\text{disk}}(W_1, W_2)$ to be an SLE$_\kappa(W_1-2; W_2-2)$ in $(D, x, y)$; when $n\ge 3$, the measure $\mathcal{P}^{\text{disk}}(W_1, ..., W_n)$ on non-crossing curves $(\eta_1, ... ,\eta_{n-1})$ is defined recursively by first sampling $\eta_{n-1}$ from SLE$_\kappa(W_1+...+W_{n-1}-2;W_n-2)$ on $(D, x, y)$ and then  the tuple $(\eta_1, ... ,\eta_{n-2})$ as concatenation of samples from $\mathcal{P}^{\text{disk}}(W_1, ..., W_{n-1})$ in each connected component $(D'_i, x'_i, y'_i)$  of $D\backslash \eta_{n-1}$ lying to the left of $\eta_{n-1}$ (where $x'_i$ and $y'_i$ are the first and the last point on the boundary $\partial D'_i$ visited by $\eta_{n-1}$; see also Fig.~\ref{fig-conformal_welding}). We remark that when $W_1+...+W_{n-1}<\frac{\gamma^2}{2}$ there are (countably) infinitely many connected components $(D'_i, x'_i, y'_i)$, while when $W_1+...+W_{n-1}\geq \frac{\gamma^2}{2}$ there is only one component $(D', x', y')$.
	
	Note that using conformal invariance of SLE, the definition above can be extended to all proper simply connected domains $D$ of $\mathbb C$ with two boundary points $x$ and $y$.
	
	We also analogously define the probability measure $\mathcal{P}^{\text{sph}}(W_1, ..., W_{n})$ on  $n$-tuple of curves $(\eta_0, ..., \eta_{n-1})$ in $\wh{\mathbb{C}}$ from $0$ to $\infty$ as follows. First sample a whole-plane SLE$_{\kappa}(W_1+...+W_n-2)$ curve $\eta_0$ from 0 to $\infty$ in $\wh{\mathbb{C}}$ and then the tuple $(\eta_1, ... ,\eta_{n-1})$ as concatenation of samples from $\mathcal{P}^{\text{disk}}(W_1, ..., W_n)$ in each connected component of $\wh{\mathbb{C}}\backslash \eta_0$. We remark that when when $W_1+...+W_n<\frac{\gamma^2}{2}$ there are (countably) infinitely many connected components, while when $W_1+...+W_{n}\geq \frac{\gamma^2}{2}$ there is only one component.
	
	Given a measure $\mathcal{M}$ on quantum surfaces with $k$ marked points and a conformally invariant measure $\mathcal{P}$ on curves, let $\mathcal{M}\otimes\mathcal{P}$ be the measure on the curve-decorated surfaces with $k$ marked points constructed by first sampling a surface $(D, \psi, z_1, ..., z_k)$ from $\mathcal{M}$ and then drawing independent curves on $D$ sampled from the measure $\mathcal{P}$.  
	Note that we require that the measure $\mathcal{P}$ on curves is conformally invariant (which is satisfied in the above case of $\mathrm{SLE}_\kappa${-}type curves) so that this notation is compatible with the coordinate change \eqref{eqn-lqg-changecoord}. Sometimes the curves are required to start and/or end at given marked points of the surface; this will either be stated explicitly or be clear from the context.
	
	Now we are ready to state one of the main results of \cite{AHS20}. 
	
	\begin{theorem}[Theorem 2.4 of \cite{AHS20}]\label{thm-sph-welding}
		Fix $n\ge 1$ and $W_1, ..., W_n>0$. Let $W = W_1+...+W_n$. Then there exists a constant $c\in (0,\infty)$ depending only on $\kappa=\gamma^2{\in(0,4)}$ and $W_1, ..., W_n$, such that  
		\begin{multline*}
			\mathcal{M}_2^{\textup{sph}}(W)\otimes \mathcal{P}^{\textup{sph}}(W_1, ..., W_n) \\
			= c\int_{\bbR_{+}^{n}}\mathcal{M}_2^{\textup{disk}}(W_1;\ell_0, \ell_1)\times\mathcal{M}_2^{\textup{disk}}(W_2;\ell_1, \ell_2)\times\cdots\times\mathcal{M}_2^{\textup{disk}}(W_n;\ell_{n-1}, \ell_0)\,d\ell_0\,...\,d\ell_{n-1}.
		\end{multline*}	
	\end{theorem}	
	
	We refer to this type of results as \emph{conformal welding of quantum surfaces}.
	We now give a more informal interpretation of the above result in order to help the reader to develop some intuition on the statement of Theorem~\ref{thm-sph-welding}.  The right-hand side of the indented equation in the theorem can be interpreted as the ``conformal welding'' of the $n$ quantum disks sampled from the measures $\mathcal{M}_2^{\textup{disk}}(W_j;\ell_{j-1}, \ell_j)$ into a quantum sphere with law $\mathcal{M}_2^{\textup{sph}}(W)$ decorated with $n$ SLE$_\kappa$-type curves with joint law $\mathcal{P}^{\textup{sph}}(W_1, ..., W_n)$. More precisely, one can first conformally weld the first pair of quantum disks sampled from $\mathcal{M}_2^{\textup{disk}}(W_1;\ell_0, \ell_1)\times\mathcal{M}_2^{\textup{disk}}(W_2;\ell_1, \ell_2)$ along their length $\ell_1$ boundary arcs, yielding a new quantum disk with two marked boundary points, a SLE$_\kappa$-type curve joining them, and two boundary arcs of quantum lengths $\ell_0$ and $\ell_2$. Then one can iterate this procedure by conformally welding this new curve-decorated quantum disk with the next quantum disks sampled from $\mathcal{M}_2^{\textup{disk}}(W_j;\ell_{j-1}, \ell_j)$ for all $j=3,\dots,n$ ($\ell_n=\ell_0$), obtaining in the end another quantum disk with two marked boundary points, $n-1$ SLE$_\kappa$-type curves joining them, and two boundary arcs of equal quantum lengths $\ell_0$. Welding together the left and the right boundary of this final quantum disk, yield to a quantum sphere decorated by $n$ SLE$_\kappa$-type curves. Theorem~\ref{thm-sph-welding} states that the law of this curve-decorated quantum-sphere is $\mathcal{M}_2^{\textup{sph}}(W)\otimes \mathcal{P}^{\textup{sph}}(W_1, ..., W_n)$. We refer the curious reader to the original paper \cite{AHS20} for further details.
	
	\subsection{Rerooting invariance of quantum spheres and its consequences on the skew Brownian permuton}\label{sect:res_prop}
	
	In this section we review the rerooting invariance of the marked points on a unit-area quantum sphere and give an alternative expression for the intensity measure $\mathbb{E}	[\mu_{\rho,q}]$ of the skew Brownian permuton.
	The following result is \cite[Proposition A.13]{DMS14} and  is the base point of our arguments.
	
	\begin{proposition}[Rerooting invariance of quantum spheres]\label{prop-qs-resampling}
		Let $\gamma\in(0,2)$. Suppose $(\wh{\mathbb{C}}, h, 0, \infty)$ is a unit-area quantum sphere of weight $4-\gamma^2$. Then conditional on the surface $(\wh{\mathbb{C}}, h)$, the points which corresponds to $0$ and $\infty$ are distributed independently and uniformly from the quantum area measure $\mu_h$. That is, if $x,y$ in $\wh{\mathbb{C}}$ are chosen uniformly from $\mu_h$ and $\varphi:\wh{\mathbb{C}}\to \wh{\mathbb{C}}$ is a conformal map with $\varphi(x) = 0$ and $\varphi(y)=  \infty$, then $(\wh{\mathbb{C}},h\circ\varphi^{-1}+Q\log|(\varphi^{-1})'|,0,\infty)$ has the same law as $(\wh{\mathbb{C}},h,0,\infty)$ when viewed  as doubly marked quantum surfaces.
	\end{proposition} 
	In particular, if we condition on $y = \infty$ in the statement of Proposition~\ref{prop-qs-resampling} and resample $x$ according to the quantum area measure $\mu_h$, then the quantum surface $(\wh{\mathbb{C}}, h, x, \infty)$ has the same law as $(\wh{\mathbb{C}}, h, 0, \infty)$. 
	
	Before proving the main result of this section, we introduce some more notation. Let $\wh{h}$ be a whole-plane GFF (viewed modulo a global additive integer multiple of $2\pi \chi$). For $w\in\wh{\mathbb{C}}$, we denote by $\eta_{\op{E}}^{w},\eta_{\theta}^{w},\eta_{\op{W}}^{w}, \eta_{\theta+{\pi}}^{w}$ the flow lines of $\wh{h}$ issued from $w$ with corresponding angles $-\frac{\pi}{2},\theta,\frac{\pi}2, \theta+\pi$ (defined at the end of Section \ref{sect:sle_ig}). 
	Recall that from \cite[Theorem 1.7]{MS17} flow lines from the same point with different angles might bounce off each other but can never cross or merge.
	We denote by $A_1^{w}, A_2^{w}, A_3^{w}, A_4^{w}$ the areas of the four regions cut out by the four flow lines $\eta_{\op{E}}^{w},\eta_{\theta}^{w},\eta_{\op{W}}^{w}, \eta_{\theta+{\pi}}^{w}$, labeled as in Figure~\ref{fig-flowlines}. When $w=0$, we simply write $\eta_{\op{E}},\eta_{\theta},\eta_{\op{W}}, \eta_{\theta+{\pi}}$ for $\eta_{\op{E}}^{0},\eta_{\theta}^{0},\eta_{\op{W}}^{0}, \eta_{\theta+{\pi}}^{0}$ and $A_1, A_2, A_3, A_4$ for $A_1^{0}, A_2^{0}, A_3^{0}, A_4^{0}$. In this case, it can be argued using the imaginary geometry coupling in 
	\cite[Theorem 1.1]{MS16a} and \cite[Theorem 1.1]{MS17} {that} the joint law of the four flow lines $\eta_E,\eta_{\theta},\eta_W, \eta_{\theta+{\pi}}$ can be viewed as $\mathcal{P}^{\text{sph}}(W_1, W_2, W_3, W_4)$ with $(W_1,W_2,W_3,W_4)$ determined by 	
	\begin{equation}\label{eqn-weight}
		W_1 = W_3 =  2 - \frac{\gamma^2}{2}-\frac{4-\gamma^2}{2\pi}(\theta+\pi/2);\qquad W_2 = W_4 =\frac{4-\gamma^2}{2\pi}(\theta+\pi/2).
	\end{equation} 
	See~\cite[Tables 1.1 and 1.2]{DMS14} for the complete correspondence between imaginary geometry angles and quantum surface weights.
	
	\begin{figure}[ht]
		\centering
		\includegraphics[scale=0.33]{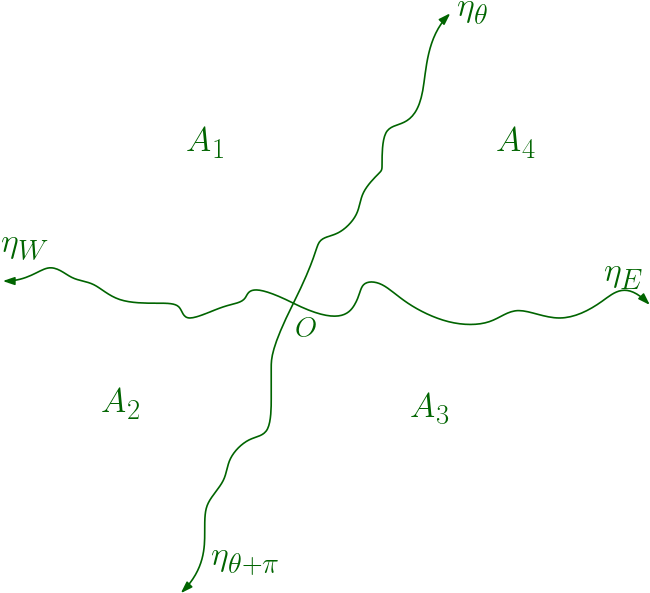}
		\caption{The flow lines $\eta_{\op{E}},\eta_{\theta},\eta_{\op{W}}, \eta_{\theta+{\pi}}$  of $\wh{h}$ with corresponding angles $-\frac{\pi}{2}$, $\theta$, $\frac{\pi}{2}$, $\theta+\pi$ issued from 0 for $\theta\in(-\frac{\pi}{2},\frac{\pi}{2})$.
			They cut the quantum sphere $(\wh{\mathbb{C}}, h, 0, \infty)$ into four quantum disks with areas $A_1, A_2, A_3, A_4$ as labeled. These four quantum disks (which can be either thin or thick quantum disks depending on the values of the parameters $\gamma$ and $\theta$) are independent conditioned on having the same boundary arc quantum lengths (from the welding) and total area 1, i.e.\ $A_1+A_2+A_3+A_4=1$; see Section~\ref{sec:density-a} for further details. We also highlight that the counterflow line $\eta'_0$ first visits the regions with area $A_2$ and $A_3$ and then the regions with area $A_1$ and $A_4$, while  the counterflow line $\eta'_{\theta-\frac\pi 2}$ first visits the regions with area $A_1$ and $A_2$ and then the regions with area $A_3$ and $A_4$. Moreover, the flow lines $\eta_{\op{W}}$ and $ \eta_{\op{E}}$ (resp.\ $\eta_{\theta}$ and $\eta_{\theta+{\pi}}$) are a.s.\ left and right boundaries of $\eta'$ (resp.\ $\eta'_{\theta-\frac\pi 2}$) stopped upon hitting $0$, as explained in Section~\ref{sect:sle_ig}.\label{fig-flowlines}}
	\end{figure} 
	
	\begin{proposition}\label{prop-baxter-disk}
		Let $\gamma\in(0,2)$. Let $(\wh{\mathbb{C}}, h, 0, \infty)$ be a unit-area quantum sphere of weight $4-\gamma^2$, and let $\theta\in [-\frac \pi 2,\frac \pi 2]$. Let $\wh{h}$ be a whole-plane GFF (viewed modulo a global additive integer multiple of $2\pi \chi$) independent of $h$ and consider the corresponding four areas $A_1,A_2,A_3,A_4$ defined above (see also Figure~\ref{fig-flowlines}). Set $\rho\in(-1,1)$ and $q\in[0,1]$ such that $\rho=-\cos(\pi\gamma^2/4)$ and $q = q_\gamma(\theta)$ and consider the skew Brownian permuton $\mu_{\rho,q}$. Then for all rectangles $[x_1, x_2]\times [y_1, y_2]\subset [0,1]^2$,
		\begin{multline*}
			\bbE [\mu_{\rho,q}]\Big([x_1, x_2]\times [y_1, y_2]\Big)= \\
			\mathcal{M}_2^{\textup{sph}}(4-\gamma^2; 1)^\#\otimes \mathcal{P}^{\textup{sph}}(W_1, W_2, W_3, W_4) \Big(A_2+A_3\in [x_1, x_2], A_1+A_2\in [y_1, y_2]\Big),
		\end{multline*}
		where $W_1, W_2, W_3, W_4$ are given in \eqref{eqn-weight}.
	\end{proposition}
	
	\begin{proof}
		Given the unit-area quantum sphere $(\wh{\mathbb{C}}, h, 0, \infty)$, we uniformly sample a point $\mathbf{w}$ according to the $\gamma$-LQG area measure $\mu_h$. Consider the flow lines $\eta_{\op{E}}^{\mathbf{w}}, \eta_{\theta}^{\mathbf{w}}, \eta_{\op{W}}^{\mathbf{w}}, \eta_{\theta+\pi}^{\mathbf{w}}$ of the whole-plane GFF $\wh{h}$ starting from $\mathbf{w}$ and going to infinity. Also assume that the skew Brownian permuton $\mu_{\rho,q}$ is coupled with $(\wh{\mathbb{C}}, h, 0, \infty)$ and  $\wh{h}$ under the same probability measure $\mathbb{P}$ as in Proposition~\ref{prop:permuton-q-area}. On the one hand, by Proposition~\ref{prop:permuton-q-area}, $\bbE \mu_{\rho,q}([x_1, x_2]\times [y_1, y_2])$ is the probability of $\mathbf{w}$ falling into the random set $\eta'([x_1, x_2])\cap \eta'_{\theta-\frac\pi 2}([y_1, y_2])$. On the other hand,  $\mathbf{w}$ is a.s.\ not a double point for neither $\eta'$ nor $\eta'_{\theta-\frac\pi 2}$; and by the definition of space-filling SLE curves given at the end of Section \ref{sect:sle_ig}, the flow lines $\eta_{\op{W}}^{\mathbf{w}}$ and $ \eta_{\op{E}}^{\mathbf{w}}$ (resp.\ $\eta_{\theta}^{\mathbf{w}}$ and $\eta_{\theta+{\pi}}^{\mathbf{w}}$) are a.s.\ left and right boundaries of $\eta'$ (resp.\ $\eta'_{\theta-\frac\pi 2}$) stopped upon hitting $\mathbf{w}$. From this and the fact that we are parametrizing the curves $\eta'$ and $\eta'_{\theta-\frac\pi 2}$ using $\mu_h$, we see that a.s.\ $\mathbf{w}$ falls into $\eta'([x_1, x_2])\cap \eta'_{\theta-\frac\pi 2}([y_1, y_2])$ if and only if   $A_1^{\mathbf{w}}+A_2^{\mathbf{w}}\in [y_1, y_2]$ and $A_2^{\mathbf{w}}+A_3^{\mathbf{w}}\in [x_1, x_2]$, which implies that
		\[
		\bbE [\mu_{\rho,q}]\Big([x_1, x_2]\times [y_1, y_2]\Big) = {\bbP} \Big(A_2^{\mathbf{w}}+A_3^{\mathbf{w}}\in [x_1, x_2], A_1^{\mathbf{w}}+A_2^{\mathbf{w}}\in [y_1, y_2]\Big).
		\]
		Now we treat $\mathbf{w}$ and $\infty$ as the two marked points of the quantum sphere, and consider the shift $z\mapsto z-\mathbf{w}$. Let  $(\wh{\mathbb{C}}, h^\mathbf{w}, 0, \infty)$ be the corresponding doubly marked surface, where $h^{\mathbf{w}} = h(\cdot+\mathbf{w})$. We also set $\wh{h}^\mathbf{w} := \wh{h}(\cdot+\mathbf{w})$. It is clear that given $\mathbf{w}$ and the quantum sphere  $(\widehat{{\mathbb C}}, h, 0, \infty)$, the field $\wh{h}^\mathbf{w}$ has the law as a whole-plane GFF (modulo a global additive integer multiple of $2\pi\chi$),  
		and the four flow lines $\eta_{\op{E}}^{\mathbf{w}}, \eta_{\theta}^{\mathbf{w}}, \eta_{\op{W}}^{\mathbf{w}}, \eta_{\theta+\pi}^{\mathbf{w}}$ are mapped by the shift $z\mapsto z-\mathbf{w}$ to corresponding flow lines of $\wh{h}^\mathbf{w}$ issued from 0. Moreover, by the rerooting invariance stated in Proposition~\ref{prop-qs-resampling}, $(\wh{\mathbb{C}}, h^\mathbf{w}, 0, \infty)$ has the same law as the unit-area quantum sphere $(\wh{\mathbb{C}}, h, 0, \infty)$ and it is independent of the whole-plane GFF $\wh{h}^\mathbf{w}$. Since, as discussed above, the joint law of the four flow lines is $\mathcal{P}^{\text{sph}}(W_1, W_2, W_3, W_4)$ where $W_1, W_2, W_3, W_4$ are given in \eqref{eqn-weight}, and the law of a unit-area quantum sphere is $\mathcal{M}_2^{\textup{sph}}(4-\gamma^2; 1)^\#$ by Definition~\ref{def-quantum-sphere} and \eqref{eqn-disintegrate-area}, it follows that 
		\begin{multline*}
			{\bbP} \Big(A_2^{\mathbf{w}}+A_3^{\mathbf{w}}\in [x_1, x_2], A_1^{\mathbf{w}}+A_2^{\mathbf{w}}\in [y_1, y_2]\Big) \\
			= \mathcal{M}_2^{\textup{sph}}(4-\gamma^2; 1)^\#\otimes \mathcal{P}^{\textup{sph}}(W_1, W_2, W_3, W_4) \left(A_2+A_3\in [x_1, x_2], A_1+A_2\in [y_1, y_2]\right),
		\end{multline*}
		which justifies the proposition.
	\end{proof}
	
	\subsection{Density of the Baxter permuton}\label{sec:density-a}
	In this section we conclude the proof of Theorem~\ref{thm-baxter-density}. First we derive in Section~\ref{sec:density-skew-Bp} a formula for the density of the skew Brownian permuton which holds for all $\rho\in(-1,1)$ and $q\in(0,1)$ (Theorem \ref{thm-baxter-density-general}), and in Section~\ref{sec:density-Baxter} we simplify this formula in the special case of the Baxter permuton. Finally, in Section~\ref{sect:rel_tetra} we sketch how the formula can be made yet more explicit for the Baxter permuton via known formulas for the volume of spherical tetrahedra.
	
	\subsubsection{Density of the skew Brownian permuton in terms of quantum disks}
	\label{sec:density-skew-Bp}
	Recall from Lemma~\ref{prop-density} that for any $W\in(0,2+\frac{\gamma^2}{2})$ and $\ell_1,\ell_2>0$ the quantum area $A$ of a sample from  $\mathcal{M}_2^{\textup{disk}}(W; \ell_1, \ell_2)$ is absolutely continuous with respect to Lebesgue measure. 
	Let $p_W(a, \ell_1, \ell_2)$ denote the density of $A$, that is, for any non-negative measurable function $g$, we have
	\begin{equation}\label{eqn-def-area-density}
		\int g(\mu_{\psi}(D))\,d\mathcal{M}_2^{\textup{disk}}(W; \ell_1, \ell_2) = \int_0^\infty g(a)p_W(a, \ell_1, \ell_2)\,da,
	\end{equation}
	where $(D,\psi,x,y)$ is an embedding of a sample from $\mathcal{M}_2^{\textup{disk}}(W; \ell_1, \ell_2)$ (recall the definition of embedding from Section~\ref{sect:quan_surf}).
	The aim of this section is to prove the following.
	\begin{theorem}\label{thm-baxter-density-general}
		Consider the skew Brownian permuton $\mu_{\rho,q}$ of parameters $\rho\in(-1,1)$ and $q\in(0,1)$. Let $\gamma\in(0,2)$ and $\theta\in[-\frac \pi 2,\frac \pi 2]$ be defined by $\rho=-\cos(\pi\gamma^2/4)$ and $\theta = \theta_\gamma(q)$. Set $(W_1,W_2,W_3,W_4)$ as in \eqref{eqn-weight} and denote by $p_i(a; \ell_1, \ell_2):=p_{W_i}(a; \ell_1, \ell_2)$ the density of the quantum area of a sample from  $\mathcal{M}_2^{\textup{disk}}(W_i; \ell_1, \ell_2)$ in the sense of \eqref{eqn-def-area-density}.
		Then the intensity measure $\bbE [\mu_{\rho,q}]$ is absolutely continuous with respect to the Lebesgue measure on $[0,1]^2$ and has the following density function 
		\begin{equation*}
			{(x,y)\mapsto}c\int_{\max\{0, x+y-1\}}^{\min\{x, y\}}\int_{\bbR_{+}^4}p_1(y-z, \ell_1, \ell_2)p_2(z, \ell_2, \ell_3)p_3(x-z, \ell_3, \ell_4)p_4(1+z-x-y, \ell_4, \ell_1)\,d\ell_1d\ell_2d\ell_3d\ell_4\,dz,
		\end{equation*}
		where $c$ is a normalizing constant. 
	\end{theorem}
	
	We start the proof by recalling that  the joint law of the four flow lines $\eta_{\op{E}},\eta_{\theta},\eta_{\op{W}}, \eta_{\theta+{\pi}}$ can be viewed as $\mathcal{P}^{\text{sph}}(W_1, W_2, W_3, W_4)$ with $(W_1,W_2,W_3,W_4)$ determined by \eqref{eqn-weight}. Then, in order to prove Theorem~\ref{thm-baxter-density-general}, we first use the scaling property of quantum disks and quantum spheres to remove the conditioning on having total quantum area one (see Proposition~\ref{prop-baxter-disk-1}), and then we conclude the proof by applying Theorem~\ref{thm-sph-welding}.
	
	\begin{proposition}\label{prop-baxter-disk-1}
		Let $\gamma\in(0,2)$. Let $(\wh{\mathbb{C}}, h, 0, \infty)$ be a quantum sphere of weight $4-\gamma^2$ (here we do not condition on the area of the quantum sphere to be 1), and let $\theta\in [-\frac \pi 2,\frac \pi 2]$. Let $\wh{h}$ be a whole-plane GFF (viewed modulo a global additive integer multiple of $2\pi \chi$) independent of $h$ and consider the corresponding four areas $A_1,A_2,A_3,A_4$ defined above (see also Figure~\ref{fig-flowlines}). Set $\rho\in(-1,1)$ and $q\in[0,1]$ such that $\rho=-\cos(\pi\gamma^2/4)$ and $q = q_\gamma(\theta)$ and consider the skew Brownian permuton $\mu_{\rho,q}$.
		
		Let $f$ be a non-zero function on $[0,\infty)$ with $\int_0^\infty |f(t)|t^{-\frac{4}{\gamma^2}}\,dt<\infty$. There exists a universal constant $c$ depending only on $\gamma$, $\theta$ and $f$ (and so only on $\rho$, $q$ and $f$), such that for all $0\le x_1\le x_2\le 1$ and $0\le y_1\le y_2\le 1$, it holds that
		\begin{equation}\label{eqn-density}
			\bbE \mu_{\rho,q}([x_1, x_2]\times [y_1, y_2]) = c\int f(A)\mathds{1}_{\left\{\frac{A_1+A_2}{A}\in [y_1, y_2], \frac{A_2+A_3}{A}\in [x_1, x_2]\right\}}\,d\mathcal{M}_2^{\textup{sph}}(4-\gamma^2)\otimes \mathcal{P}^{\textup{sph}}(W_1, W_2, W_3, W_4),
		\end{equation}
		where $A$ denotes the area of a quantum sphere sampled from $\mathcal{M}_2^{\textup{sph}}(4-\gamma^2)$, and the weights $W_1, W_2, W_3, W_4$ are given by \eqref{eqn-weight}.
	\end{proposition}
	
	\begin{remark}
		We remark that the function $f$ purely serves as a test function and scaling factor, which shall be eliminated later once we apply the scaling property of quantum disks. The condition $\int_0^\infty |f(t)|t^{-\frac{4}{\gamma^2}}\,dt<\infty$ is made to assure that the integral on the right hand side of \eqref{eqn-density} is finite. 
	\end{remark}
	
	\begin{proof}[Proof of Proposition~\ref{prop-baxter-disk-1}]
		We disintegrate the right-hand side of \eqref{eqn-density} in terms of quantum area. By Lemma~\ref{lmm-scale-sphere}, we have the following relation for any fixed $a>0$ 
		\begin{multline}\label{eqn-density-2}
			a^{-\frac{4}{\gamma^2}}\int \mathds{1}_{\left\{A_1+A_2\in [y_1, y_2], A_2+A_3\in [x_1, x_2]\right\}}\,d\mathcal{M}_2^{\textup{sph}}( 4-\gamma^2; 1)\otimes \mathcal{P}^{\textup{sph}}(W_1, W_2, W_3, W_4)\\
			\quad=\int \mathds{1}_{\left\{\frac{A_1+A_2}{a}\in [y_1, y_2], \frac{A_2+A_3}{a}\in [x_1, x_2]\right\}}\,d\mathcal{M}_2^{\textup{sph}}( 4-\gamma^2; a)\otimes \mathcal{P}^{\textup{sph}}(W_1, W_2, W_3, W_4).
		\end{multline}
		Recall that $A$ denotes the area of the quantum sphere sampled from $\mathcal{M}_2^{\textup{sph}}(4-\gamma^2)$. By multiplying both sides of \eqref{eqn-density-2} by $f(a)$ and integrate over $a\in (0, \infty)$, we get
		\begin{multline*}
			\left(\int_0^\infty f(a)a^{-\frac{4}{\gamma^2}}\,da\right)\left(\int \mathds{1}_{\left\{A_1+A_2\in [y_1, y_2], A_2+A_3\in [x_1, x_2]\right\}}\,d\mathcal{M}_2^{\textup{sph}}(4-\gamma^2; 1)\otimes \mathcal{P}^{\textup{sph}}(W_1, W_2, W_3, W_4)\right)\\
			=\int_0^\infty\int f(a)\mathds{1}_{\left\{\frac{A_1+A_2}{a}\in [y_1, y_2], \frac{A_2+A_3}{a}\in [x_1, x_2]\right\}}\,d\mathcal{M}_2^{\textup{sph}}(4-\gamma^2; a)\otimes \mathcal{P}^{\textup{sph}}(W_1, W_2, W_3, W_4)\,da\\
			= \int f(A)\mathds{1}_{\left\{\frac{A_1+A_2}{A}\in [y_1, y_2], \frac{A_2+A_3}{A}\in [x_1, x_2]\right\}}\,d\mathcal{M}_2^{\textup{sph}}(4-\gamma^2)\otimes \mathcal{P}^{\textup{sph}}(W_1, W_2, W_3, W_4).
		\end{multline*}
		where on the last equality we used the disintegration  formula \eqref{eqn-disintegrate-area} and the fact that a sample from $\mathcal{M}_2^{\textup{sph}}(4-\gamma^2; a)$ has quantum area $a$. 
		
		The conclusion follows from Proposition~\ref{prop-baxter-disk} with $c = \left(\left|\mathcal{M}_2^{\textup{sph}}(4-\gamma^2; 1)\right|\int_0^\infty f(a)a^{-\frac{4}{\gamma^2}}\,da\right)^{-1}.$ 
	\end{proof}
	
	We can now apply the conformal welding result stated in Theorem~\ref{thm-sph-welding} to the right-hand side of \eqref{eqn-density}. To simplify the expressions, we first need the following scaling property of quantum disks.
	
	\begin{lemma}\label{lmm-area-scaling}
		For any $\lambda>0$, the density $p_W(a, \ell_1, \ell_2)$ defined in \eqref{eqn-def-area-density} satisfies the scaling property
		\[
		p_W(\lambda^2a, \lambda \ell_1, \lambda \ell_2) = \lambda^{-\frac{2}{\gamma^2}W-3}p_W(a, \ell_1, \ell_2).
		\]
	\end{lemma}
	
	\begin{proof}
		The lemma is an easy consequence of Lemma~\ref{lmm-scale-disks}, from which we know that
		\begin{equation*}
			\int g(\mu_{\psi}(D))\,d\mathcal{M}_2^{\textup{disk}}(W; \lambda \ell_1, \lambda \ell_2) = \lambda^{-\frac{2}{\gamma^2}W-1}\int g(\mu_{\psi+\frac{2}{\gamma}\log\lambda}(D))\,d\mathcal{M}_2^{\textup{disk}}(W; \ell_1, \ell_2),
		\end{equation*}
		for any non-negative measurable function $g$, where both surfaces in the above equation are embedded in the planar domain $D$.  Then from the definition given in \eqref{eqn-def-area-density} we have that
		\begin{equation*}
			\begin{split}
				\int_0^\infty g(a)p_W(a, \lambda \ell_1, \lambda \ell_2)\,da &= \lambda^{-\frac{2}{\gamma^2}W-1}\int_0^\infty g(\lambda^2 a)p_W(a, \ell_1, \ell_2)\,da \\
				&= \lambda^{-\frac{2}{\gamma^2}W-3}\int_0^\infty g( a)p_W(\lambda^{-2}a, \ell_1, \ell_2)\,da
			\end{split}
		\end{equation*} 
		and the conclusion readily follows.
	\end{proof}
	
	We now complete the proof of Theorem~\ref{thm-baxter-density-general}.
	
	\begin{proof}[Proof of Theorem~\ref{thm-baxter-density-general}]
		Recall that $p_i(a; \ell_1, \ell_2)=p_{W_i}(a; \ell_1, \ell_2)$. By Theorem~\ref{thm-sph-welding} and the definition given in \eqref{eqn-def-area-density}, we can write the right-hand side of \eqref{eqn-density} as
		\begin{equation}\label{eqn-density-3}
			c\int_{\bbR_{+}^8}f(a)\mathds{1}_{\left\{\frac{a_1+a_2}{a}\in [y_1, y_2], \frac{a_2+a_3}{a}\in [x_1, x_2]\right\}}p_1(a_1, \ell_1, \ell_2)p_2(a_2, \ell_2, \ell_3)p_3(a_3, \ell_3, \ell_4)p_4(a_4, \ell_4, \ell_1)\prod_{i=1}^4 da_i\prod_{i=1}^4 d\ell_i,
		\end{equation}
		where $a = a_1+a_2+a_3+a_4$. Applying the change of variables
		\begin{equation*}
			x = \frac{a_2+a_3}{a}; \qquad y = \frac{a_1+a_2}{a};\qquad z = \frac{a_2}{a};\qquad a = a_1+a_2+a_3+a_4;
		\end{equation*}
		then one can compute that $\left|\frac{\partial(a_1, a_2, a_3, a_4) }{\partial(x,y,z,a)}\right| = a^3$. Then \eqref{eqn-density-3} is equal to 
		\begin{multline}\label{eqn-density-4}
			c\int_{\bbR_{+}^4}\int_{x_1}^{x_2}\int_{y_1}^{y_2}\int_0^\infty\int_{\max\{0, x+y-1\}}^{\min\{x,y\}}a^3f(a) p_1((y-z)a, \ell_1, \ell_2)\cdot\\
			p_2(za, \ell_2, \ell_3)p_3((x-z)a, \ell_3, \ell_4)p_4((1+z-x-y)a, \ell_4, \ell_1)\,dz\,da\,dy\,dx\prod_{i=1}^4 d\ell_i.
		\end{multline}
		{If} we further apply the change of variables $\ell_i \mapsto \sqrt{a}\ell_i$, then from Lemma~\ref{lmm-area-scaling} and the fact that $W_1+W_2+W_3+W_4 = W = 4-\gamma^2$ from \eqref{eqn-weight}, we get that \eqref{eqn-density-4} is the same as
		\begin{multline*}
			c\int_0^\infty f(a)a^{-\frac{4}{\gamma^2}}da\int_{\bbR_{+}^4}\int_{x_1}^{x_2}\int_{y_1}^{y_2}\int_{\max\{0, x+y-1\}}^{\min\{x,y\}} p_1(y-z, \ell_1, \ell_2)\cdot\\
			p_2(z, \ell_2, \ell_3)p_3(x-z, \ell_3, \ell_4)p_4(1+z-x-y, \ell_4, \ell_1)\,dz\,dy\,dx\,\prod_{i=1}^4 d\ell_i,
		\end{multline*}
		which concludes the proof.
	\end{proof}
	
	\subsubsection{The explicit formula for the density of the Baxter permuton}
	\label{sec:density-Baxter}
	In this section we restrict to the case when $q = \frac{1}{2}$ and $\gamma = \sqrt{4/3}$. Then, as remarked below Theorem~\ref{thm:sbpfromlqg}, we have that $\theta = \theta_\gamma(q)= 0$. In addition, from \eqref{eqn-weight}, we also have that 
	\[W_1 = W_2=W_3=W_4=1-\frac{\gamma^2}{4} = \frac{\gamma^2}{2}=\frac{2}{3}.\] 
	
	We refer the reader to Remark~\ref{rem:generalization} for a discussion on the difficulties to address the general case $\rho\in(-1,1)$ and $q\in(0,1)$. 
	From Proposition~\ref{prop-disk-excursion}, if $W=\frac{\gamma^2}{2}$ then the quantum area of a sample from $\mathcal{M}_2^{\textup{disk}}(\frac{\gamma^2}{2}; \ell_1, \ell_2)^\#$ has the same law as the duration of a sample from $\mu_{\mathcal{C}_\phi}^\#(\ell_1\sqrt{2\sin\phi}, \ell_2\sqrt{2\sin\phi}e^{i\phi})$ with $\phi = \frac{\pi\gamma^2}{4}$ (where we recall that $\mu_{\mathcal{C}_\phi}^\#(\ell, re^{i\phi})$ denotes the law of the Brownian excursion in the cone $\mathcal{C}_\phi$ from $\ell$ to $re^{i\phi}$ with non-fixed time duration). In our specific case $\frac{\gamma^2}{2}=\frac{2}{3}$ and $\phi = \frac{\pi}{3}$. 
	
	Building on this, we prove in the next proposition that the density of the area of a quantum disk sampled from $\mathcal{M}_2^{\text{disk}}(\frac{2}{3};x,r)$ introduced in \eqref{eqn-def-area-density} is a constant times the function $
	\rho$ given by \eqref{eqn-rho}. This will conclude the proof of Theorem~\ref{thm-baxter-density}.
	
	\begin{proposition}\label{prop-bm-duration}
		For $\phi = \frac{\pi}{3}$, $x,r>0$ the duration $\tau$ of a sample from $\mu_{\mathcal{C}_{\frac{\pi}{3}}}^\#(x, re^{\frac{\pi i}{3}})$ has density
		\begin{equation*}
			\wt{p}(t,x,r):=\left(\left(\frac{3x r}{2t}-1\right)e^{-\frac{x^2+r^2-x r}{2t}}+e^{-\frac{(x+r)^2}{2t}}\right)\frac{(x^3+r^3)^2}{18x^2r^2}\cdot\frac{1}{t^2}\cdot \mathds 1_{t>0}.
		\end{equation*}
	\end{proposition}
	
	\begin{proof}
		Let $(e_1, e_2)$ be the standard basis for $\bbR^2$. For $j = 1, ..., 5$, let $F_j$ be the reflection on $\bbR^2$ about line $y = \tan \frac{j\pi }{3}x$,  $T_0 = \textup{id}$ and $T_j = F_j\circ T_{j-1}$. Also for $z = x+iy =  re^{i\theta}$, let $\wt{z} = re^{i(\frac{\pi}{3}-\theta)} = \frac{x+\sqrt{3}y}{2}+\frac{\sqrt{3}x-y}{2}i$ be its reflection about $y = \tan \frac{\pi }{6}x$. Then for a standard Brownian motion $(W_t)_{t\ge 0} = (X_t, Y_t)_{t\ge 0}$ started at $z\in \mathcal{C}_\phi$ killed upon leaving $\mathcal{C}_\phi$ (the corresponding probability measure is denoted by $\mathbb{P}^z$), following \cite[Equation 16]{Ty85}, its duration $\tau$ and the hitting point $W_\tau$ has joint law
		\begin{equation}\label{eqn-excursion-jointlaw}
			\bbP^z(\tau\in dt, W_\tau\in re^{\frac{\pi i}{3}}\,dr) = \frac{1}{4\pi t^2}\sum_{k=0}^5 (-1)^k e^{-\frac{|\wt{z} - rT_ke_1|^2}{2t}}(\wt{z}\cdot T_ke_2)\,dtdr:=p_1(x,y,t,r)\,dtdr,
		\end{equation}
		where the dot represents the usual inner product in $\bbR^2$. Note that $T_0e_1 = T_5e_1 = 1$, $T_1e_1 = T_2e_1 = e^{\frac{2\pi i}{3}}$, $T_3e_1 = T_4e_1 = e^{\frac{4\pi i}{3}}$, $T_0e_2 = -T_5e_2 = i$, $T_1e_2 = -T_2e_2 = e^{\frac{\pi i}{6}}$, $T_3e_2 = -T_4e_2 = e^{\frac{5\pi i}{6}}$. Then the right-hand side of \eqref{eqn-excursion-jointlaw} can be written as
		\begin{equation*}
			\frac{1}{2\pi t^2}\left(\frac{\sqrt{3}x-y}{2}e^{-\frac{x^2+y^2+r^2-r(x+\sqrt{3}y)}{2t}}-\frac{\sqrt{3}x+y}{2}e^{-\frac{x^2+y^2+r^2-r(x-\sqrt{3}y)}{2t}} + ye^{-\frac{x^2+y^2+r^2+2rx}{2t}}  \right)dtdr.
		\end{equation*}
		On the other hand, using the conformal mapping $z\mapsto z^3$ and the conformal invariance of planar Brownian motion, together with standard planar Brownian exit probability calculations on $\mathbb{H}$, one has
		\begin{equation*}
			\bbP^z\left(W_\tau\in  re^{\frac{\pi i}{3}}dr\right) = \frac{3r^2}{\pi}\frac{3x^2y-y^3}{(-r^3-(x^3-3xy^2))^2+(3x^2y-y^3)^2)}\,dr:=p_2(x,y,r)\,dr,
		\end{equation*}
		and it follows that $	\bbP^z(\tau\in dt\,|\, W_\tau\in re^{\frac{\pi i}{3}}\,dr) = \frac{p_1(x,y,t,r)}{p_2(x,y,r)}\,dt$. 
		
		Now for fixed $x,t, r$, as $y\to 0^+$, we have 
		\begin{equation}\label{eqn-excursion-jointlaw-2}
			p_1(x,y,t,r) = \frac{1}{2\pi t^2}\big((\frac{3xr}{2t}-1)e^{-\frac{x^2+r^2-rx}{2t}}+e^{-\frac{(x+r)^2}{2t}}  \big)y + o(y^2);
		\end{equation}
		\begin{equation}\label{eqn-excursion-jointlaw-3}
			p_2(x,y,r) = \frac{9x^2r^2}{\pi(x^3+r^3)^2}y+o(y^2).
		\end{equation}
		Therefore combining \eqref{eqn-excursion-jointlaw-2}, \eqref{eqn-excursion-jointlaw-3} along with the convergence \eqref{eqn-excursion-convergence}, it follows that for $x, r>0$, 
		\[ 
		\bbP^x\left(\tau\in dt\middle| W_\tau\in re^{\frac{\pi i}{3}}dr\right) = \frac{(r^3+x^3)^2}{18r^2x^2t^2}\left(\left(\frac{3rx}{2t}-1\right)e^{-\frac{x^2+r^2-rx}{2t}}+e^{-\frac{(x+r)^2}{2t}}  \right) dt,
		\]
		and this concludes the proof.
	\end{proof}
	
	\begin{remark}
		We remark that the sum \eqref{eqn-excursion-jointlaw} comes from solving the heat equation
		\begin{equation*}
			\partial_t u(t, z) = \frac{1}{2}\Delta u(t, z), \ u(0, z) = f(z), \ z\in \mathcal{C}_\phi; \qquad u(t,z) = 0, \ z\in \partial \mathcal{C}_\phi,
		\end{equation*}
		via the \emph{method of images}. The solution takes a simple form if $\phi = \frac{\pi}{m}$ for an integer $m>0$, while for general $\phi\in(0,\pi)$, the  \eqref{eqn-excursion-jointlaw} can be written as an infinite sum in terms of the Bessel functions \cite[Equation 8]{Ty85}.
	\end{remark}
	
	We can now complete the proof of Theorem~\ref{thm-baxter-density}.
	
	\begin{proof}[Proof of Theorem~\ref{thm-baxter-density}]
		Combining Propositions~\ref{prop-disk-excursion}~and~\ref{prop-bm-duration} along with \eqref{eqn-size-disk-measure} for $\gamma= \sqrt{4/3}$, we see that the quantum area of a sample from $\mathcal{M}_2^{\textup{disk}}(\frac{\gamma^2}{2}; x, r)$ has density given by a universal constant $c$ times the function $p(t,x,r)$ introduced in \eqref{eqn-rho}. Then the conclusion is straightforward from Theorem~\ref{thm-baxter-density-general}.
	\end{proof}
	
	\begin{remark}\label{rem:generalization}
		We remark that for general $q\in(0,1)$ and $\gamma\in(-1,1)$, Theorem~\ref{thm-baxter-density-general} gives a description of the skew Brownian permuton in terms of the density $p_W$ of quantum disks. For  $W\neq \frac{\gamma^2}{2} $, an explicit description of the law of the quantum area under $\mathcal{M}_2^{\textup{disk}}\left(W; \ell, r\right)$ will be given in a forthcoming work~\cite{ARS22} of Ang, Remy, Zhu and the third author of this paper. This and other results from~\cite{ARS22} will then be used to give a formula for  $\theta_\gamma(q)$  by Ang and the third and the fourth authors of this paper. The law of the quantum area is much more involved than its counterpart when $W=\frac{\gamma^2}{2}$, but preliminary calculations suggest that the formula for $\theta_\gamma(q)$ is rather simple. 
	\end{remark}
	
	\subsubsection{Relations between the density of the Baxter permuton and spherical tetrahedrons}\label{sect:rel_tetra}
	In this subsection, we comment on the relation between the density $p_B(x, y)$ given by \eqref{eqn-baxter-density} and the area function of spherical tetrahedrons in $\mathbb{S}^3:=\left\{(x_1,x_2,x_3,x_4)\in\mathbb R^4:x_1^2+x_2^2+x_3^2+x_4^2=1\right\}$.
	
	Recall the function $\rho(t, x, r)$ in \eqref{eqn-rho}, that is
	\begin{equation*}
		\rho(t, x, r):= \frac{1}{t^2} \left(\left(\frac{3rx}{2t}-1\right)e^{-\frac{r^2+x^2-rx}{2t}}+e^{-\frac{(x+r)^2}{2t}}\right).
	\end{equation*}
	Let  
	\begin{equation*}
		g(a_1, a_2, a_3, a_4):=\int_{\bbR_{+}^4} \rho(a_1, \ell_1, \ell_2)\rho(a_2, \ell_2, \ell_3)\rho(a_3, \ell_3, \ell_4)\rho(a_4, \ell_4, \ell_1)\,d\ell_1\,d\ell_2\,d\ell_3\,d\ell_4.
	\end{equation*}
	From Propositions~\ref{prop-disk-excursion}~and~\ref{prop-bm-duration} we know that $g$ is the joint law of the quantum areas of the four weight $\frac{\gamma^2}{2}$ quantum disks obtained by welding a weight $4-\gamma^2$ quantum sphere (as in the statement of Theorem~\ref{thm-sph-welding}) {for $\gamma=\sqrt{4/3}$}. On the other hand, the density $p_B$ of the intensity measure $\bbE [\mu_B]$ of the Baxter permuton $\mu_B$ satisfies, as stated in \eqref{eqn-baxter-density},
	\begin{equation*}
		p_B(x, y) =  c\int_{\max\{0, x+y-1\}}^{\min\{x, y\}}g(y-z, z, x-z, 1+z-x-y)\,dz.
	\end{equation*}
	For fixed $a_1, a_2, a_3, a_4$, the function $g$ can be written as a linear combination of integrals
	\begin{equation}\label{eqn-gaussians}
		\int_{\bbR_{+}^4} \prod_{j\in J}x_jx_{j+1} e^{-\frac{1}{2}x^T\Sigma x}\,dx,
	\end{equation}
	where $J\subset\{1,2,3,4\}$, $x_5=x_1$, and $\Sigma$ is a non-negative definite matrix depending only on $a_1, a_2, a_3, a_4$. 
	
	{There are two cases: (i) $\Sigma$ is non-singular (implying that $\Sigma$ is positive definite) and (ii) $\Sigma$ is singular. We will only consider case (i) below, but remark that \eqref{eqn-gaussians} for $\Sigma$ singular can be approximated arbitrarily well by \eqref{eqn-gaussians} for $\Sigma$ non-singular, so the discussion below is also relevant for case (ii) as we can consider an approximating sequence of non-singular matrices $\Sigma$.} 
	
	For $x\in \bbR^4$, we write $x\succeq 0$ if all the entries of $x$ are non-negative. Consider the function 
	\[F(\Sigma):=\int_{\bbR^4}e^{-\frac{1}{2}x^T\Sigma x}\mathds1_{x\succeq0}\,dx\] 
	defined on the space $\{\Sigma\in \bbR^{4\times 4}: \Sigma^T = \Sigma, \ \exists {\delta}>0, x^T\Sigma x\ge \delta \|x\|^2, \forall x\succeq 0  \}$, {which in particular contains the set of positive definite symmetric $4\times 4$ matrices}. Then it is clear that $F(\Sigma)$ is a smooth function in this domain.	
	{Since we assume} $\Sigma$ is positive definite, we have
	\[
	F(\Sigma) = \det(\Sigma)^{-\frac{1}{2}}\int_{\bbR^4}e^{-\frac{1}{2}y^Ty}\mathds 1_{\Sigma^{-\frac{1}{2}}y\succeq 0}\,dy.
	\]
	Using polar coordinates and letting
	\begin{equation}\label{eqn-s-function}
		S(\Sigma):=\left|\left\{y\in \mathbb{S}^3: \Sigma y\succeq 0\right\}\right|,
	\end{equation} 
	we get
	\begin{equation*}
		F(\Sigma) = \det(\Sigma)^{-\frac{1}{2}}S(\Sigma^{-\frac{1}{2}})\int_0^\infty r^3e^{-\frac{1}{2}r^2}\,dr = 2\det(\Sigma)^{-\frac{1}{2}}S(\Sigma^{-\frac{1}{2}}).
	\end{equation*}
	Hence the integral in \eqref{eqn-gaussians} can be expressed in terms of the function $S$ defined in \eqref{eqn-s-function} as follows
	\begin{equation*}
		\int_{\bbR_{+}^4} \prod_{j\in J}x_jx_{j+1} e^{-\frac{1}{2}x^T\Sigma x}\,dx =
		{(-1)^{|J|}\prod_{j\in J}\frac{\partial}{\partial \sigma_{j,j+1}} F(\Sigma)=}
		(-1)^{|J|}\prod_{j\in J}\frac{\partial}{\partial \sigma_{j,j+1}}2\det(\Sigma)^{-\frac{1}{2}}S(\Sigma^{-\frac{1}{2}}),
	\end{equation*}
	where $\Sigma = (\sigma_{ij})$ is viewed as element of $\bbR^{10}$.
	
	The function $S(\Sigma)$ can be described in terms of volume of spherical tetrahedron. The region $\{y\in \mathbb{S}^3: \Sigma y\succeq 0\}$ can be thought as points on the sphere staying on the positive side of the hyperplanes passing through the origin induced by rows of $\Sigma$. Then it follows that the six dihedral angles are given by $\left(\pi - \frac{\langle \sigma_i, \sigma_j \rangle}{|\sigma_i||\sigma_j|}\right)_{1\le i<j\le 4},$ where $\sigma_i$ is the $i$-th row of $\Sigma$, while the Gram matrix has entries $\frac{\langle \sigma_i, \sigma_j \rangle}{|\sigma_i||\sigma_j|} $. $S(\Sigma)$ is precisely given by the volume of the spherical tetrahedron with dihedral angles $\cos\theta_i = -\frac{\langle \sigma_1, \sigma_{i+1} \rangle}{|\sigma_1||\sigma_{i+1}|}$ for $i = 1, 2, 3$,  $\cos\theta_4 = -\frac{\langle \sigma_3, \sigma_4 \rangle}{|\sigma_3||\sigma_4|}$, $\cos\theta_5 = -\frac{\langle \sigma_2, \sigma_4 \rangle}{|\sigma_2||\sigma_4|}$, $\cos\theta_6 =- \frac{\langle \sigma_2, \sigma_3 \rangle}{|\sigma_2||\sigma_3|}$, which can be traced from \cite[Theorem 1.1]{Mur12} and also the Sforza's formula as listed in \cite[Theorem 2.7]{AM14}. Therefore the value of $S(\Sigma)$ can be described as a linear combination of dilogarithm functions. 
	
	\subsection{Expected proportion of inversions in the skew Brownian permuton}\label{sec:E21theta}
	
	In this section we {prove Proposition~\ref{prop-expected-occ-21}.} 
	We start with the following description for sampling a point $(x,y)$ in the unit square $[0,1]^2$ from the skew Brownian permuton $\mu_{\rho,q}$. Recall the notation for the quantum areas $A_1^{w},A_2^{w},A_3^{w},A_4^{w}$ introduced before Proposition~\ref{prop-baxter-disk} (see Figure~\ref{fig-flowlines}).
	
	\begin{lemma}\label{lm-permuton-sample}
		With probability 1, given an instance of the unit-area quantum sphere $(\wh{\mathbb{C}}, h, 0, \infty)$ and a whole-plane GFF $\wh{h}$ (viewed modulo a global additive integer multiple of $2\pi \chi$) with associated space-filling counterflow lines $\eta'$ and $\eta'_{\theta-\frac\pi2}$, the following two sampling procedures agree:
		\begin{enumerate}
			\item Let $\mu_{\rho,q}$ be the skew Brownian permuton constructed from the tuple $(h, \eta', \eta'_{\theta-\frac\pi2})$ as in Theorem~\ref{thm:sbpfromlqg}. Sample $(x,y)$ from $\mu_{\rho,q}$.
			\item First sample a point $\mathbf{w}\in \wh{\mathbb{C}}$ from the quantum area measure $\mu_h$. Output $(A_2^{\mathbf{w}}+A_3^{\mathbf{w}}, A_1^{\mathbf{w}}+A_2^{\mathbf{w}})$.
		\end{enumerate}
	\end{lemma}
	
	\begin{proof}
		Using the same reasoning as in Propositions~\ref{prop:permuton-q-area}~and~\ref{prop-baxter-disk}, by our choice of parameterization, a.s.\
		\begin{equation*}
			\int_{\mathbb{C}}\mathds{1}_{\left\{A_2^{w}+A_3^{w}\in [x_1,x_2], A_1^{w}+A_2^{w}\in [y_1,y_2]\right\}}\mu_h(dw)=\mu_h\left(\eta'([x_1, x_2])\cap\eta'_{\theta-\frac\pi2}([y_1, y_2])\right).
		\end{equation*}
		Applying Proposition~\ref{prop:permuton-q-area} once more, $\mu_{h}\Big(  \eta'_0([x_1, x_2])\cap \eta'_{\theta-\frac{\pi}{2}}([y_1, y_2])\Big)=\mu_{\rho,q}\Big([x_1, x_2]\times [y_1, y_2]\Big)$.
	\end{proof}
	
	We have the following expression for $\pocc(21, \mu_{\rho,q})$ (recall its definition from \eqref{eq:occ_permuton}).
	\begin{lemma}\label{lm-expected-21}	
		Let $(\wh{\mathbb{C}}, h, 0, \infty)$ be a unit-area quantum sphere  and  $\wh{h}$ an independent whole-plane GFF (viewed modulo a global additive integer multiple of $2\pi \chi$) with associated space-filling counterflow lines $\eta'$ and $\eta'_{\theta-\frac\pi2}$. Let $\mu_{\rho,q}$ be the skew Brownian permuton constructed from the tuple $(h, \eta', \eta'_{\theta-\frac\pi2})$ as in Theorem~\ref{thm:sbpfromlqg}. For  a point $\mathbf{w}\in \wh{\mathbb{C}}$ sampled from the quantum area measure $\mu_h$, it a.s.\ holds that
		\begin{equation*}
			\pocc(21, \mu_{\rho,q}) = 2\cdot\bbE \left[A_1^{\mathbf{w}}\middle|(h, \wh{h})\right].
		\end{equation*}
	\end{lemma}
	\begin{proof}
		By symmetry and the definition given in \eqref{eq:occ_permuton}, 
		\begin{equation}\label{eqn-permuton-21}
			\pocc(21, \mu_{\rho,q}) = 2\iint_{[0,1]^{2}}\mathds{1}_{\left\{x_1<x_2;\,y_1>y_2\right\}}\mu_{\rho,q}(dx_1dy_1)\mu_{\rho,q}(dx_2dy_2).
		\end{equation}
		Therefore applying Lemma~\ref{lm-permuton-sample}, if we first independently sample $(\mathbf{w}, \wt{\mathbf{w}})$ from the quantum area measure $\mu_h$, then the right-hand side of \eqref{eqn-permuton-21} is the same as
		\begin{equation}\label{eqn-permuton-21-a}
			2\iint_{\mathbb{C}^2}\mathds{1}_{\left\{A_2^{w}+A_3^{w}<A_2^{\wt{w}}+A_3^{\wt{w}};\, A_1^{w}+A_2^{w}>A_1^{\wt{w}}+A_2^{\wt{w}} \right\}}\mu_h(dw)\mu_h(d\wt{w}).
		\end{equation}
		Using again the definition of space-filling SLE curves given at the end of Section \ref{sect:sle_ig} (recall also Figure~\ref{fig-flowlines}), we observe that \[A_2^{\mathbf{w}}+A_3^{\mathbf{w}}<A_2^{\wt{\mathbf{w}}}+A_3^{\wt{\mathbf{w}}}\qquad \text{and} \qquad A_1^{\mathbf{w}}+A_2^{\mathbf{w}}>A_1^{\wt{\mathbf{w}}}+A_2^{\wt{\mathbf{w}}}\] if and only if $\eta'$ hits the point $\wt{\mathbf{w}}$ after hitting $\mathbf{w}$, and $\eta'_{\theta-\frac{\pi}{2}}$ hits $\wt{\mathbf{w}}$ before hitting  $\mathbf{w}$. This implies that  $\wt{\mathbf{w}}$ falls into the region  between $\eta_{\theta}^{{\mathbf{w}}}$ and $\eta_{W}^{{\mathbf{w}}}$ (i.e.\ the region with quantum area $A_1^{\mathbf{w}}$).  Therefore we conclude the proof by integrating \eqref{eqn-permuton-21-a} over $\wt{w}$.
	\end{proof}
	
	\begin{proof}[Proof of Proposition~\ref{prop-expected-occ-21}]
		By Lemma~\ref{lm-expected-21}, it suffices to show that $\bbE \left[A_1^{\mathbf{w}}\right] = \frac{\pi-2\theta}{4\pi}$. By the +rerooting invariance stated in Proposition~\ref{prop-qs-resampling}, the quantum area $A_1^{\mathbf{w}}$ has the same distribution as $A_1:=A_1^{0}$. It remains to prove that
		\begin{equation}\label{eqn-expected-occ-21}
			\bbE [A_1] = \frac{\pi-2\theta}{4\pi}.
		\end{equation}
		First assume that $\theta = \theta_{m,n}:=\left(\frac{m}{n}-\frac{1}2\right)\pi,$ where $0\le m\le 2n-1$ are integers. By Theorem~\ref{thm-sph-welding}, the flow lines $\eta_{\theta_{0,n}}$, ..., $\eta_{\theta_{2n-1,n}}$ of $\wh{h}$, with angle $\theta_{0,n}, ..., \theta_{2n-1,n}$, cut the whole sphere into $2n$ quantum disks {each} of weight $\frac{4-\gamma^2}{2n}$. We denote the quantum area of the region between $\eta_{\frac{i\pi}{n}}$ and $\eta_{\frac{(i+1)\pi}{n}}$ by $A_{i,n}$, for $i = 0, ..., 2n-1$ (with the convention that $\eta_{2\pi} = \eta_0$). Since the total area {is} 
		1, 
		by symmetry $\bbE  [A_{i,n}] = \frac{1}{2n}$. Then from linearity of expectation, we see that 
		\begin{equation*}
			\bbE [A_1] = \bbE[A_{m, n}]+...+\bbE[A_{n-1,n}] = \frac{n-m}{2n} = \frac{\pi-2\theta}{4\pi},
		\end{equation*}
		which verifies \eqref{eqn-expected-occ-21} for $\theta\in\mathbb{Q}$. Now for general $\theta$, we observe that by flow line monotonicity (see \cite[Theorem 1.5]{MS16a} and \cite[Theorem 1.9]{MS17}) the flow lines starting from the same point with different angles will not cross each other, and it follows that the expression $\bbE [A_1]$ is decreasing in $\theta$. Then it is clear that \eqref{eqn-expected-occ-21} holds for any $\theta\in [-\frac{\pi}2,\frac\pi2]$, which concludes the proof.
	\end{proof}
	
	\begin{remark}
		Although the quantity $A_1^{\mathbf{w}}$ appearing in Lemma~\ref{lm-expected-21} is generally tractable using the rerooting invariance for marked points of quantum spheres, its conditional expectation given $(h,\wh{h})$ would be more tricky. In particular, the rerooting invariance is a key technical step in the proof.
	\end{remark}
	
	\begin{remark}
		The proof of Lemma~\ref{lm-expected-21} does not only give the expectation of $\wt{\occ}(21,\mu_{\rho,q})$ as done in Proposition~\ref{prop-expected-occ-21}; it also gives a description of the law of this random variable in terms of formulas for LQG surfaces. The law can be expressed in terms of the function $\theta_\gamma(q)$, the function $p_W$ from Section~\ref{sec:density-skew-Bp}, and counterparts of Proposition~\ref{prop-disk-excursion} for disks of other weights.
	\end{remark}
	
	\begin{remark}
		Also $\wt{\occ}(\pi,\mu_{\rho,q})$ for other choices of $\pi$ can be expressed in terms of the LQG area of certain domains cut out by flow lines started from a fixed number of points sampled from the LQG area measure. However, for general patterns $\pi$, the expectation of the relevant LQG area is not as straightforward to compute, and the intersection pattern of the flow lines is more involved. We therefore do not pursue more general formulas. However, we do prove in Section~\ref{sec:positivity} that we have a.s.\ positivity of $\wt{\occ}(\pi,\mu_{\rho,q})$ for all (standard) patterns $\pi$.
	\end{remark}
	
	\section{Positivity of pattern densities of the skew Brownian permuton}
	\label{sec:positivity}
	
	The goal of this section is to prove Theorem~\ref{thm:positivity}. Our proof will use the theory of imaginary geometry from \cite{MS16a,MS17} (see also \cite{dub09a}). Following these papers, let 
	\[
	\kappa\in(0,4),\qquad
	\chi = \frac{2}{\sqrt\kappa}-\frac{\sqrt\kappa}{2} ,\qquad \lambda'=\frac{\pi\sqrt\kappa}{4}.
	\]
	Recall from Section~\ref{sect:sle_ig} that if $\wh h$ is a whole-plane GFF (defined modulo a global additive integer multiple of $2\pi\chi$), $\theta\in\R$ and $z\in\C$, then we can define the flow line $\eta^z_\theta$ of $e^{i(\wh h/\chi+\theta)}$ from $z$ to $\infty$ of angle $\theta$, which is an SLE$_{\kappa}(2-\kappa)$ curve. We refer to flow lines of angle $\theta=0$ (resp.\ $\theta=\pi/2$) as north-going (resp.\ west-going). As explained in \cite{MS16a,MS17}, one can also define flow lines if one has a GFF in a subset $D\subseteq \C$, including flow lines which start from a point on the domain boundary $\partial D$ for appropriate boundary conditions.
	
	Let $z\in\C$, $\theta\in\R$, and $\wh h$ be as in the previous paragraph, and let $\tau$ be a stopping time for $\eta^z_\theta$. 
	Conditioned on $\eta^z_\theta|_{[0,\tau]}$, the conditional law of $\wh h$ is given by a zero boundary GFF in $\C\setminus \eta^z_\theta([0,\tau])$, plus the function $f$ which is harmonic in this domain and has boundary conditions along $\eta^z_\theta([0,\tau])$ given by $\chi$ times the winding of the curve plus $-\lambda'-\theta\chi$ (resp.\ $\lambda'-\theta\chi$) on the left (resp.\ right) side, where the winding is relative to a segment of the curve going straight upwards. We refer to \cite[Section 1]{MS17} for the precise description of this conditional law and in particular to \cite[Figures 1.9 and 1.10]{MS17} for more details on boundary conditions and the concept of winding. The analogous statement holds if we consider flow lines of a GFF $\wh h$ in a subset $D\subset\C$.
	
	As mentioned above, the whole-plane GFF is typically only defined modulo a global additive integer multiple of $2\pi\chi$ in the setting of imaginary geometry. Throughout the remainder of this section we will fix this additive constant by requiring that the average of the GFF on the unit circle is between 0 and $2\pi\chi$. Fixing the additive constant is convenient when considering the height difference between two interacting flow lines and when we want to describe the absolute boundary values along each flow line. 
	
	To simplify notation we will focus on the case $q=1/2$ of Theorem~\ref{thm:positivity} throughout the section, and then afterwards explain the necessary (very minor) modification which is needed for general $q\in(0,1)$. The key input to the proof of Theorem~\ref{thm:positivity} is the following lemma. See Figure~\ref{fig-flowlines-perm} for an illustration.
	
	\begin{figure}[ht]
		\centering
		\includegraphics[scale=1]{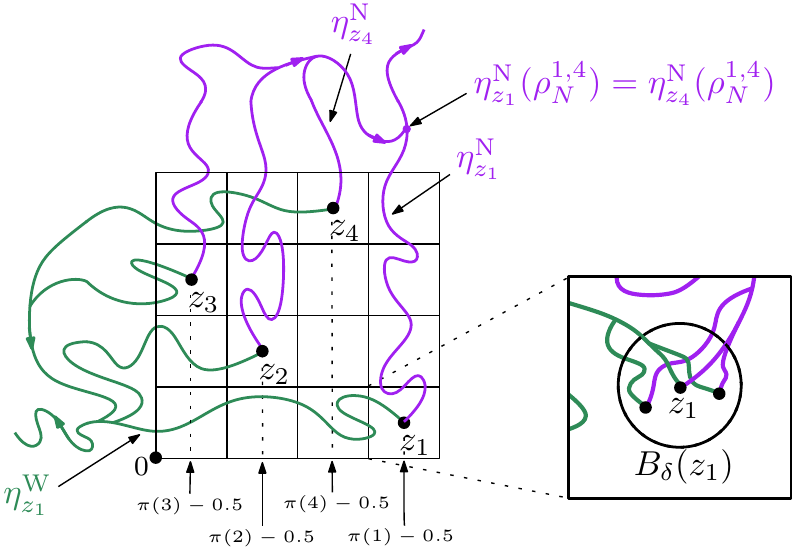}
		\caption{Illustration of (i)-(iii) in Lemma~\ref{prop-pos-dens-key0} for $\pi=4213$, i.e. $\pi^{-1}=3241$. Recalling the explanations from Section~\ref{sect:sle_ig} (see also Figure~\ref{flow-lines}), we have that if the merging structure of the west (resp.\ north) flow line is as the one in the picture, then the ordering in which the points $z_1,z_2,z_3,z_4$ are visited by the space-filling SLE$_{16/\kappa}$ counterflow line $\eta'$ (resp.\ $\eta'_{-\frac\pi2}$) is  $z_1,z_2,z_3,z_4$ (resp.\ $z_3,z_2,z_4,z_1$). Indeed, by the constriction of counterflow lines, $\eta'$ (resp.\ $\eta'_{-\frac\pi2}$) visits the points $z_1,z_2,z_3,z_4$ in the same order as the contour of the green (resp. purple) tree oriented from south to north (resp.\ from west to east). Note also that this implies that $\Perm_k\Big(\left((\eta')^{-1}(z_i),(\eta'_{-\frac{\pi}{2}})^{-1}(z_i)\right)_{i\in[4]}\Big)=\pi$. On the right, we zoom inside the square $(3,4)\times(0,1)$ and display condition (iii) in Lemma~\ref{prop-pos-dens-key0}. Note that the flow lines started from points (we displayed two of them) inside the ball $B_{\delta}(z_1)$ merge into the flow lines started from $z_1$ before leaving the square.}
		\label{fig-flowlines-perm}
	\end{figure}
	
	\begin{lemma}
		Let $k\in\{2,3,\dots \}$ and let $\pi\in\mcl S_k$ be a (standard) pattern of size $k$. For $j=1,\dots,k$ 
		let $z_j= (\pi(j)-0.5)+(j-0.5)\mathbf{i}\in\C$ and  let $\rho_{\op{N}}^{j,i}$ (resp.\ $\rho_{\op{W}}^{j,i}$) be the time at which $\eta^{z_j}_{\op{N}}$ (resp.\ $\eta^{z_j}_{\op{W}}$) merges into $\eta^{z_i}_{\op{N}}$ (resp.\ $\eta^{z_i}_{\op{W}}$) for $i\in\{1,\dots,k \}\setminus\{j \}$. Then there is a $\delta\in(0,1/10)$ such that with strictly positive probability the following events occur for all $i,j=1,\dots,k$, $i\neq j$.
		\begin{itemize}
			\item[(i)] $\eta^{z_j}_{\op{N}}$ merges into $\eta^{z_i}_{\op{N}}$ on its left side if and only if $\pi(j)<\pi(i)$; these two flow lines merge before leaving the ball $B_{4k}(0)$; and $\eta^{z_j}_{\op{N}}(\rho_{\op{N}}^{j,i})\not\in  (\pi(j)-1,\pi(j))\times(j-1,j)$. 
			\item[(ii)] $\eta^{z_j}_{\op{W}}$ merges into $\eta^{z_i}_{\op{W}}$ on its left side if and only if $j<i$; these two flow lines merge before leaving the ball $B_{4k}(0)$; and $\eta^{z_j}_{\op{W}}(\rho_{\op{W}}^{j,i})\not\in  (\pi(j)-1,\pi(j))\times(j-1,j)$.
			
			\item[(iii)] For all $z\in B_\delta(z_j)$ 
			the flow line $\eta^z_{\op{N}}$ (resp.\ $\eta^z_{\op{W}}$) merges into $\eta^{z_j}_{\op{N}}$ (resp.\ $\eta^{z_j}_{\op{W}}$) before leaving the square $(\pi(j)-1,\pi(j))\times(j-1,j)$.
		\end{itemize}
		\label{prop-pos-dens-key0}
	\end{lemma}

	Before proceeding to the proof of Lemma~\ref{prop-pos-dens-key0}, we give the proof of Theorem~\ref{thm:positivity} conditioned on this result. 
	
	\begin{proof}[Proof of Theorem~\ref{thm:positivity} for $q=1/2$ given Lemma~\ref{prop-pos-dens-key0}]
		For $m\in\N\cup\{0 \}$ and $\wh h$ a whole-plane GFF as above, let $E(m,\wh h)$ be the event (i)-(ii)-(iii) of Lemma~\ref{prop-pos-dens-key0}, but with all points scaled by $2^{-m}$, i.e., we consider the setting of the lemma under the image of the map $z\mapsto 2^{-m}z$ (equivalently,  (i)-(ii)-(iii) occur for the field $\wh h(2^m\cdot)$).
		By Lemma~\ref{prop-pos-dens-key0} we have $s:=\P[E(0,\wh h)]>0$, and by scale invariance of the GFF we have $\P[E(m,\wh h)]=s$ for all $m\in\N$. Since the occurrence of $E(m,\wh h)$ is determined by $\wh h|_{B_{2^{-m+2}k}(0)}$, 
		we get by tail triviality 
		of  $(\wh h|_{B_{2^{-m+2}k}(0)}\,:\,m\in\N)$ 
		{(see e.g.\ \cite[Lemma 2.2]{hs-euclidean-mating})}
		that $E(m,\wh h)$ occurs for infinitely many $m$ a.s. In particular, we can a.s.\ find some (random) $m_0\in\N$ such that  $E(m_0,\wh h)$ occurs. 
		
		Recall that $\eta'$ and $\eta'_{-\frac\pi2}$ denotes the angle 0 and the angle $-\frac{\pi}{2}$ space-filling SLE$_{16/\kappa}$ counterflow lines constructed from $\wh h$.
		By the definition of space-filling SLE$_{16/\kappa}$ counterflow line as given at the end of Section \ref{sect:sle_ig} (see also Figure~\ref{fig-flowlines-perm}), if (i) (resp.\ (ii)) occurs then the ordering in which the points $z_1,\dots,z_k$ are visited by $\eta'$ (resp.\ $\eta'_{-\frac\pi2}$) is $z_1,\dots,z_k$ (resp.\ $z_{\pi^{-1}(1)},\dots,z_{\pi^{-1}(k)}$). Furthermore, by the same argument, if (i) (resp.\ (ii)) occurs in the setting where points are rescaled by $2^{-m_0}$ then the ordering in which the points $2^{-m_0}z_1,\dots,2^{-m_0}z_k$ are visited by $\eta'$ (resp.\ $\eta'_{-\frac\pi2}$) is $2^{-m_0}z_1,\dots,2^{-m_0}z_k$ (resp.\ $2^{-m_0}z_{\pi^{-1}(1)},\dots,2^{-m_0}z_{\pi^{-1}(k)}$). 
		
		Finally, also by the definition of space-filling SLE$_{16/\kappa}$, if (iii) occurs in addition to (i) and (ii), then every points $z\in B_{2^{-m_0}\delta}(2^{-m_0}z_i)$ and $w\in B_{2^{-m_0}\delta}(2^{-m_0}z_j)$ are visited in the same relative ordering as $2^{-m_0}z_i$ and $2^{-m_0}z_j$ for both $\eta'$ and $\eta'_{-\frac\pi2}$. Indeed, for $z\in B_{2^{-m_0}\delta}(2^{-m_0}z_i)$ and $w\in B_{2^{-m_0}\delta}(2^{-m_0}z_j)$, if (i) and (iii) occur, we have that $\eta_{\op{N}}^{z}$ merges into $\eta_{\op{N}}^{w}$ on its left side if and only if $\eta_{\op{N}}^{z_i}$ merges into $\eta_{\op{N}}^{z_j}$ on its left side, and the corresponding statements hold with (ii) and W instead of (i) and N, respectively. 
		
		We now consider a unit-area quantum sphere $(\wh{\mathbb{C}}, h, 0, \infty)$ independent of $\wh h$ and the skew Brownian permuton $\mu_{\rho,q}$ constructed form the tuple $(h, \eta', \eta'_{-\frac\pi2})$ as in Theorem~\ref{thm:sbpfromlqg}.
		Note that by \eqref{eq:occ_permuton} and Theorem~\ref{thm:sbpfromlqg}, a.s.
		\begin{multline*}
			\pocc(\pi, \mu_{\rho,q})
			=
			\int_{[0,1]^{2k}}
			\mathds{1}_{\left\{\Perm_k((x_i,y_i)_{i\in[k]})=\pi\right\}}
			\prod_{i=1}^k \mu_{\rho,q}(dx_i,dy_i)\\
			=
			\int_{\C^k}
			\mathds{1}_{\left\{\Perm_k\left(\left(
				(\eta')^{-1}(w_i),
				(\eta'_{-\frac{\pi}{2}})^{-1}(w_i)
				\right)_{i\in[k]}\right)=\pi\right\}}
			\prod_{i=1}^k \mu_{h}(dw_i),
		\end{multline*}
		where $\mu_h$ is the $\gamma$-LQG area measure associated with $(\wh{\mathbb{C}}, h, 0, \infty)$.
		Hence, if (i), (ii), and (iii) occur (in the setting where all points are rescaled by $2^{-m_0}$) then it a.s.\  holds 
		\begin{equation}\label{eq:low_bound}
			\pocc(\pi, \mu_{\rho,q})
			\geq 
			\int_{B_{2^{-m_0}\delta}(2^{-m_0}z_1)}\dots
			\int_{B_{2^{-m_0}\delta}(2^{-m_0}z_k)} 
			\prod_{i=1}^k \mu_h(dw_i).	
		\end{equation}
		The latter bound concludes the proof since the balls $B_{2^{-m_0}\delta}(2^{-m_0}z_j)$ for $j=1,\dots,k$ a.s.\ have positive $\mu_h$ Liouville quantum area measure. 
	\end{proof}
	
	\begin{proof}[Proof of Theorem~\ref{thm:positivity} for general $q\in(0,1)$]
		All steps of the proof carry through precisely as in the case $q=1/2$, except that we consider $\eta_\theta^z$ instead of $\eta_{\op{N}}^z$ throughout the proof for $\theta$ such that $q=q_\gamma(\theta)$.
	\end{proof}
	
	The rest of this section is devoted to the proof of Lemma~\ref{prop-pos-dens-key0}. We will in fact instead prove Lemma~\ref{prop-pos-dens-key} below, which immediately implies Lemma~\ref{prop-pos-dens-key0}. In order to state Lemma~\ref{prop-pos-dens-key}, we first need the following definition.
	\begin{definition}
		Let $B\subset\C$ be a set of the form $(a,a+s)\times(b,b+s)$ for $a,b\in\R$ and $s>0$,
		denote its top (resp.\ bottom, left, right) boundary arc by $\partial_{\op{T}} B$ (resp.\ $\partial_{\op{B}} B,\partial_{\op{L}} B,\partial_{\op{R}} B$), and let $z\not\in B$. We say that $\eta^z_{\op{N}}$ \emph{crosses $B$ nicely in north direction} if the following criteria are satisfied, where $\tau=\inf\{t\geq 0\,:\,\eta^z_{\op{N}}(t)\in B \}$ is the first time at which $\eta^z_{\op{N}}$ hits $B$.
		\begin{itemize}
			\item[(i)] $\tau<\infty$ and $\eta^z_{\op{N}}(\tau)\in\partial_{\op{B}} B$.
			\item[(ii)] Let $\wt\eta$ be a path which agrees with $\eta^z_{\op{N}}$ until time $\tau$ and which parametrizes a vertical line segment in $B$ during $[\tau,\tau+1]$. Let $f$ be the function which is harmonic in $\C\setminus\wt\eta([0,\tau+1])$, is equal to $-\lambda'$ (resp.\ $\lambda'$) on the left (resp.\ right) side of the vertical segment $\wt\eta([\tau,\tau+1])$, and which otherwise along $\wt\eta$ changes by $\chi$ times the winding of $\wt\eta$. We require that the boundary conditions of $\wh h$ along $\eta^z_{\op{N}}|_{[0,\tau]}$ are as given by $f$. 
			\item[(iii)] $\eta^z_{\op{N}}$ does not have any top-bottom crossings, 
			i.e., if $\tau'=\inf\{t\geq 0\,:\,\eta^z_{\op{N}}(t) \in \partial_{\op{T}} B \}$ then $\tau'<\infty$ and  $\eta^z_{\op{N}}([\tau',\infty))\cap \partial_{\op{B}} B=\emptyset$. 
		\end{itemize}
		We say that $\eta^z_{\op{W}}$ \emph{crosses $B$ nicely in west direction} if the following criteria are satisfied, where $\tau''=\inf\{t\geq 0\,:\,\eta^z_{\op{W}}(t)\in B \}$ is the first time at which $\eta^z_{\op{W}}$ hits $B$.
		\begin{itemize}
			\item[(i')] $\tau''<\infty$ and $\eta^z_{\op{W}}(\tau'')\in\partial_{\op{R}} B$.
			\item[(ii')] Let $\wt\eta$ be a path which agrees with $\eta^z_{\op{W}}$ until time $\tau''$ and which parametrizes a horizontal line segment in $B$ during $[\tau'',\tau''+1]$. Let $f$ be the function which is harmonic in $\C\setminus\wt\eta([0,\tau+1])$, is equal to $-\lambda'-\pi\chi/2$ (resp.\ $\lambda'-\pi\chi/2$) on the bottom (resp.\ top)  side of the horizontal segment $\wt\eta([\tau'',\tau''+1])$, and which otherwise along $\wt\eta$ changes by $\chi$ times the winding of $\wt\eta$. We require that the boundary conditions of $\wh h$ along $\eta^z_{\op{W}}|_{[0,\tau'']}$ are as given by $f$.
		\end{itemize}
		\label{def-nice}
	\end{definition}
	Notice that the requirements in (ii) and (ii') above are automatically satisfied if we are only interested in the boundary conditions of the curve modulo a global additive integer multiple of $2\pi\chi$, but that these requirements are non-trivial in our setting (since we fixed the additive constant of the field) and depend on the winding of the flow lines about their starting point. For example, suppose $\eta_{\op{N}}^z$ would make an additional counterclockwise loop around $z$ before entering $B$; then its boundary conditions when crossing $B$ would increase by $2\pi\chi$, and we need to keep track of these multiples of $2\pi\chi$ when checking whether (ii) occurs. It is important to keep track of these multiples of $2\pi\chi$ when studying the interaction of two flow lines, e.g.\ in Lemmas \ref{prop-channel-hit-bdy} and \ref{prop-hit-merge-angle} below.
	
	Also notice that we do not require the counterpart of (iii) for west-going flow lines. This is due to the specific argument we use below where we first sample north-going flow lines and then sample the west-going flow lines conditioned on the realization of the north-going flow lines, and property (iii) is introduced in order to guarantee that it is possible to sample well-behaved west-going flow lines conditioned on the realization of the north-going flow lines.
	
	\begin{figure}[ht]
		\centering
		\includegraphics[scale=1]{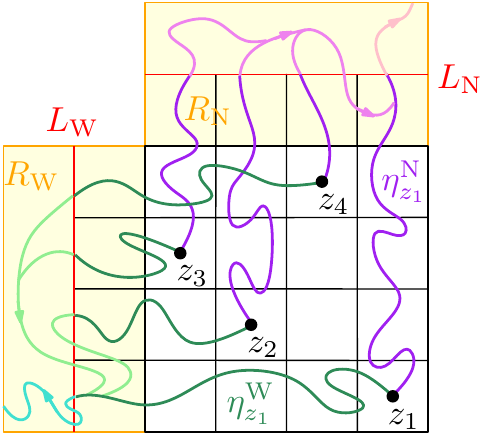}
		\caption{Illustration of (i)-(iv) in Lemma~\ref{prop-pos-dens-key} (or (i)-(ii) in Lemma~\ref{prop-pos-dens-key0}) for $\pi=4213$, i.e. $\pi^{-1}=3241$. The green curves represent west-going flow lines and the purple/pink curves represent north-going flow lines. The flow lines are shown in a different color before and after the times $\tau^{\op{N}}_j,\tau^{\op{W}}_j$.}
		\label{fig-pos-dens-key}
	\end{figure}

	\begin{lemma}
		Let $k\in\{2,3,\dots \}$ and let $\pi\in\mcl S_k$ be a (standard) pattern of size $k$. For $j=1,\dots,k$ 
		let $z_j= (\pi(j)-0.5)+(j-0.5)\mathbf{i}\in\C$ and make the following definitions (see Figure~\ref{fig-pos-dens-key}):
		\begin{align*}
			&L_{\op{N}}=[0,k]\times\{k+1 \},\qquad
			R_{\op{N}}=[0,k]\times[k,k+2], \qquad
			\tau_{\op{N}}^j=\inf\{t\geq 0\,:\, \eta^{z_j}_{\op{N}}(t) \in L_{\op{N}} \},\\
			&L_{\op{W}}=\{-1 \}\times[0,k], \qquad\,\,\,\,
			R_{\op{W}}=[-2,0]\times[0,k],\qquad\,\,\,\,\,\,
			\tau_{\op{W}}^j=\inf\{t\geq 0\,:\, \eta^{z_j}_{\op{W}}(t) \in L_{\op{W}} \}.
		\end{align*}
		Also let $\rho_{\op{N}}^{j,i}$ (resp.\ $\rho_{\op{W}}^{j,i}$) be the time at which $\eta^{z_j}_{\op{N}}$ (resp.\ $\eta^{z_j}_{\op{W}}$) merges into $\eta^{z_i}_{\op{N}}$ (resp.\ $\eta^{z_i}_{\op{W}}$) for $i\in\{1,\dots,k \}\setminus\{j \}$. Then there is a $\delta\in(0,1/10)$ such that with strictly positive probability the following events occur for all $i,j=1,\dots,k$, $i\neq j$.
		\begin{itemize}
			\item[(i)] The flow line $\eta^{z_j}_{\op{N}}$ stays inside $(\pi(j)-1,\pi(j))\times(j-1,k+1)$ until time $\tau_{\op{N}}^j<\infty$, and $\eta^{z_j}_{\op{N}}|_{[0,\tau_{\op{N}}^j]}$ crosses $(\pi(j)-1,\pi(j))\times(m-1,m)$ nicely in north direction for $m=j+1,\dots,k$.
			\item[(ii)] The flow line $\eta^{z_j}_{\op{W}}$ stays inside $(-1,\pi(j))\times(j-1,j)$ until time $\tau_{\op{W}}^j<\infty$, and $\eta^{z_j}_{\op{W}}|_{[0,\tau_{\op{W}}^j]}$ crosses $(m-1,m)\times(j-1,j)$ nicely in west direction for $m=1,\dots,\pi(j)-1$.
			\item[(iii)] $\eta^{z_j}_{\op{N}}$ merges into $\eta^{z_i}_{\op{N}}$ on its left side if and only if $\pi(j)<\pi(i)$, and $\eta^{z_j}_{\op{N}}([\tau^j_{\op{N}},\rho_{\op{N}}^{j,i}])\subset R_{\op{N}}$. 
			\item[(iv)] $\eta^{z_j}_{\op{W}}$ merges into $\eta^{z_i}_{\op{W}}$ on its left side if and only if $j<i$, and $\eta^{z_j}_{\op{W}}([\tau_{\op{W}}^j,\rho_{\op{W}}^{j,i}])\subset R_{\op{W}}$.
			\item[(v)] For all $z\in B_\delta(z_j)$ 
			the flow line $\eta^z_{\op{N}}$ (resp.\ $\eta^z_{\op{W}}$) merges into $\eta^{z_j}_{\op{N}}$ (resp.\ $\eta^{z_j}_{\op{W}}$) before leaving the square $(\pi(j)-1,\pi(j))\times(j-1,j)$.
		\end{itemize}
		\label{prop-pos-dens-key}
	\end{lemma}
	
	Note that Lemma~\ref{prop-pos-dens-key} immediately implies Lemma~\ref{prop-pos-dens-key0}.
	
	The next two lemmas say, roughly speaking, that a flow line stays close to any given curve $\gamma$ with positive probability. In the first lemma we consider the flow line until it hits a given curve in the bulk of the domain, while in the second lemma we consider the flow line until it hits the domain boundary. Closely related results are proved in \cite{MS17}. These two results will be stated for flow lines of general angle $\theta\in\BB R$ since they will be applied both to north-going and west-going flow lines, and the result for a general angle is no harder to prove that the result for any fixed angle. See Figure \ref{fig-flowlines-pos2} for an illustration of the following result.
	
	\begin{figure}[ht]
		\centering
		\includegraphics[scale=1.3]{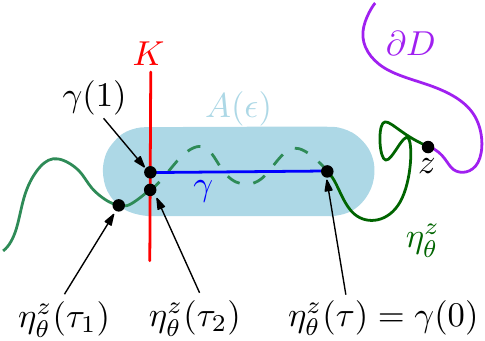}
		\caption{Illustration of the statement of Lemma \ref{prop-channel-hit-interior}. Here we choose the domain $D$ equal to the complement of the purple curve (which is some given curve) and $z\in\ol D$ to be the tip of the purple curve. In green we plotted the flow line $\eta_\theta^z$. The trace of the vertical red segment is the set $K$ and $\gamma$ is the blue horizontal curve, which is parametrized from right to left. The light blue region $A(\ep)$ is the $\ep$-neighborhood of $\gamma([0,1])$.  The figure is illustrating the event $\{\tau_2<\tau_1\}$.}
		\label{fig-flowlines-pos2}
	\end{figure}

	\begin{lemma}[Bulk case]
		Let $\wh h$ be a GFF 
		in a domain $D\subseteq\C$. 
		Let $z\in\ol D$, $\theta\in\R$, and $\eta_\theta^z$ be the flow line of $e^{i(\wh h/\chi+\theta)}$ of angle $\theta$ started from $z$. Let $K\subset D\setminus\{z \}$ be the trace of a simple curve in $D\setminus\{z \}$. Let also $\tau$ be an almost surely strictly positive and finite stopping time for $\eta^z_\theta$ such that $\eta^z_\theta(\tau)\not \in \eta^z_\theta([0,\tau))$, $\eta^z_\theta([0,\tau))\cap K=\emptyset$, and 
		$K$ and $\eta_\theta^z(\tau)$ are in the same connected component of $D\setminus \eta_\theta^z([0,\tau))$
		almost surely.\footnote{If $z\in\partial D$ we require in particular that the boundary conditions of $\wh h$ in $D$ close to $z$ are such that the flow line and an appropriate stopping time $\tau$ exists.} Given $\eta^z_\theta|_{[0,\tau]}$, let $\gamma:[0,1]\to D$ be a simple path satisfying $\gamma(0)=\eta^z_\theta(\tau)$, $\gamma(1)\in K$, and $\gamma((0,1))\cap (\eta^z_\theta([0,\tau))\cup K)=\emptyset$. For fixed $\ep>0$, let $A(\ep)$ denote the $\ep$-neighborhood of $\gamma([0,1])$, and define
		\[
		\tau_1 := \inf\{t\geq\tau\,:\, \eta^z_\theta(t)\not\in A(\ep) \},\qquad 
		\tau_2 = \inf\{t\geq\tau\,:\,\eta^z_\theta(t)\in K \}.
		\]
		Then $\P[ \tau_2<\tau_1\,|\, \eta^z_\theta|_{[0,\tau]} ]>0$. 
		\label{prop-channel-hit-interior}
	\end{lemma}
	\begin{proof} Our proof is very similar to that of \cite[Lemma 3.8]{MS17} and we will therefore only explain the difference as compared to that proof. The reader should consult that proof for the definition of $U$ and $x_0$ below. There are two differences between our lemma and \cite[Lemma 3.8]{MS17}. First, the latter lemma requires $D=\C$, while we consider general domains $D$ and allow $z\in\partial D$. Second, we define $\tau_2$ to be the hitting time of the set $K$ instead of letting it be the time that $\eta^z_\theta$ gets within distance $\ep$ of $\gamma(1)$. The proof carries through just as before with the first change. For the second change, the proof also carries through just as before except that (in the notation of the proof of \cite[Lemma 3.8]{MS17}) we pick the point $x_0\in\partial U$ in the proof such that any path in $U$ connecting $\eta(\tau)$ and $x_0$ must intersect $K$.
	\end{proof}
	
	The following lemma is \cite[Lemma 3.9]{MS17}, except that we have stated it for flow lines of a general angle $\theta\in \BB R$. We first introduce some terminology appearing in the next lemma. It is recalled below the statement of the lemma in \cite{MS17} that the admissible range of \emph{height differences for hitting} is $(-\pi\chi, 2\lambda - \pi\chi)$ (resp.\ $(\pi\chi-2\lambda, \pi\chi)$) if the flow line is hitting on the right (resp.\ left) side, where we refer to \cite[Figure 1.13]{MS17} for the definition of the \emph{height difference} between two flow lines when they intersect. \emph{Flow line boundary conditions} means that the boundary conditions for the GGF $\wh h$ determining the flow line change by $\chi$ times the winding of the curve.
	\begin{lemma}[Boundary case]
		Suppose that $\wh h$ is a GFF on a proper subdomain $D \subseteq \C$ whose boundary consists of a finite disjoint union of continuous paths, each with flow line boundary conditions of a given angle (which can change from path to path). Fix $z \in D$ and $\theta\in\R$ and let  $\eta_\theta^z$ be the flow line of $e^{i(\wh h/\chi+\theta)}$ of angle $\theta$ started from $z$.  Fix any almost surely positive and finite stopping time $\tau$ for $\eta_\theta^z$ such that $\eta_\theta^z([0,\tau]) \cap \partial D = \emptyset$ and $\eta_\theta^z(\tau) \notin \eta_\theta^z([0,\tau))$ almost surely.  Given $\eta_\theta^z|_{[0,\tau]}$, let $\gamma \colon [0,1] \to \ol{D}$ be any simple path in $\ol{D}$ starting from $\eta_\theta^z(\tau)$ such that $\gamma((0,1])$ is contained in the unbounded connected component of $\C \setminus \eta_\theta^z([0,\tau])$, $\gamma([0,1)) \cap \partial D = \emptyset$, and $\gamma(1) \in \partial D$.  Moreover, assume that if we extended the boundary conditions of the conditional law of $\wh h$ given $\eta_\theta^z|_{[0,\tau]}$  along $\gamma$ as if it were a flow line then the height difference of $\gamma$ and $\partial D$ upon intersecting at time $1$ is in the admissible range of height differences for hitting.  Fix $\epsilon > 0$, let $A(\epsilon)$ be the $\epsilon$-neighborhood of $\gamma([0,1])$ in $D$, and let
		\[ \tau_1 = \inf\{t \geq \tau : \eta_\theta^z(t) \notin A(\epsilon)\} \quad\text{and}\quad
		\tau_2 = \inf\{t \geq \tau : \eta_\theta^z(t) \in \partial D\}.\]
		Then $\P[ \tau_2 < \tau_1 \,|\, \eta_\theta^z|_{[0,\tau]}] > 0$.
		\label{prop-channel-hit-bdy}
	\end{lemma}
	
	The following result is a restatement of (part of) \cite[Theorem 1.7]{MS17} and gives a criterion to determine when two flow lines cross or merge when they hit each other.
	\begin{lemma}[Criterion for crossing/merging]
		Let $\wh h$ be GFF with arbitrary boundary conditions on $D\subseteq\C$. For $\theta_1,\theta_2\in\R$ and $z_1,z_2\in\ol D$ let $\tau$ be a stopping time for $\eta^{z_1}_{\theta_1}$ given $\eta^{z_2}_{\theta_2}$ and work on the event that $\eta^{z_1}_{\theta_1}$ hits $\eta^{z_2}_{\theta_2}$ on its right side at time $\tau$. Let $\Delta$ denote the height difference between $\eta^{z_1}_{\theta_1}$ and $\eta^{z_2}_{\theta_2}$ upon intersecting at $\eta^{z_1}_{\theta_1}(\tau)$. Then the following hold.
		\begin{itemize}
			\item[(i)] If $\Delta\in(-\pi\chi,0)$ then $\eta^{z_1}_{\theta_1}$ crosses $\eta^{z_2}_{\theta_2}$ at time $\tau$ and does not subsequently cross back.
			\item[(ii)] If $\Delta=0$ then $\eta^{z_1}_{\theta_1}$ merges with $\eta^{z_2}_{\theta_2}$ at time $\tau$ and does not subsequently separate from $\eta^{z_2}_{\theta_2}$.
		\end{itemize}
		\label{prop-hit-merge-angle}
	\end{lemma}
	
	The following lemma will be used to argue that the north-going flow lines in Lemma~\ref{prop-pos-dens-key} behave according to condition (i) with positive probability.
	
	\begin{figure}[ht]
		\centering
		\includegraphics[scale=0.8]{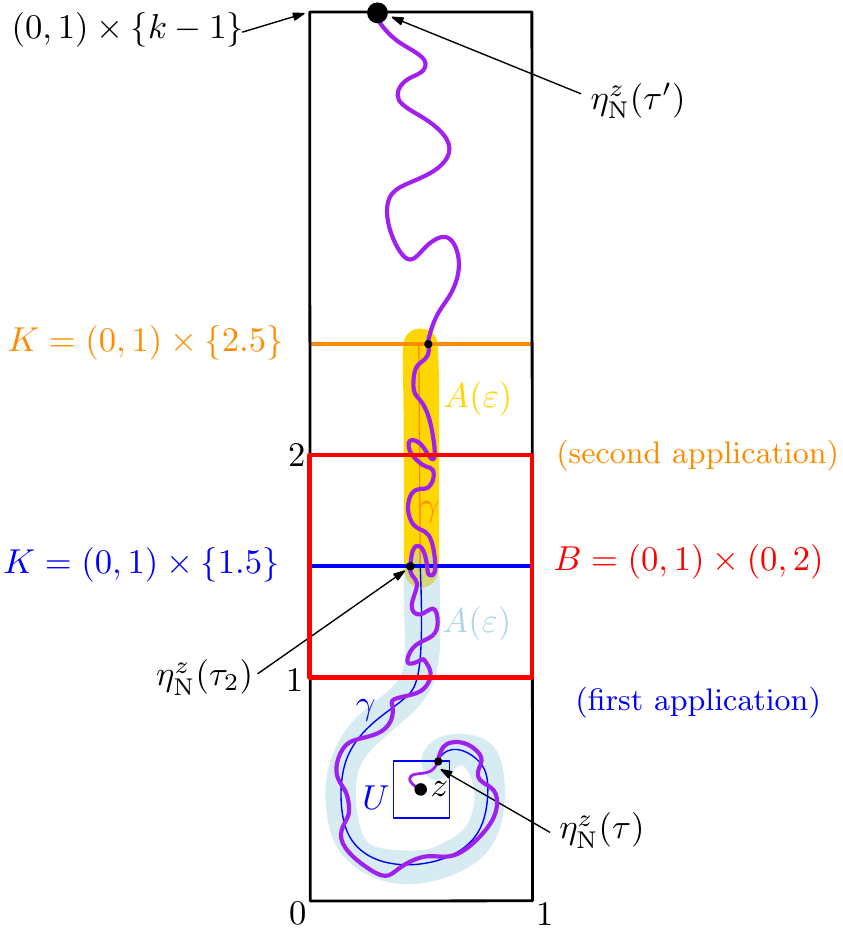}
		\caption{Illustration of the proof of Lemma \ref{prop-nice}. The path $\gamma$ and the set $K$ used in the first application of Lemma~\ref{prop-channel-hit-interior} are in blue, while the path $\gamma$ and the set $K$ used in the second application of Lemma~\ref{prop-channel-hit-interior} are in orange.}
		\label{fig-flow-lines-crossing2}
	\end{figure}
	
	\begin{lemma}[Condition (i) in Lemma~\ref{prop-pos-dens-key}]
		Let $z\in(0,1)\times(0,1)$, let $k\in\N$, and let $\tau'=\inf\{t\geq 0\,:\, \eta^z_{\op{N}}(t)\not\in (0,1)\times(0,k+1) \}$ be the time at which $\eta_{\op{N}}^z$ exits the rectangle $(0,1)\times(0,k+1)$. Then it holds with positive probability that $\eta_{\op{N}}^z(\tau')\in (0,1)\times\{k+1 \}$ and that $\eta_{\op{N}}^z$ crosses $(0,1)\times(i-1,i)$ nicely in north direction for $i=2,\dots,k$. 
		\label{prop-nice}
	\end{lemma}
	\begin{proof}
		See Figure \ref{fig-flow-lines-crossing2} for an illustration.
		The lemma follows by repeated applications of Lemma~\ref{prop-channel-hit-interior}. Throughout the proof we assume that $\ep\in(0, 1/10)$ in the statement of Lemma~\ref{prop-channel-hit-interior} is sufficiently small. First, we apply Lemma~\ref{prop-channel-hit-interior} to make sure conditions (i) and (ii) in Definition~\ref{def-nice} are satisfied for the square $B=(0, 1)\times(1,2)$ with positive probability. Letting $U$  
		be some neighborhood of $z$ which is compactly contained in $(0,1)\times(0,1)$, we apply Lemma~\ref{prop-channel-hit-interior} with $\tau=\inf\{t\geq 0\,:\, \eta^{z}_{\op{N}}(t)\not\in U \}$, $K=(0, 1)\times\{1.5 \}$ and $\gamma$ a path in $(0,1)\times(0,1.5)$ which winds around $z$ appropriately many times such that condition (ii) of Definition~\ref{def-nice} is satisfied for $B$. 
		We then apply Lemma~\ref{prop-channel-hit-interior} another time with $\tau$ equal to $\tau_2$ in the previous application of the lemma, $K=(0,1)\times\{ 2.5 \}$, and with $\gamma$ such that condition (iii) of Definition~\ref{def-nice} is satisfied for $B$ if $\tau_2<\tau_1$ and the flow line does not reenter $B$ after time $\tau_2$ in the second application of the lemma.
		
		We iteratively apply Lemma~\ref{prop-channel-hit-interior} for each square $(0,1)\times(i-1,i)$ with $i=2,\dots,k$ in order to guarantee that all the requirements of the lemma are satisfied. Note in particular that we need to stop the flow line at least once in each square in order to guarantee that (iii) in Definition~\ref{def-nice} is satisfied since Lemma~\ref{prop-channel-hit-interior} itself only guarantees that the flow line stays close to some reference path and does not rule out that the flow line oscillates many times back and forth along the reference path.
	\end{proof}
	
	The following lemma will be used to argue that the west-going flow lines in Lemma~\ref{prop-pos-dens-key} behave according to condition (ii) with positive probability.
	
	\begin{lemma}[Condition (ii) in Lemma~\ref{prop-pos-dens-key}]
		Let $k\in\{2,3,\dots \}$, $z\in(k-1,k)\times(0,1)$,  
		$\mcl I\subseteq\{1,\dots,k-1 \}$, 
		$z_j\in (j-1,1)\times(-\infty,0)$ for $j\in\mcl I$. 
		Suppose that for $j\in\mcl I$ the flow line $\eta^{z_j}_{\op{N}}$ crosses $(j-1,1)\times(0,1)$ nicely in north direction. Let $\tau'=\inf\{t\geq 0\,:\,\eta^z_{\op{W}}(t)\not\in(0,k)\times(0,1) \}$ be the time at which $\eta_{\op{W}}^z$ exits the rectangle $(0,k)\times(0,1)$. Then it holds with positive probability that 
		$\eta_{\op{W}}^z(\tau')\in \{0 \}\times(0,1)$ and that $\eta_{\op{W}}^z$ crosses each box $(i-1,i)\times(0,1)$ for $i=1,\dots,k-1$ nicely in west direction.
		\label{prop-cross-curves}
	\end{lemma}
	\begin{figure}[ht]
		\centering
		\includegraphics[scale=1]{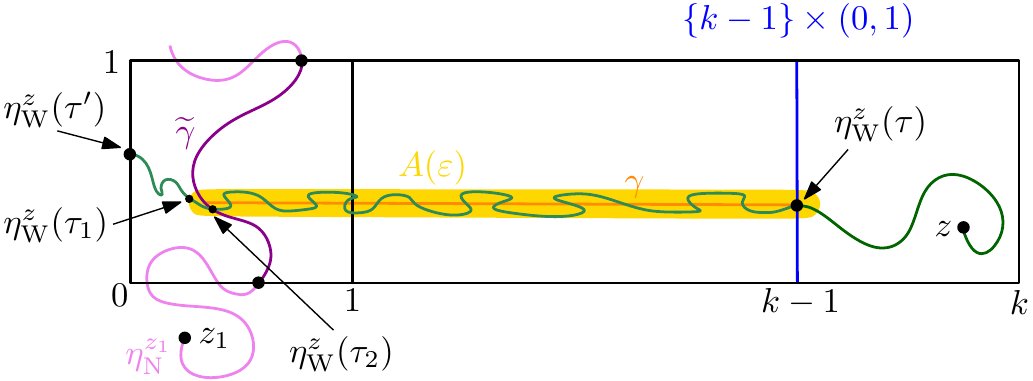}
		\caption{Illustration of the proof of Lemma \ref{prop-cross-curves} in the case $\mcl I=\{1 \}$. \label{fig-flow-lines-crossing}}
	\end{figure}
	\begin{proof}
		For concreteness we consider the case $\mcl I=\{1 \}$ but the general case can be treated similarly. See Figure \ref{fig-flow-lines-crossing} for an illustration. First apply Lemma~\ref{prop-channel-hit-interior} with $K=\{k-1 \}\times(0,1)$ similarly in the proof of Lemma~\ref{prop-nice} to make sure conditions (i') and (ii') in Definition~\ref{def-nice} are satisfied for the square $B=(k-2, k-1)\times(0,1)$ with positive probability. Let $\wt\gamma$ be the segment of $\eta^{z_1}_{\op{N}}$ corresponding to the (unique, by condition (ii) of Definition~\ref{def-nice}) up-crossing of $(0,1)\times(0,1)$, i.e., it is a path which starts (resp.\ ends) on the lower (resp.\ upper) boundary of $(0,1)\times(0,1)$. 
		Apply Lemma~\ref{prop-channel-hit-bdy} with 
		$\tau$ equal to the hitting time of $K=\{k-1 \}\times(0,1)$,
		$D$ equal to the infinite connected component of the complement of $\eta_{\op{N}}^{z_1}$,
		and $\gamma$ equal to a path starting at $\eta_{\op{W}}^z(\tau)$ and ending at an interior point of $\wt\gamma$, such that $\gamma$ does not cross $\eta^{z_1}_{\op{N}}$; it is possible to find an appropriate $\gamma$ by condition (iii) of Definition~\ref{def-nice}. When applying this lemma we note that by Definition \ref{def-nice}, the boundary conditions of the two flow lines is such that their height difference $\Delta=-\pi\chi/2$ is in the admissible range for hitting. By the first assertion of Lemma~\ref{prop-hit-merge-angle}, $\eta_{\op{W}}^z$ will cross $\wt\gamma$ without coming back immediately after hitting $\wt\gamma$. We now conclude the proof by applying Lemma~\ref{prop-channel-hit-interior} again. 
	\end{proof}
	
	Combining the lemmas above, we can now conclude the proof of Lemma~\ref{prop-pos-dens-key}.
	\begin{proof}[Proof of Lemma~\ref{prop-pos-dens-key}]
		We will first argue that (i)-(iv) occur with positive probability. Condition (i) occurs with positive probability by Lemma~\ref{prop-nice}, where we can apply the latter lemma iteratively for all the flow lines $\eta^{z_j}_{\op{N}}$ since the law of the field restricted to $(\pi(j)-1,\pi(j))\times(j-1,k+1)$ conditioned on the realization of a subset of the other flow lines is absolutely continuous with respect to the unconditional law of the field, conditioned on the event that none of the other flow lines intersect $[\pi(j)-1,\pi(j)]\times[j-1,k+1]$. Conditioned on (i), we get that (iii) occurs with positive probability by Lemma~\ref{prop-channel-hit-bdy} and the second assertion of Lemma~\ref{prop-hit-merge-angle},
		where we apply Lemma~\ref{prop-channel-hit-bdy} with $\tau=\inf\{ t\geq 0\,:\,\eta^{z_j}_{\op{N}}\not\in (\pi(j)-1,\pi(j))\times (j-1,k+1) \}$ and the flow lines are in the admissible range for merging (i.e., $\Delta=0$) due to condition (ii) of Definition~\ref{def-nice}. Note that in order to also guarantee that $\eta^{z_j}_{\op{N}}([\tau^j_{\op{N}},\rho_{\op{N}}^{j,i}])\subset R_{\op{N}}$, one can use a similar argument as in the proofs of Lemmas \ref{prop-nice} and \ref{prop-cross-curves}.
		
		Similarly, condition (ii) occurs with positive conditional probability given occurrence of (i) and (iii) by Lemma~\ref{prop-cross-curves}, and finally condition (iv) occurs with positive conditional probability given (i)-(iii) by Lemma~\ref{prop-channel-hit-bdy} and the second assertion of Lemma~\ref{prop-hit-merge-angle}. We conclude that (i)-(iv) occur with positive probability, and we denote this probability by $s>0$.
		
		To prove the full lemma, it is sufficient to argue that we a.s.\ can find a (random) $\delta>0$ such that the event in (v) occurs. Indeed, this implies that with probability at least $1-s/2$ the event in (v) occurs for some sufficiently small fixed $\delta>0$, which concludes the proof by the result of the previous paragraph and a union bound. We will now argue the a.s.\ existence of such a $\delta>0$. It is sufficient to consider only the north-going flow line starting from $z_1$. By continuity of the space-filling SLE $\eta'_{-\frac{\pi}{2}}$ generated by the north-going flow lines we can a.s.\ find an open interval $I$ such that  $\eta'_{-\frac{\pi}{2}}(I)$ is contained in $(\pi(1)-1,\pi(1))\times(0,1)$ and $z_1$ is contained in $\eta'_{-\frac{\pi}{2}}(I)$ 
		Furthermore, since $z_1$ is a.s.\ not a double point of $\eta'_{-\frac{\pi}{2}}$ and so $z_1$ must be contained in the interior of $\eta'_{-\frac{\pi}{2}}(I)$, there is a.s.\ a (random) $\delta>0$ such that $B_\delta(z_1)$ is contained in $\eta'_{-\frac{\pi}{2}}(I)$. This $\delta$ satisfies our requirement since for all $z\in B_\delta(z_1)$ the flow line $\eta^z_{\op{N}}$ merges into $\eta^{z_1}_{\op{N}}$ before leaving $\eta'_{-\frac{\pi}{2}}(I)\subset (\pi(1)-1,\pi(1))\times(0,1)$.
	\end{proof}
	
	\appendix
	
	\section{Permutation patterns}\label{sect:intro_patt}
	
	Recall that $\mathcal{S}_n$ denotes the set of permutations of size $n$ and $\mathcal{S}=\bigcup_{n\in\Z_{>0}} \mathcal{S}_n$ denotes the set of permutations of finite size. We write permutations using the \emph{one-line notation}, that is, if $\sigma$ is a permutation of size $n$ then we write $\sigma=\sigma(1)\dots\sigma(n)$. Given a permutation $\sigma$ of size $n$, its \emph{diagram} is a $n \times n$ table with $n$ points at position $(i,\sigma(i))$ for all $i\in[n]:=\{1,2,\dots,n\}$ (see the left-hand side of Figure~\ref{fig:pattern_perm_exemp}). Given a subset $I$ of the indexes of $\sigma$, i.e.\ $I\subseteq[n]$, recall that the \emph{pattern induced by} $I$ in $\sigma$, denoted $\pat_{I}(\sigma)$, is the permutation corresponding to the diagram obtained by rescaling the points $(i,\sigma(i))_{i\in I}$ in a $|I|\times |I|$ table (keeping the relative position of the points). Later, whenever $\pat_{I}(\sigma)=\pi$, we will also say that $(\sigma(i))_{i\in I}$ is an \emph{occurrence} of $\pi$ in $\sigma$. An example will be given in Example~\ref{exemp:pattern_perm_exemp} and Figure~\ref{fig:pattern_perm_exemp} below.
	
	A \emph{(standard) pattern} of size $k$ is just a permutation of size $k$.
	A permutation $\sigma$ \emph{avoids a (standard) pattern} $\pi$ if it is not possible to find a subset $I$ of the indexes of $\sigma$ such that $\pat_{I}(\sigma)=\pi$. The collection of all permutations (of any size) avoiding a set of (standard) patterns is often called a \emph{permutation class}.
	
	A \emph{generalized pattern} $\pi$ of size $k$, sometime also called \emph{vincular pattern}, is a permutation $\pi=\pi(1)\dots\pi(k)$, where some of its consecutive values are underlined. For instance, the permutation $7\underbracket[.5pt][1pt]{41}3\underbracket[.5pt][1pt]{526}$ is a generalized pattern. A permutation $\sigma$ \emph{avoids a generalized pattern} $\pi$, if it is not possible to find a subset $I$ of the indexes of $\sigma$ such that $\pat_{I}(\sigma)=\pi$ and
	$I$ has consecutive elements corresponding to the underlined values of $\pi$. We clarify the latter definition in the following example.
	
	\begin{example}\label{exemp:pattern_perm_exemp}
		We consider the permutation $\sigma=23641587$. Its diagram is plotted on the right-hand side of Figure~\ref{fig:pattern_perm_exemp}. Given the set of indices $I=\{2,3,5,6\}$, the pattern induced by $I$ in $\sigma$ is the permutation $\pat_I(\sigma)=2413$, plotted in the middle of Figure~\ref{fig:pattern_perm_exemp}. Therefore the permutation $\sigma$ does not avoid the standard pattern $2413$, but for instance it avoids the standard pattern $4321$ because it is not possible to find 4 points in the diagram of $\sigma$ that are in decreasing order.
		
		We also note that the permutation $\sigma$ avoids the generalized pattern $2\underbracket[.5pt][1pt]{41}3$. Indeed it is not possible to find four indices $i,j,j+1,k$ such that $1\leq i<j<j+1<k\leq 8$ and $\sigma(j+1)<\sigma(i)<\sigma(k)<\sigma(j)$.
	\end{example}
	
	We remark that Baxter permutations, introduced in Definition~\ref{def:baxter}, can be described as permutations avoiding the generalized patterns $2\underbracket[.5pt][1pt]{41}3$ and $3\underbracket[.5pt][1pt]{14}2$.
	
	\begin{figure}[ht]
		\centering
		\includegraphics[scale=1]{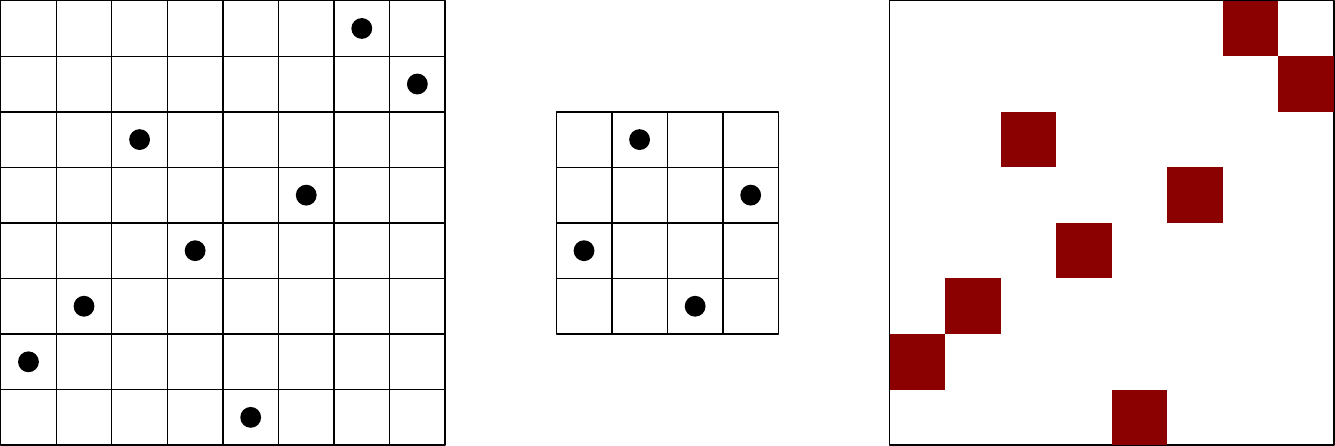}
		\caption{\label{fig:pattern_perm_exemp}\textbf{Left:} The diagram of the permutation $\sigma=23641587$. \textbf{Middle:} The pattern induced by the set of indices $I=\{2,3,5,6\}$ in $\sigma$, that is the permutation $\pat_I(\sigma)=2413$. \textbf{Right:} The permuton $\mu_{\sigma}$ corresponding to the permutation $\sigma=23641587$.}
	\end{figure}
	
	\noindent\textbf{Data availability statement:}
	Data sharing is not applicable to this article as no new data were created or analyzed in this study.
	
	\medskip
	
	\noindent\textbf{Conflict of interest:}
	The authors have no competing interests to declare that are relevant to the content of this article.

	\bibliographystyle{alpha}

	\bibliographystyle{alpha}

\newcommand{\etalchar}[1]{$^{#1}$}
\begin{thebibliography}{KMSW19}

\bibitem[ADK21]{alon2021runsort}
Noga Alon, Colin Defant, and Noah Kravitz.
\newblock The runsort permuton.
\newblock {\em arXiv preprint:2106.14762}, 2021.

\bibitem[AHS20]{AHS20}
Morris Ang, Nina Holden, and Xin Sun.
\newblock Conformal welding of quantum disks.
\newblock {\em arXiv preprint:2009.08389}, 2020.

\bibitem[AHS21]{AHS21}
Morris Ang, Nina Holden, and Xin Sun.
\newblock Integrability of {SLE} via conformal welding of random surfaces.
\newblock {\em arXiv preprint:2104.09477}, 2021.

\bibitem[AM14]{AM14}
Nikolay Abrosimov and Alexander Mednykh.
\newblock Volumes of polytopes in spaces of constant curvature.
\newblock In {\em Rigidity and Symmetry}, pages 1--26. Springer, 2014.

\bibitem[ARS21]{ARS21}
Morris {Ang}, Guillaume {Remy}, and Xin {Sun}.
\newblock {FZZ formula of boundary Liouville CFT via conformal welding}.
\newblock {\em arXiv preprint:2104.09478}, 2021.

\bibitem[ARSZon]{ARS22}
Morris Ang, Guillaume Remy, Xin Sun, and Tunan Zhu.
\newblock {Integrability of bulk-boundary coupling in Liouville CFT }.
\newblock In preparation.

\bibitem[AS21]{AS21}
Morris {Ang} and Xin {Sun}.
\newblock {Integrability of the conformal loop ensemble}.
\newblock {\em arXiv preprint: 2107.01788}, 2021.

\bibitem[Bax64]{Bax64}
Glen Baxter.
\newblock On fixed points of the composite of commuting functions.
\newblock {\em Proceedings of the American Mathematical Society},
  15(6):851--855, 1964.

\bibitem[BBF{\etalchar{+}}18]{bassino2018separable}
Fr\'{e}d\'{e}rique Bassino, Mathilde Bouvel, Valentin F\'{e}ray, Lucas Gerin,
  and Adeline Pierrot.
\newblock The {B}rownian limit of separable permutations.
\newblock {\em Ann. Probab.}, 46(4):2134--2189, 2018.

\bibitem[BBF{\etalchar{+}}19]{bassino2019scaling}
Fr{\'e}d{\'e}rique Bassino, Mathilde Bouvel, Valentin F{\'e}ray, Lucas Gerin,
  Micka{\"e}l Maazoun, and Adeline Pierrot.
\newblock Scaling limits of permutation classes with a finite specification: a
  dichotomy.
\newblock {\em arXiv preprint:1903.07522}, 2019.

\bibitem[BBF{\etalchar{+}}20]{bassino2017universal}
Fr\'{e}d\'{e}rique Bassino, Mathilde Bouvel, Valentin F\'{e}ray, Lucas Gerin,
  Micka\"{e}l Maazoun, and Adeline Pierrot.
\newblock Universal limits of substitution-closed permutation classes.
\newblock {\em J. Eur. Math. Soc. (JEMS)}, 22(11):3565--3639, 2020.

\bibitem[BBFS20]{MR4115736}
Jacopo Borga, Mathilde Bouvel, Valentin F\'{e}ray, and Benedikt Stufler.
\newblock A decorated tree approach to random permutations in
  substitution-closed classes.
\newblock {\em Electron. J. Probab.}, 25:Paper No. 67, 52, 2020.

\bibitem[BDS21]{borga2021almost}
Jacopo Borga, Enrica Duchi, and Erik Slivken.
\newblock Almost square permutations are typically square.
\newblock {\em Annales de l'Institut Henri Poincar{\'e}, Probabilit{\'e}s et
  Statistiques}, 57(4):1834--1856, 2021.

\bibitem[BGRR18]{MR3882946}
Mathilde Bouvel, Veronica Guerrini, Andrew Rechnitzer, and Simone Rinaldi.
\newblock Semi-{B}axter and strong-{B}axter: two relatives of the {B}axter
  sequence.
\newblock {\em SIAM J. Discrete Math.}, 32(4):2795--2819, 2018.

\bibitem[BM03]{MR2028288}
M.~Bousquet-M\'{e}lou.
\newblock Four classes of pattern-avoiding permutations under one roof:
  generating trees with two labels.
\newblock {\em Electron. J. Combin.}, 9(2):Research paper 19, 31, 2002/03.
\newblock Permutation patterns (Otago, 2003).

\bibitem[BM20]{BM20}
Jacopo Borga and Micka{\"e}l Maazoun.
\newblock Scaling and local limits of {B}axter permutations and bipolar
  orientations through coalescent-walk processes.
\newblock {\em arXiv preprint:2008.09086 (to appear in {A}nnals of
  {P}robability)}, 2020.

\bibitem[Bor21a]{borga2021permuton}
Jacopo Borga.
\newblock The permuton limit of strong-{B}axter and semi-{B}axter permutations
  is the skew {B}rownian permuton.
\newblock {\em arXiv preprint:2112.00159}, 2021.

\bibitem[Bor21b]{borga2021random}
Jacopo Borga.
\newblock Random permutations -- a geometric point of view.
\newblock {\em arXiv preprint:2107.09699 (Ph.D. Thesis)}, 2021.

\bibitem[Bor21c]{borga2021skewperm}
Jacopo Borga.
\newblock The skew {B}rownian permuton: a new universality class for random
  constrained permutations.
\newblock {\em arXiv preprint:2112.00156}, 2021.

\bibitem[Boy67]{MR0250516}
William~M. Boyce.
\newblock Generation of a class of permutations associated with commuting
  functions.
\newblock {\em Math. Algorithms 2 (1967), 19--26; addendum, ibid.}, 3:25--26,
  1967.

\bibitem[BP21]{berestycki-powell-notes}
Nathanael Berestycki and Ellen Powell.
\newblock {G}aussian free field, {L}iouville quantum gravity and {G}aussian
  multiplicative chaos.
\newblock {\em Lecture notes available at
  https://homepage.univie.ac.at/nathanael.berestycki/Articles/master.pdf},
  2021.

\bibitem[Can10]{MR2679559}
Hal Canary.
\newblock Aztec diamonds and {B}axter permutations.
\newblock {\em Electron. J. Combin.}, 17(1):Research Paper 105, 12, 2010.

\bibitem[CGHK78]{MR491652}
Fan-Rong~K. Chung, Ronald~L. Graham, Verner~Emil Hoggatt, Jr., and Mark
  Kleiman.
\newblock The number of {B}axter permutations.
\newblock {\em J. Combin. Theory Ser. A}, 24(3):382--394, 1978.

\bibitem[Dau18]{dauvergne2018archimedean}
Duncan Dauvergne.
\newblock The {A}rchimedean limit of random sorting networks.
\newblock {\em arXiv preprint:1802.08934 (To appear in Journal of the American
  Mathematical Society)}, 2018.

\bibitem[DMS21]{DMS14}
Bertrand Duplantier, Jason Miller, and Scott Sheffield.
\newblock Liouville quantum gravity as a mating of trees.
\newblock {\em Ast\'{e}risque}, 427, 2021.

\bibitem[DP14]{MR3238333}
Theodore Dokos and Igor Pak.
\newblock The expected shape of random doubly alternating {B}axter
  permutations.
\newblock {\em Online J. Anal. Comb.}, 9:12, 2014.

\bibitem[DS11]{DS11}
Bertrand Duplantier and Scott Sheffield.
\newblock Liouville quantum gravity and {K}{P}{Z}.
\newblock {\em Inventiones mathematicae}, 185(2):333--393, 2011.

\bibitem[Dub09]{dub09a}
Julien Dub{\'e}dat.
\newblock {SLE} and the free field: partition functions and couplings.
\newblock {\em Journal of the American Mathematical Society}, 22(4):995--1054,
  2009.

\bibitem[FFNO11]{MR2763051}
Stefan Felsner, \'{E}ric Fusy, Marc Noy, and David Orden.
\newblock Bijections for {B}axter families and related objects.
\newblock {\em J. Combin. Theory Ser. A}, 118(3):993--1020, 2011.

\bibitem[GHS16]{GHS16}
Ewain Gwynne, Nina Holden, and Xin Sun.
\newblock Joint scaling limit of a bipolar-oriented triangulation and its dual
  in the peanosphere sense.
\newblock {\em arXiv preprint:1603.01194}, 2016.

\bibitem[GHS19]{GHS19}
Ewain Gwynne, Nina Holden, and Xin Sun.
\newblock Mating of trees for random planar maps and {L}iouville quantum
  gravity: a survey.
\newblock {\em arXiv preprint:1910.04713}, 2019.

\bibitem[HS18]{hs-euclidean-mating}
Nina Holden and Xin Sun.
\newblock {SLE} as a mating of trees in {E}uclidean geometry.
\newblock {\em Communications in Mathematical Physics}, 364(1):171--201, 2018.

\bibitem[Iye85]{Ty85}
Satish Iyengar.
\newblock Hitting lines with two-dimensional brownian motion.
\newblock {\em SIAM Journal on Applied Mathematics}, 1985.

\bibitem[KMSW19]{KMSW19}
Richard Kenyon, Jason Miller, Scott Sheffield, and David~B Wilson.
\newblock Bipolar orientations on planar maps and {SLE}$_{12}$.
\newblock {\em The Annals of Probability}, 47(3):1240--1269, 2019.

\bibitem[LW04]{LW04}
Gregory Lawler and Wendelin Werner.
\newblock The brownian loop soup.
\newblock {\em Probability theory and related fields}, 128(4):565--588, 2004.

\bibitem[Maa20]{MR4079636}
Micka\"{e}l Maazoun.
\newblock On the {B}rownian separable permuton.
\newblock {\em Combin. Probab. Comput.}, 29(2):241--266, 2020.

\bibitem[Mal79]{MR555815}
Colin~L. Mallows.
\newblock Baxter permutations rise again.
\newblock {\em J. Combin. Theory Ser. A}, 27(3):394--396, 1979.

\bibitem[MS16]{MS16a}
Jason Miller and Scott Sheffield.
\newblock Imaginary {G}eometry {I}: Interacting {SLE}s.
\newblock {\em Probability Theory and Related Fields}, 164(3-4):553--705, 2016.

\bibitem[MS17]{MS17}
Jason Miller and Scott Sheffield.
\newblock {Imaginary geometry IV: interior rays, whole-plane reversibility, and
  space-filling trees}.
\newblock {\em Probability Theory and Related Fields}, 169(3):729--869, 2017.

\bibitem[Mur12]{Mur12}
Jun Murakami.
\newblock Volume formulas for a spherical tetrahedron.
\newblock {\em Proceedings of the American Mathematical Society},
  140(9):3289--3295, 2012.

\bibitem[Rom06]{MR2266895}
Dan Romik.
\newblock Permutations with short monotone subsequences.
\newblock {\em Adv. in Appl. Math.}, 37(4):501--510, 2006.

\bibitem[Sch00]{Sch00}
Oded Schramm.
\newblock Scaling limits of loop-erased random walks and uniform spanning
  trees.
\newblock {\em Israel Journal of Mathematics}, 118(1):221--288, 2000.

\bibitem[Sta09]{starr2009thermodynamic}
Shannon Starr.
\newblock Thermodynamic limit for the {M}allows model on {$S_n$}.
\newblock {\em J. Math. Phys.}, 50(9):095208, 15, 2009.

\end{thebibliography}


\begin{thebibliography}{KMSW19}
		
		\bibitem[ADK22]{alon2021runsort}
		Noga Alon, Colin Defant, and Noah Kravitz.
		\newblock The runsort permuton.
		\newblock {\em Adv. in Appl. Math.}, 139:Paper No. 102361, 18, 2022.
		
		\bibitem[AHS17]{ahs-sphere}
		Juhan Aru, Yichao Huang, and Xin Sun.
		\newblock Two perspectives of the 2{D} unit area quantum sphere and their
		equivalence.
		\newblock {\em Comm. Math. Phys.}, 356(1):261--283, 2017.
		
		\bibitem[AHS20]{AHS20}
		Morris Ang, Nina Holden, and Xin Sun.
		\newblock Conformal welding of quantum disks.
		\newblock {\em arXiv preprint:2009.08389}, 2020.
		
		\bibitem[AHS21]{AHS21}
		Morris Ang, Nina Holden, and Xin Sun.
		\newblock Integrability of {SLE} via conformal welding of random surfaces.
		\newblock {\em arXiv preprint:2104.09477}, 2021.
		
		\bibitem[AKL12]{alberts-kozdron-lawler-radial-green}
		Tom Alberts, Michael~J. Kozdron, and Gregory~F. Lawler.
		\newblock The {G}reen function for the radial {S}chramm-{L}oewner evolution.
		\newblock {\em J. Phys. A}, 45(49):494015, 17, 2012.
		
		\bibitem[ALS22]{als-cle-radius}
		Juhan Aru, Titus Lupu, and Avelio Sep\'{u}lveda.
		\newblock Extremal distance and conformal radius of a {$\rm CLE_4$} loop.
		\newblock {\em Ann. Probab.}, 50(2):509--558, 2022.
		
		\bibitem[AM14]{AM14}
		Nikolay Abrosimov and Alexander Mednykh.
		\newblock Volumes of polytopes in spaces of constant curvature.
		\newblock In {\em Rigidity and Symmetry}, pages 1--26. Springer, 2014.
		
		\bibitem[ARS21]{ARS21}
		Morris {Ang}, Guillaume {Remy}, and Xin {Sun}.
		\newblock {FZZ formula of boundary Liouville CFT via conformal welding}.
		\newblock {\em arXiv preprint:2104.09478}, 2021.
		
		\bibitem[ARSZ22]{ARS22}
		Morris Ang, Guillaume Remy, Xin Sun, and Tunan Zhu.
		\newblock Integrability of bulk-boundary coupling in {L}iouville {CFT}.
		\newblock {\em In preparation}, 2022+.
		
		\bibitem[AS21]{AS21}
		Morris {Ang} and Xin {Sun}.
		\newblock {Integrability of the conformal loop ensemble}.
		\newblock {\em arXiv preprint: 2107.01788}, 2021.
		
		\bibitem[ASY22]{ASY22}
		Morris Ang, Xin Sun, and Pu~Yu.
		\newblock Quantum triangles and imaginary geometry flow lines.
		\newblock {\em arXiv preprint: 2211.04580}, 2022.
		
		\bibitem[Bax64]{Bax64}
		Glen Baxter.
		\newblock On fixed points of the composite of commuting functions.
		\newblock {\em Proceedings of the American Mathematical Society},
		15(6):851--855, 1964.
		
		\bibitem[BBF{\etalchar{+}}18]{bassino2018separable}
		Fr\'{e}d\'{e}rique Bassino, Mathilde Bouvel, Valentin F\'{e}ray, Lucas Gerin,
		and Adeline Pierrot.
		\newblock The {B}rownian limit of separable permutations.
		\newblock {\em Ann. Probab.}, 46(4):2134--2189, 2018.
		
		\bibitem[BBF{\etalchar{+}}20]{bassino2017universal}
		Fr\'{e}d\'{e}rique Bassino, Mathilde Bouvel, Valentin F\'{e}ray, Lucas Gerin,
		Micka\"{e}l Maazoun, and Adeline Pierrot.
		\newblock Universal limits of substitution-closed permutation classes.
		\newblock {\em J. Eur. Math. Soc. (JEMS)}, 22(11):3565--3639, 2020.
		
		\bibitem[BBF{\etalchar{+}}22]{bassino2019scaling}
		Fr\'{e}d\'{e}rique Bassino, Mathilde Bouvel, Valentin F\'{e}ray, Lucas Gerin,
		Micka\"{e}l Maazoun, and Adeline Pierrot.
		\newblock Scaling limits of permutation classes with a finite specification: a
		dichotomy.
		\newblock {\em Adv. Math.}, 405:Paper No. 108513, 84, 2022.
		
		\bibitem[BBFS20]{MR4115736}
		Jacopo Borga, Mathilde Bouvel, Valentin F\'{e}ray, and Benedikt Stufler.
		\newblock A decorated tree approach to random permutations in
		substitution-closed classes.
		\newblock {\em Electron. J. Probab.}, 25:Paper No. 67, 52, 2020.
		
		\bibitem[BBMF10]{BBMF11}
		Nicolas Bonichon, Mireille Bousquet-M{\'e}lou, and {\'E}ric Fusy.
		\newblock Baxter permutations and plane bipolar orientations.
		\newblock {\em S{\'e}minaire Lotharingien de Combinatoire}, 61:B61Ah, 2010.
		
		\bibitem[BDS21]{borga2021almost}
		Jacopo Borga, Enrica Duchi, and Erik Slivken.
		\newblock Almost square permutations are typically square.
		\newblock {\em Annales de l'Institut Henri Poincare, Probabilites et
			Statistiques}, 57(4):1834--1856, 2021.
		
		\bibitem[BGRR18]{MR3882946}
		Mathilde Bouvel, Veronica Guerrini, Andrew Rechnitzer, and Simone Rinaldi.
		\newblock Semi-{B}axter and strong-{B}axter: two relatives of the {B}axter
		sequence.
		\newblock {\em SIAM J. Discrete Math.}, 32(4):2795--2819, 2018.
		
		\bibitem[BGS22]{bgs22meanders}
		Jacopo Borga, Ewain Gwynne, and Xin Sun.
		\newblock Permutons, meanders, and {SLE}-decorated {L}iouville quantum gravity.
		\newblock {\em arXiv preprint:2207.02319}, 2022.
		
		\bibitem[BI12]{beliaev-izyurov}
		Dmitri Beliaev and Konstantin Izyurov.
		\newblock A proof of factorization formula for critical percolation.
		\newblock {\em Comm. Math. Phys.}, 310(3):611--623, 2012.
		
		\bibitem[BJV13]{beliaev-viklund}
		Dmitry Beliaev and Fredrik Johansson~Viklund.
		\newblock Some remarks on {SLE} bubbles and {S}chramm's two-point observable.
		\newblock {\em Comm. Math. Phys.}, 320(2):379--394, 2013.
		
		\bibitem[BM03]{MR2028288}
		M.~Bousquet-M\'{e}lou.
		\newblock Four classes of pattern-avoiding permutations under one roof:
		generating trees with two labels.
		\newblock {\em Electron. J. Combin.}, 9(2):Research paper 19, 31, 2002/03.
		\newblock Permutation patterns (Otago, 2003).
		
		\bibitem[BM17]{BM15browniandisk}
		J\'{e}r\'{e}mie Bettinelli and Gr\'{e}gory Miermont.
		\newblock Compact {B}rownian surfaces {I}: {B}rownian disks.
		\newblock {\em Probab. Theory Related Fields}, 167(3-4):555--614, 2017.
		
		\bibitem[BM22]{BM20}
		Jacopo Borga and Micka\"{e}l Maazoun.
		\newblock Scaling and local limits of {B}axter permutations and bipolar
		orientations through coalescent-walk processes.
		\newblock {\em Ann. Probab.}, 50(4):1359--1417, 2022.
		
		\bibitem[BN14]{berestycki-sle-notes}
		Nathana{\"e}l {Berestycki} and James {Norris}.
		\newblock Lectures on {S}chramm--{L}oewner {E}volution.
		\newblock {A}vailable at
		\url{https://homepage.univie.ac.at/nathanael.berestycki/wp-content/uploads/2022/05/sle.pdf},
		2014.
		
		\bibitem[Bor21a]{borga2021random}
		Jacopo Borga.
		\newblock Random permutations -- a geometric point of view.
		\newblock {\em arXiv preprint:2107.09699 (Ph.D. Thesis)}, 2021.
		
		\bibitem[Bor21b]{borga2021skewperm}
		Jacopo Borga.
		\newblock The skew {B}rownian permuton: a new universality class for random
		constrained permutations.
		\newblock {\em arXiv preprint:2112.00156}, 2021.
		
		\bibitem[Bor22]{borga2021permuton}
		Jacopo Borga.
		\newblock The permuton limit of strong-{B}axter and semi-{B}axter permutations
		is the skew {B}rownian permuton.
		\newblock {\em Electron. J. Probab.}, 27:--, 2022.
		
		\bibitem[Boy67]{MR0250516}
		William~M. Boyce.
		\newblock Generation of a class of permutations associated with commuting
		functions.
		\newblock {\em Math. Algorithms 2 (1967), 19--26; addendum, ibid.}, 3:25--26,
		1967.
		
		\bibitem[BP21]{berestycki-powell-notes}
		Nathana~l Berestycki and Ellen Powell.
		\newblock {G}aussian free field, {L}iouville quantum gravity and {G}aussian
		multiplicative chaos.
		\newblock {\em Lecture notes available at
			\url{https://homepage.univie.ac.at/nathanael.berestycki/Articles/master.pdf}},
		2021.
		
		\bibitem[BPZ84]{bpz-conformal-symmetry}
		A.~A. Belavin, A.~M. Polyakov, and A.~B. Zamolodchikov.
		\newblock Infinite conformal symmetry in two-dimensional quantum field theory.
		\newblock {\em Nuclear Phys. B}, 241(2):333--380, 1984.
		
		\bibitem[BS20]{MR4149526}
		Jacopo Borga and Erik Slivken.
		\newblock Square permutations are typically rectangular.
		\newblock {\em Ann. Appl. Probab.}, 30(5):2196--2233, 2020.
		
		\bibitem[Can10]{MR2679559}
		Hal Canary.
		\newblock Aztec diamonds and {B}axter permutations.
		\newblock {\em Electron. J. Combin.}, 17(1):Research Paper 105, 12, 2010.
		
		\bibitem[CCM20]{chen-curien-maillard}
		Linxiao Chen, Nicolas Curien, and Pascal Maillard.
		\newblock The perimeter cascade in critical boltzmann quadrangulations
		decorated by an $ o (n) $ loop model.
		\newblock {\em Annales de l'Institut Henri Poincare D}, 7(4):535--584, 2020.
		
		\bibitem[Cer21]{cercle-quantum-disk}
		Baptiste Cercl\'{e}.
		\newblock Unit boundary length quantum disk: a study of two different
		perspectives and their equivalence.
		\newblock {\em ESAIM Probab. Stat.}, 25:433--459, 2021.
		
		\bibitem[CGHK78]{MR491652}
		Fan-Rong~K. Chung, Ronald~L. Graham, Verner~Emil Hoggatt, Jr., and Mark
		Kleiman.
		\newblock The number of {B}axter permutations.
		\newblock {\em J. Combin. Theory Ser. A}, 24(3):382--394, 1978.
		
		\bibitem[Dau22]{dauvergne2018archimedean}
		Duncan Dauvergne.
		\newblock The {A}rchimedean limit of random sorting networks.
		\newblock {\em J. Amer. Math. Soc.}, 35(4):1215--1267, 2022.
		
		\bibitem[DKRV16]{DKRV16}
		Fran{\c{c}}ois David, Antti Kupiainen, R{\'e}mi Rhodes, and Vincent Vargas.
		\newblock Liouville quantum gravity on the {R}iemann sphere.
		\newblock {\em Communications in Mathematical Physics}, 342(3):869--907, 2016.
		
		\bibitem[DMS21]{DMS14}
		Bertrand Duplantier, Jason Miller, and Scott Sheffield.
		\newblock Liouville quantum gravity as a mating of trees.
		\newblock {\em Ast\'{e}risque}, 427, 2021.
		
		\bibitem[DO94]{do-dozz}
		H.~{Dorn} and H.-J. {Otto}.
		\newblock {Two- and three-point functions in Liouville theory}.
		\newblock {\em Nuclear Physics B}, 429:375--388, October 1994.
		
		\bibitem[DP14]{MR3238333}
		Theodore Dokos and Igor Pak.
		\newblock The expected shape of random doubly alternating {B}axter
		permutations.
		\newblock {\em Online J. Anal. Comb.}, 9:12, 2014.
		
		\bibitem[DS11]{DS11}
		Bertrand Duplantier and Scott Sheffield.
		\newblock Liouville quantum gravity and {K}{P}{Z}.
		\newblock {\em Inventiones mathematicae}, 185(2):333--393, 2011.
		
		\bibitem[Dub06]{dub-watts}
		Julien Dub\'{e}dat.
		\newblock Excursion decompositions for {SLE} and {W}atts' crossing formula.
		\newblock {\em Probab. Theory Related Fields}, 134(3):453--488, 2006.
		
		\bibitem[Dub09]{dub09a}
		Julien Dub{\'e}dat.
		\newblock {SLE} and the free field: partition functions and couplings.
		\newblock {\em Journal of the American Mathematical Society}, 22(4):995--1054,
		2009.
		
		\bibitem[FFNO11]{MR2763051}
		Stefan Felsner, \'{E}ric Fusy, Marc Noy, and David Orden.
		\newblock Bijections for {B}axter families and related objects.
		\newblock {\em J. Combin. Theory Ser. A}, 118(3):993--1020, 2011.
		
		\bibitem[GHM20]{ghm-kpz}
		Ewain Gwynne, Nina Holden, and Jason Miller.
		\newblock An almost sure {KPZ} relation for {SLE} and {B}rownian motion.
		\newblock {\em Ann. Probab.}, 48(2):527--573, 2020.
		
		\bibitem[GHS16]{GHS16}
		Ewain Gwynne, Nina Holden, and Xin Sun.
		\newblock Joint scaling limit of a bipolar-oriented triangulation and its dual
		in the peanosphere sense.
		\newblock {\em arXiv preprint:1603.01194}, 2016.
		
		\bibitem[GHS19]{GHS19}
		Ewain Gwynne, Nina Holden, and Xin Sun.
		\newblock Mating of trees for random planar maps and {L}iouville quantum
		gravity: a survey.
		\newblock {\em arXiv preprint:1910.04713}, 2019.
		
		\bibitem[GKRV20]{gkrv-bootstrap}
		Colin Guillarmou, Antti Kupiainen, R{\'e}mi Rhodes, and Vincent Vargas.
		\newblock Conformal bootstrap in {L}iouville theory.
		\newblock {\em arXiv preprint:2005.11530}, 2020.
		
		\bibitem[GM19]{GM19browniandisk}
		Ewain Gwynne and Jason Miller.
		\newblock Convergence of the free {B}oltzmann quadrangulation with simple
		boundary to the {B}rownian disk.
		\newblock {\em Ann. Inst. Henri Poincar\'{e} Probab. Stat.}, 55(1):551--589,
		2019.
		
		\bibitem[GRV19]{gkrv-genus}
		Colin Guillarmou, R\'{e}mi Rhodes, and Vincent Vargas.
		\newblock Polyakov's formulation of {$2d$} bosonic string theory.
		\newblock {\em Publ. Math. Inst. Hautes \'{E}tudes Sci.}, 130:111--185, 2019.
		
		\bibitem[GTF06]{garban-tf}
		Christophe Garban and Jos\'{e}~A. Trujillo~Ferreras.
		\newblock The expected area of the filled planar {B}rownian loop is {$\pi/5$}.
		\newblock {\em Comm. Math. Phys.}, 264(3):797--810, 2006.
		
		\bibitem[HL22]{hl-cle-on-lqg}
		Nina Holden and Matthis Lehmkuehler.
		\newblock Liouville quantum gravity weighted by conformal loop ensemble nesting
		statistics.
		\newblock {\em arXiv preprint:2204.09905}, 2022.
		
		\bibitem[HRV18]{HRV-disk}
		Yichao Huang, R\'{e}mi Rhodes, and Vincent Vargas.
		\newblock Liouville quantum gravity on the unit disk.
		\newblock {\em Ann. Inst. Henri Poincar\'{e} Probab. Stat.}, 54(3):1694--1730,
		2018.
		
		\bibitem[HS11]{hongler-smirnov-number-of-clusters}
		Cl\'{e}ment Hongler and Stanislav Smirnov.
		\newblock Critical percolation: the expected number of clusters in a rectangle.
		\newblock {\em Probab. Theory Related Fields}, 151(3-4):735--756, 2011.
		
		\bibitem[HS18]{hs-euclidean-mating}
		Nina Holden and Xin Sun.
		\newblock {SLE} as a mating of trees in {E}uclidean geometry.
		\newblock {\em Communications in Mathematical Physics}, 364(1):171--201, 2018.
		
		\bibitem[Iye85]{Ty85}
		Satish Iyengar.
		\newblock Hitting lines with two-dimensional brownian motion.
		\newblock {\em SIAM Journal on Applied Mathematics}, 1985.
		
		\bibitem[Kah85]{kahane}
		Jean-Pierre Kahane.
		\newblock Sur le chaos multiplicatif.
		\newblock {\em Ann. Sci. Math. Qu\'ebec}, 9(2):105--150, 1985.
		
		\bibitem[KMSW19]{KMSW19}
		Richard Kenyon, Jason Miller, Scott Sheffield, and David~B Wilson.
		\newblock Bipolar orientations on planar maps and {SLE}$_{12}$.
		\newblock {\em The Annals of Probability}, 47(3):1240--1269, 2019.
		
		\bibitem[KPZ88]{kpz-scaling}
		V.G. Knizhnik, A.M. Polyakov, and A.B. Zamolodchikov.
		\newblock {Fractal structure of 2D-quantum gravity}.
		\newblock {\em {Modern Phys. Lett A}}, 3(8):819--826, 1988.
		
		\bibitem[KRV20]{krv-dozz}
		Antti Kupiainen, R\'{e}mi Rhodes, and Vincent Vargas.
		\newblock Integrability of {L}iouville theory: proof of the {DOZZ} formula.
		\newblock {\em Ann. of Math. (2)}, 191(1):81--166, 2020.
		
		\bibitem[Law05a]{lawler-book}
		Gregory~F. Lawler.
		\newblock {\em Conformally invariant processes in the plane}, volume 114 of
		{\em Mathematical Surveys and Monographs}.
		\newblock American Mathematical Society, Providence, RI, 2005.
		
		\bibitem[Law05b]{Law08}
		Gregory~F. Lawler.
		\newblock {\em Conformally invariant processes in the plane}, volume 114 of
		{\em Mathematical Surveys and Monographs}.
		\newblock American Mathematical Society, Providence, RI, 2005.
		
		\bibitem[LV19]{lenells-viklund}
		Jonatan Lenells and Fredrik Viklund.
		\newblock Schramm's formula and the {G}reen's function for multiple {SLE}.
		\newblock {\em J. Stat. Phys.}, 176(4):873--931, 2019.
		
		\bibitem[LW04]{LW04}
		Gregory Lawler and Wendelin Werner.
		\newblock The brownian loop soup.
		\newblock {\em Probability theory and related fields}, 128(4):565--588, 2004.
		
		\bibitem[Maa20]{MR4079636}
		Micka\"{e}l Maazoun.
		\newblock On the {B}rownian separable permuton.
		\newblock {\em Combin. Probab. Comput.}, 29(2):241--266, 2020.
		
		\bibitem[Mal79]{MR555815}
		Colin~L. Mallows.
		\newblock Baxter permutations rise again.
		\newblock {\em J. Combin. Theory Ser. A}, 27(3):394--396, 1979.
		
		\bibitem[MS16]{MS16a}
		Jason Miller and Scott Sheffield.
		\newblock Imaginary {G}eometry {I}: Interacting {SLE}s.
		\newblock {\em Probability Theory and Related Fields}, 164(3-4):553--705, 2016.
		
		\bibitem[MS17]{MS17}
		Jason Miller and Scott Sheffield.
		\newblock {Imaginary geometry IV: interior rays, whole-plane reversibility, and
			space-filling trees}.
		\newblock {\em Probability Theory and Related Fields}, 169(3):729--869, 2017.
		
		\bibitem[MS20]{lqg-tbm-1}
		Jason Miller and Scott Sheffield.
		\newblock Liouville quantum gravity and the {B}rownian map {I}: the {${\rm
				QLE}(8/3,0)$} metric.
		\newblock {\em Invent. Math.}, 219(1):75--152, 2020.
		
		\bibitem[MS21]{lqg-tbm-2}
		Jason Miller and Scott Sheffield.
		\newblock Liouville quantum gravity and the {B}rownian map {II}: {G}eodesics
		and continuity of the embedding.
		\newblock {\em Ann. Probab.}, 49(6):2732--2829, 2021.
		
		\bibitem[Mur12]{Mur12}
		Jun Murakami.
		\newblock Volume formulas for a spherical tetrahedron.
		\newblock {\em Proceedings of the American Mathematical Society},
		140(9):3289--3295, 2012.
		
		\bibitem[Rem20]{remy-fb-formula}
		Guillaume Remy.
		\newblock The {F}yodorov-{B}ouchaud formula and {L}iouville conformal field
		theory.
		\newblock {\em Duke Math. J.}, 169(1):177--211, 2020.
		
		\bibitem[Rom06]{MR2266895}
		Dan Romik.
		\newblock Permutations with short monotone subsequences.
		\newblock {\em Adv. in Appl. Math.}, 37(4):501--510, 2006.
		
		\bibitem[RS05]{RS05}
		S.~Rohde and O.~Schramm.
		\newblock Basic properties of {SLE}.
		\newblock {\em Ann. of Math.}, 161(2), 2005.
		
		\bibitem[RV10]{RV10}
		Raoul Robert and Vincent Vargas.
		\newblock Gaussian multiplicative chaos revisited.
		\newblock {\em The Annals of Probability}, 38(2):605--631, 2010.
		
		\bibitem[RV11]{rhodes-vargas-log-kpz}
		R{\'e}mi Rhodes and Vincent Vargas.
		\newblock K{PZ} formula for log-infinitely divisible multifractal random
		measures.
		\newblock {\em ESAIM Probab. Stat.}, 15:358--371, 2011.
		
		\bibitem[Sch00]{Sch00}
		Oded Schramm.
		\newblock Scaling limits of loop-erased random walks and uniform spanning
		trees.
		\newblock {\em Israel Journal of Mathematics}, 118(1):221--288, 2000.
		
		\bibitem[Sch01]{schramm-left-passage}
		Oded Schramm.
		\newblock A percolation formula.
		\newblock {\em Electronic Communications in Probability}, 6:115--120, 2001.
		
		\bibitem[Sch11]{Sch06ICM}
		Oded Schramm.
		\newblock Conformally invariant scaling limits: an overview and a collection of
		problems.
		\newblock In {\em Selected works of {O}ded {S}chramm. {V}olume 1, 2}, Sel.
		Works Probab. Stat., pages 1161--1191. Springer, New York, 2011.
		
		\bibitem[She07]{She07}
		Scott Sheffield.
		\newblock Gaussian free fields for mathematicians.
		\newblock {\em Probab. Theory Related Fields}, 139, 2007.
		
		\bibitem[She16]{shef-zipper}
		Scott Sheffield.
		\newblock Conformal weldings of random surfaces: {SLE} and the quantum gravity
		zipper.
		\newblock {\em Ann. Probab.}, 44(5):3474--3545, 2016.
		
		\bibitem[Smi06]{SmirnovICM}
		Stanislav Smirnov.
		\newblock Towards conformal invariance of 2{D} lattice models.
		\newblock In {\em International {C}ongress of {M}athematicians. {V}ol. {II}},
		pages 1421--1451. Eur. Math. Soc., Z\"{u}rich, 2006.
		
		\bibitem[SSW09]{ssw-radii}
		Oded Schramm, Scott Sheffield, and David~B. Wilson.
		\newblock Conformal radii for conformal loop ensembles.
		\newblock {\em Comm. Math. Phys.}, 288(1):43--53, 2009.
		
		\bibitem[Sta09]{starr2009thermodynamic}
		Shannon Starr.
		\newblock Thermodynamic limit for the {M}allows model on {$S_n$}.
		\newblock {\em J. Math. Phys.}, 50(9):095208, 15, 2009.
		
		\bibitem[SW11]{sw-watts}
		Scott Sheffield and David~B. Wilson.
		\newblock Schramm's proof of {W}atts' formula.
		\newblock {\em Ann. Probab.}, 39(5):1844--1863, 2011.
		
		\bibitem[SZ10]{schramm-zhou}
		Oded Schramm and Wang Zhou.
		\newblock Boundary proximity of {SLE}.
		\newblock {\em Probab. Theory Related Fields}, 146(3-4):435--450, 2010.
		
		\bibitem[WP20]{pw-gff-notes}
		Wendelin {Werner} and Ellen {Powell}.
		\newblock {Lecture notes on the Gaussian Free Field}.
		\newblock {\em arXiv preprint: 2004.04720}, April 2020.
		
		\bibitem[ZZ96]{zz-dozz}
		A.~{Zamolodchikov} and A.~{Zamolodchikov}.
		\newblock {Conformal bootstrap in Liouville field theory}.
		\newblock {\em Nuclear Physics B}, 477:577--605, February 1996.
	\end{thebibliography}
	\newcommand{\etalchar}[1]{$^{#1}$}
	
\end{document}